\input ./style/arxiv-general.cfg
\documentclass[aop,MSNbibl,seceqn,dvips]{arximspdf}
\makeatletter
   \@ifpackageloaded{graphicx}{}{\usepackage{graphicx}}
\makeatother
\doi{10.1214/14-AOP916} %kopijuoti is PTS
\volume{43}
\issue{4}
\pubyear{2015}
\firstpage{1663}
\lastpage{1711}
\makeatletter
\renewcommand{\epsilon}{\varepsilon}
\newcommand{\binom}[2]{\pmatrix{#1\cr #2}}

\newtheorem{theorem}{Theorem}
\newproclaim{remark}{Remark}[section]
\newtheorem{corollary}{Corollary}[section]
\newtheorem{proposition}{Proposition}[section]
\newtheorem{lemma}{Lemma}[section]

\renewcommand{\a}{\alpha}
\renewcommand{\b}{\beta}
\newcommand{\ga}{\gamma}
\newcommand{\La}{\Lambda}
\newcommand{\R}{{\mathbb{R}}}
\newcommand{\N}{{\mathbb{N}}}
\newcommand{\D}{{\mathcal{D}}}
\newcommand{\E}{{\mathcal{E}}}
\newcommand{\G}{{\mathcal{G}}}
\newcommand{\F}{{\mathcal{F}}}
\makeatother

\begin{document}
\begin{frontmatter}

\title{Spectral gap for stochastic energy exchange model with
nonuniformly positive rate function}
\runtitle{SG for stochastic energy exchange model}

\begin{aug}
\author[A]{\fnms{Makiko}~\snm{Sasada}\corref{}\thanksref{T1}\ead[label=e1]{sasada@math.keio.ac.jp}}
%\and
%\author{\fnms{}~\snm{}}
\runauthor{M. Sasada}
\affiliation{Keio University}
%\dedicated{}
\address[A]{Department of Mathematics\\
Keio University\\
3-14-1, Hiyoshi, Kohoku-ku\\
Yokohama 223-8522\\
Japan\\
\printead{e1}} %adresu isvedimo komanda gale!
%\address{}
\end{aug}
\thankstext{T1}{Supported by JSPS Grant-in-Aid for Research Activity
Start-up Grant Number 23840036.}

% HISTORY:
\received{\smonth{7} \syear{2013}}
\revised{\smonth{1} \syear{2014}}
%\accepted{\smonth{} \syear{}}

% ABSTRACT
%
\begin{abstract}
We give a lower bound on the spectral gap for a class of stochastic
energy exchange models. In 2011, Grigo et al. introduced the model and
showed that, for a class of stochastic energy exchange models with a
uniformly positive rate function, the spectral gap of an $N$-component
system is bounded from below by a function of order $N^{-2}$. In this
paper, we consider the case where the rate function is not uniformly
positive. For this case, the spectral gap depends not only on $N$ but
also on the averaged energy $\mathcal{E}$, which is the conserved
quantity under the dynamics. Under some assumption, we obtain a lower
bound of the spectral gap which is of order $C(\mathcal{E})N^{-2}$
where $C(\mathcal{E})$ is a positive constant depending on $\mathcal
{E}$. As a corollary of the result, a lower bound of the spectral gap
for the mesoscopic energy exchange process of billiard lattice studied
by Gaspard and Gilbert
[\textit{J. Stat. Mech. Theory Exp.} \textbf{2008} (2008) p11021, \textit{J. Stat.
Mech. Theory Exp.} \textbf{2009} (2009) p08020]
%(2008, 2009)
and the stick process studied
by Feng et al.
[\textit{Stochastic Process. Appl.}
\textbf{66} (1997) 147--182]
%(1997)
are obtained.
\end{abstract}

% KEYWORDS

\end{frontmatter}

%s1 #&#
\section{Introduction}
\label{intro}

%s1.1 #&#
\subsection{Background and model}

Recently, Grigo et al. introduced a class of stochastic energy exchange
models, which are pure jump Markov processes with a continuous state
space in \cite{GKS}. The model is a generalization of the mesoscopic
energy exchange process of billiard lattice studied in \cite{GG08} and
\cite{GG09} by Gaspard and Gilbert. Showing the hydrodynamic limit for
such a mesoscopic model of mechanical origin is a very important step
for a rigorous derivation of a diffusion equation or Fourier law from a
system which is purely deterministic.

One of the key estimates required for the hydrodynamical limit is a
sharp lower bound on the spectral gap of the finite coordinate process
(cf. \cite{KL}). What is needed is that the gap, for the process
confined to cubes of size $N$, shrinks at a rate $N^{-2}$. Up to
constants, this is heuristically the best possible lower bound. For a
wide class of interacting particle systems or diffusion processes, the
desired spectral gap estimates have been obtained (cf. \cite{KL,LSV}).
On the other hand, for pure jump processes with a continuous state
space, this type of estimate has been scarcely shown. To our knowledge,
only for the Kac walk and its generalizations (cf. \cite
{Ca08,CCL,GF}), the sharp estimate of the spectral gap have been shown
before the result \cite{GKS} by Grigo et al. for stochastic energy
exchange models. Our goal is, as a first step for proving the
hydrodynamic limit, to extend the result in \cite{GKS} to the class
including the mesoscopic energy process of the billiard lattice.

The dynamics of the stochastic energy exchange model introduced by
Grigo et al. is described as follows: For each integer $N \ge2$,
denote by $\Sigma_N$ the one-dimensional cube $\{1,2,\ldots,N\}$. A
configuration of the state space $\R_+^{\Sigma_N} := [0,\infty
)^{\Sigma_N}$ is denoted by $x$, so that $x_i$ indicates the energy at
site $i \in\Sigma_N$, which is a positive real number. Fix a
nonnegative continuous function $\La\dvtx\R_+^2 \to\R_+$, which is
called a rate function, and a continuous function $P \dvtx \R_+^2 \to
\mathcal{P}([0,1])$ where $\mathcal{P}([0,1])$ is the set of
probability measures on $[0,1]$. At each nearest neighbor pair of the
lattice $(i,i+1)$, energy exchange independently happens with rate $\La
(x_i,x_{i+1})$. When the energy exchange happens between the pair
$(i,i+1)$, a number $0 \le\a\le1$ is drawn, independently of
everything else, according to a distribution $P(x_i,x_{i+1},d\a)$ and
the energy at site $i$ becomes $\a(x_i + x_{i+1})$, the energy at site
$i+1$ becomes $(1-\a)(x_i +x_{i+1})$, and all other energies remain unchanged.

%We consider a system of sticks that evolves through random
%stick-breaking events. The rate at which a particular stick breaks
%depends on its length.
%The broken- off piece is added on to another randomly chosen stick, so
%that total stick length is conserved at all times. To be precise, fix $
%\a> 1$
%and an integer $N$ which determines the size of the box. We consider
%the sticks on the finite $d$-dimensional lattice $\Sigma_N:=\{
%i=(i_1,i_2,\cdots,i_d); 0 \le i_k < N\} $
%without periodic boundary conditions.
%The state space of the process is $\chi_N:=\R_+^{\Sigma_N}=(0,\infty)^{
%\Sigma_N}$ An element of $\chi_N$ is a configuration $x = (x_i; i \in
%\Sigma_N)$ denotes the length of the stick at site $i$.

More precisely, we consider a continuous time Markov jump process
$x(t)$ on $\R^N_+$ by its infinitesimal generator $\mathcal{L}$,
acting on measurable bounded functions $f\dvtx \R^N_+ \to\R$ as
%
%e1.1 #&#
%
\begin{equation}
\label{eq:generator} \quad\mathcal{L} f(x)= \sum_{i=1}^{N-1}
\La(x_i,x_{i+1}) \int_0^1
P(x_i,x_{i+1},d\a)\bigl[f(T_{i,i+1,\a}x)-f(x)\bigr],
\end{equation}
where
\[
(T_{i,i+1,\a}x)_k= %
\cases{ x_k &\quad$
\mbox{if } k \neq i,i+1$,
\cr
\a(x_i+x_{i+1}) & \quad$
\mbox{if } k = i$,
\cr
(1-\a) (x_i+x_{i+1}) & \quad$\mbox{if
} k = i+1$. } %
\]

Obviously, the process preserves the total energy $\sum_{i=1}^N x_i$.
Therefore, for each $\E>0$, the set of configurations with mean energy
$\E$ per site
\[
\mathcal{S}_{\E,N}=\Biggl\{ x\in\R^N_+; \frac{1}{N}
\sum^N_{i=1} x_i= \E\Biggr\}
\]
is invariant for the process. Since $\mathcal{S}_{\E,N}$ is compact
and invariant, the assumed continuity of $\La$ and $P$ guarantees the
existence of at least one stationary distribution $\pi_{\E,N}$ for
$x(t)$ on each $\mathcal{S}_{\E,N}$. As mentioned, the scaling of the
rate of convergence toward the stationary distribution in terms of the
lattice size $N$ is of crucial importance in studying the hydrodynamic
limit of this model, especially if the system is of nongradient type.

%We refer the detailed background of the billiard model to \cite{GKS}.

%In this paper, we consider particle systems with random energy
%exchange introduced recently by Grigo et al in \cite{GKS}. The model
%is interesting since they are originate
%from deterministic mechanical models. We refer the background of the
%model to \cite{GKS}.

Under certain conditions, Grigo et al. proved that the spectral gap of
the generator $\mathcal{L}$ on $\mathcal{S}_{\E,N}$ is of order
$N^{-2}$ uniformly in the mean energy $\E$ \cite{GKS}. Since their
proof used the weak convergence in Wasserstein distance, it applies to
very general rate functions $\La$ and transition kernels $P$. The
existence of a lower bound on the rate function $\La$ and the
reversibility of the process are keys of their assumptions. However, as
pointed out by themselves, since the mesoscopic energy exchange process
of billiard lattice does not satisfy the first assumption, it was
desirable to remove the assumption on the existence of a uniform lower
bound of the rate function. In this paper, we relax the assumption and
study the case where a rate function satisfies $\La(a,b) \ge C
(a+b)^m$ for some $C >0$ and $m \ge0$ intuitively. We give a precise
assumption later, which is satisfied by the mesoscopic energy exchange
process of billiard lattice.

To weaken the condition on the rate function $\La$, we need a stronger
condition on the reversible measure. Precisely, we assume that our
process is reversible with respect to a product Gamma-distribution.
This condition is satisfied for general mechanical models, and hence
mesoscopic energy exchange processes of mechanical origin, such as the
mesoscopic energy exchange process of billiard lattice. See also
Remark~\ref{rem:mechanicalform} below to understand why this condition is
natural from a physical point of view.

%s1.2 #&#
\subsection{Notation and main result}

For each $\ga>0$, let $\nu_{\ga}$ denote a Gamma distribution on $\R
_+$ with a scale parameter $1$ and a shape parameter $\ga$, that is,
\[
\nu^{\ga}(dx)=x^{\ga-1}\frac{e^{-x}}{\Gamma(\ga)}\,dx.
\]
Let $\nu^{\ga,N}$ denote the product measure of $\nu^{\ga}$ on $\R
_+^N$ and $\nu^{\ga}_{\E,N}:= \nu^{\ga,N}|_{\mathcal{S}_{\E,N}}$
denote the conditional probability measure of $\nu^{\ga,N}$ on
$\mathcal{S}_{\E,N}$. From now on, we fix an arbitrary $\gamma>0 $
and assume that $\nu_{\E,N}^{\ga}$ is a reversible measure for
$\mathcal{L}$. We also denote $\nu^{\ga}_{\E,N}$ by $\nu_{\E,N}$
when there is no confusion.
%and
%$ \nu_{\E,N}$ denote the conditional probability measure of $
%\nu^N_{d,1}$ (or any $\nu^N_{d,e}$) on $\mathcal{S}_{\\E,N}$.
%Throughout this paper, we assume that $\pi_{\E,N}$ is a reversible
%measure for the process given by (\ref{eq:generator}); that is, each $
%\pi_{\E,N}$ is invariant under the dynamics and detailed balance holds.

Denote by $L^2(\nu_{\E,N})$ the Hilbert space of functions $f$ on
$\mathcal{S}_{\E,N}$ such that $ E_{\nu_{\E,N}}[f^2] < \infty$.
%Suppose that $\nu_{\E,N}$ is a reversible measure for $\mathcal{L}$.
Then the associated Dirichlet form is given by
\begin{eqnarray*}
\D(f) & =& \D_{\E,N}(f) := \int_{\mathcal{S}_{\E,N}}
\nu_{\E
,N}(dx) [- \mathcal{L}f](x) f(x)
\\
& =& \frac{1}{2} \sum_{i=1}^{N-1} \int
_{\mathcal{S}_{\E,N}} \nu_{\E,N}(dx) \La(x_i,x_{i+1})
\\
&&\hspace*{25pt}{}\times\int_0^1 P(x_i,x_{i+1},d
\a) \bigl[f(T_{i,i+1,\a}x)-f(x)\bigr]^2
\end{eqnarray*}
for all $f \in L^2( \nu_{\E,N})$.

%Since total energy is conserved, the movement is constrained on the
%\emph{microcanonical} surface of constant energy
%\begin{equation}
% \label{eq:mcsurf}
% \chi_{N,E}:=\{x; \sum_{i \in\Sigma_N}x_i=E\}.
%\end{equation}
%Consequently the \emph{microcanonical} measures
%$$
%\nu_{N,E}(\cdot) = \nu^N_{e}(\cdot|\chi_{N,E})
%$$
%are ergodic for our dynamics. Note that $\nu_{N,E}(\cdot)$ is a
%uniform measure on $\chi_{N,E}$. We define the Dirichlet form
%associated to microcanonical measures as
%\begin{eqnarray*}
%\D_{N, E}(f) & := \int_{\chi_{N,E}} - L_N(f) f \nu_{N,E}(dx) \\
%& = \frac{1}{2} \sum_{k \in\Sigma_N} \sum_{l \in\Sigma_N, |l-k|=1}
%\int_{\chi_{N,E}} \int^{x_k}_0 u^{\a-2}[f(x^{u,k,l})-f(x)]^2 du
%\nu_{N,E}(dx).
%\end{eqnarray*}

%In view of this symmetry, the second smallest eigenvalue of $-
%\mathcal{L}$ on $\mathcal{S}_{\E,N}$ is given by
We define
%
%e1.2 #&#
%
\begin{equation}
\lambda(\E,N):=\inf\biggl\{ \frac{ \D_{\E,N}(f)}{E_{ \nu_{\E
,N}}[f^2]} \Big| E_{ \nu_{\E,N}}[f]=0, f \in
L^2( \nu_{\E,N}) \biggr\}
\end{equation}
and call $\lambda(\E,N)$ the spectral gap of $- \mathcal{L}$ on
$\mathcal{S}_{\E,N}$ in $L^2 (\nu_{\E,N})$.
%\subsection{Assumptions}
%To state our assumptions, we introduce some notations.

%Throughout this paper, we assume the following assumption.
%\begin{Ass}\label{assm:general}
%(i) $ \nu_{\E,N}$ is a reversible measure for $X(t)$. \\
%(ii) There exists a unique $\ga>0$ such that $ \nu_{\E,N}$ is the
%conditional probability measure of $\nu^N_{\ga}$ on $\mathcal{S}_{
%\E,N}$ for all $\E>0$ and $N \in\N$. \\
%(iii) For each $\E>0$, $\pi_{\E, 2}$ is a unique reversible measure
%for $X(t)$ on $\mathcal{S}_{\E, 2}$.
% (Or equivalently $\lambda(2, e)>0$ for all $\E>0$).
%\end{Ass}

%th1 #&#
%
\begin{theorem}\label{thm:main}
Assume that there exist a positive constant $\tilde{C} >0$ and a real
number $m \ge0$ such that the following holds:
%
%e1.3 #&#
%
\begin{equation}
\label{assmeq:2particles} \lambda(\E, 2) \ge\tilde{C} \E^m \qquad
\mbox{for all }
\E> 0.
\end{equation}
Then there exists a positive constant $C >0$ depending only on $m$ and
$\ga$ such that
%
%e1.4 #&#
%
\begin{equation}
\label{eq:mainresult} \lambda(\E,N) \ge C\frac{\tilde{C} \E^m}{N^2}
\end{equation}
for all $\E>0$ and $N \ge2$.
%If $m \in(0,1)$, then there exists a positive constant $C$ depending
%only on $m$ and $\ga$ such that
%\begin{equation}
%\lambda(\E,N) \ge C\tilde{\Lambda} \E^m\frac{1}{N^{m+2}}.
%\end{equation}
\end{theorem}

%\begin{remark}
%We expect that even for $m \in(0,1)$, (\ref{eq:mainresult}) holds,
%but we do not have any rigorous proof so far.
%\end{remark}

%re1.1 #&#
%
\begin{remark}
For simplicity, we state the result in one-dimensional setting.
However, since our proof relies on the spectral gap estimate for the
long-range model and the kind of ``moving particle lemma'' in the
continuous state space, the result is extended to the case in any
dimension immediately, with a positive constant $C>0$ depending on $m$,
$\ga$ and $d$ the dimension of the lattice. We refer Section~2 of
\cite{M} for the detail of this argument. This is one of the advantage
of our proof compared to the preceding study.
\end{remark}

%Under Assumption~\ref{assm:general},

Grigo et al. call $(\La, P)$ is of mechanical form if the rate
function $\Lambda$ and the transition kernel $P$ are of the form
%
%e1.5 #&#
%
\begin{equation}
\label{eq:productform} % \begin{array}{cc}
\quad\Lambda(a,b) =\Lambda_s(a+ b)
\Lambda_r \biggl(\frac{a}{a+b} \biggr),\qquad P(a,b,d\a) =P
\biggl(\frac{a}{a+b},d\a\biggr) % \end{array}
\end{equation}
Carlen
and studied the processes of this form in detail in \cite{GKS},
Section~4. The form naturally occurs in models originating from mechanical
systems. Actually, the rate function and the probability kernel of the
mesoscopic energy exchange models of billiard lattice satisfies (\ref
{eq:productform}) and $\Lambda_s(s)=\sqrt{s}$ while $\Lambda_r$ is a
uniformly positive continuous function on $[0,1]$. See the explicit
expression in Section~\ref{sec:example}.

%re1.2 #&#
%
\begin{remark}\label{rem:mechanicalform}
One of the splendid results of \cite{GKS} is that if a stochastic
energy exchange model of mechanical form admits a reversible product
distribution, then this measure must necessarily be a product
Gamma-distributions (or a single atom). This is the reason why we
concentrate to study the process reversible with respect to a product
measure whose marginal is a Gamma distribution.
\end{remark}

If the process is of mechanical form, then by the definition, $\lambda
(\E,2)=\La_s(2\E)\tilde{C}$ holds where
\[
\tilde{C} = \inf\biggl\{ \frac{ \int_0^1 \mu(d\b) \La_r (\b)
\int_0^1 P(\b,d\a) [f(\a)-f(\b)]^2}{2 E_{\mu} [f^2]} \Big|
E_{\mu}[f]=0, f \in
L^2(\mu) \biggr\}
\]
and $\mu= \mu^{\ga}$ is the beta distribution on $[0,1]$ with
parameters $(\ga, \ga)$. Therefore, if the above $\tilde{C}$ is
strictly positive and $\La_s (s) \ge s^m$ for some $m \ge0$, then
(\ref{assmeq:2particles}) is satisfied.
Moreover, if ($\Lambda$, $P$) is of mechanical form and $\La
_s(s)=s^m$ for some $m \ge0$, then $\lambda(\E,N) = \E^m \lambda
(1,N)$ holds for all $\E>0$ and $N \in\N$ (cf. Lemma~\ref
{lem:scaling}). Namely, we cannot expect an order $N^{-2}$ bound of the
spectral gap to hold uniformly in $\E$. Then it is natural to ask
whether such a bound holds if we incorporate the extra factor $\E^m$,
and Theorem~\ref{thm:main} shows that this is indeed the case.

%co1.1 #&#
%
\begin{corollary}\label{cor:mechanicalcase}
Assume that $(\Lambda,P)$ is of mechanical form and $\La_s(s)=s^m$
for some $m \ge0$. Then, if
\begin{eqnarray*}
&&\inf\biggl\{ \frac{ \int_0^1 \mu^{\ga}(d\b) \La_r (\b) \int_0^1
P(\b,d\a) [f(\a)-f(\b)]^2}{ E_{\mu^{\ga} } [f^2]} \Big|
\\
&&\hspace*{114pt}E_{\mu
^{\ga}}[f]=0, f \in L^2
\bigl(\mu^{\ga}\bigr) \biggr\} >0
\end{eqnarray*}
holds, there exists a positive constant $C$ independent of $\E$ and
$N$ such that
\[
\lambda(\E,N) \ge C \frac{\E^m}{N^2}
\]
for all $\E>0$ and $N \ge2$.
\end{corollary}

%re1.3 #&#
%
\begin{remark}
To give an upper bound of the spectral gap $\lambda(\E,N)$, consider
the function $f_{\E,N}(x)=\sum_{i=1}^N i (x_i-\E)$. Then we have
\begin{eqnarray*}
&&\hspace*{-5pt}\D_{\E,N} (f_{\E,N})
\\
&&\hspace*{-5pt}\qquad= \frac{1}{2} \sum_{i=1}^{N-1}
\int_{\mathcal{S}_{\E,N}} \nu_{\E
,N}(dx) \La(x_i,x_{i+1})
\int_0^1 P(x_i,x_{i+1},d
\a) \bigl[(1-\a)x_i-\a x_{i+1}\bigr]^2
\\
&&\hspace*{-5pt}\qquad=\frac{N-1}{2} \int_{\mathcal{S}_{\E,N}}
\nu_{\E,N}(dx)
\La(x_1,x_2) \int_0^1
P(x_1,x_2,d\a) \bigl[(1-\a)x_1-\a
x_{2}\bigr]^2
\\
&&\hspace*{-5pt}\qquad\le(N-1) E_{ \nu_{\E,N}} \bigl[ \La(x_1,x_2)
\bigl(x_1^2+x_{2}^2\bigr)\bigr].
\end{eqnarray*}
On the other hand, $E_{ \nu_{\E,N}}[f_{\E,N}^2]=\sum_{i,j=1}^N ij
E_{ \nu_{\E,N}}[(x_i-\E)(x_j-\E)]$ and since $E_{ \nu_{\E
,N}}[(x_i-\E)(x_j-\E)] = -\frac{1}{N-1} E_{\nu_{\E,N}}[(x_i-\E
)^2]$ for $i \neq j$,
\[
E_{ \nu_{\E,N}}\bigl[f_{\E,N}^2\bigr] = \frac{N^2(N+1)}{12}
E_{ \nu_{\E
,N}}\bigl[(x_1-\E)^2\bigr].
\]
Then, by the equivalence of
ensembles, $E_{ \nu_{\E,N}} [ \La
(x_1,x_2) (x_1^2+x_{2}^2)]$ and\break $E_{ \nu_{\E,N}}[(x_1-\E)^2]$
converge to $E_{\nu_{\E}^2}[ \La(x_1,x_2) (x_1^2+x_{2}^2)]$ and $
E_{ \nu_{\E}^2}[(x_1-\E)^2]$ respectively as $N \to\infty$ where
$\nu_{\E}$ is the Gamma distribution with a scale parameter $\E$ and
a shape parameter $\ga$. Therefore, if $E_{\nu_{\E}^2}[ \La
(x_1,x_2) (x_1^2+x_{2}^2)] < \infty$, then there exist positive
constants $A(\E)$ and $B(\E)$ such that $\D_{\E,N} (f_{\E,N}) \le
N A(\E)$ and $E_{ \nu_{\E,N}}[f_{\E,N}^2] \ge N^3 B(\E)$, namely
$\lambda(\E,N) \le\frac{A(\E)}{N^2 B(\E)}$ for all $N \ge2$.
Therefore, our lower bound (\ref{eq:mainresult}) is optimal up to
constant depending on $\E$.

In particular, if $\La(x_1,x_2) \le C(x_1+x_2)^m$ for some $C >0$ and
$m \ge0$, then $\lambda(\E,N) \le\frac{C'\E^m}{N^2}$ with some
positive constant $C'$. Therefore, for such rate functions, our lower
bound (\ref{eq:mainresult}) is optimal up to constant which is
independent of $\E$ and $N$.
\end{remark}

The lack of a uniform lower bound complicates the rigorous analysis of
the rate of convergence to equilibrium. A similar problem was found in
the zero-range process with constant rate, and it had been an open
problem for decades. In 2005, Morris \cite{M} showed that the
spectral gap of that model is of order $(1+\rho)^{-2}N^{-2}$ where
$\rho$ is the density of particles and $N$ is the size of the system.
In the context of exclusion processes, it has been known that if the
jump rates are degenerate, the spectral gap does not have uniform lower
bound of order $N^{-2}$, and instead has a lower bound of order $C(\rho
) N^{-2}$ where $C(\rho)$ is a positive constant depending on $\rho$,
the density of particles (cf. \cite{GLT,NS}). Recently, the spectral
gap for the Kac model with hard sphere collisions is studied by Carlen
et al. in \cite{CCL2}. By projecting their model to the energy
coordinates, we obtain the process called $\mathcal{L}_{\mathrm
{LR}}^*$ in our
paper with parameters $\gamma= \frac{1}{2}$ and $0 \le m \le1$
($\gamma$ in \cite{CCL2} corresponds to $m$ in our paper). $\mathcal
{L}_{\mathrm{LR}}^*$ is a long-range (or equivalently, mean field)
version of
our process with specific $(\La,P)$, which we will denote by $(\La
^*,P^*)$. However, since we study the spectral gap in the whole
$L^2$-space, we cannot apply directly their result, which is for the
spectral gap in the symmetric sector. A classical spectral gap for the
gradient operator of the product Gamma-distribution was studied by
Barthe and Wolff in \cite{BW}, and they showed that it is of order $\E
^{-2}$ for $\ga\ge1$.

We also remark that the hydrodynamic limit for a special class of our
processes, which are gradient, was studied by Feng et al. in \cite{FIiiiS}. The process is called the stick process and of mechanical form
with $\La_s(s)=s^m$ where $m>0$ is a fixed model parameter. We show
that we can apply the main result of this paper to this class of models
in Section~\ref{sec:example}. As the hydrodynamic equation of the
stick process, the porous medium equation
\[
\partial_t \E(t,u) = \mathrm{const.} \,\partial_u\bigl(
\E(t,u)^{m} \partial_u \E(t,u)\bigr)
\]
was derived.
Gaspard and Gilbert conjectured that the hydrodynamic equation of the
mesoscopic energy exchange models of billiard lattice is also the
porous medium equation with $m=\frac{1}{2}$. By the scaling property
of the generator and the reversible measure, the same equation should
be derived from the stochastic energy exchange models of mechanical
form with $\La_s(s)=s^m$ (under the condition that the process is
reversible with respect to a product Gamma-distribution). The same
equation is derived also from an exclusion process with degenerate jump
rates in \cite{GLT}.

The rest of the article is organized as follows: In Section~\ref
{sec2}, we give
a detailed account of the proof of Theorem~\ref{thm:main}. Precisely,
we reduce the spectral gap estimate of the original process to that of
a long-range version of our process with specific $(\La,P)$, which we
will denote by $(\La^*,P^*)$. In Section~\ref{sec3}, we justify this
reduction,
and in Section~\ref{sec:long-range} we give an estimate for this
specific model. In
Section~\ref{sec:example}, we show that we can apply our result to the
mesoscopic energy exchange models of billiard lattice and the stick
process. In the \hyperref[app]{Appendix}, we give a sharp estimate of
the spectral gap
of the specific model with $N=3$. This sharp estimate is the key of our proof.

%s2 #&#
\section{Proof of Theorem \texorpdfstring{\protect\ref{thm:main}}{1}}\label{sec2}

Our basic idea of the proof is to introduce a few suitably chosen
reference processes and compare the Dirichlet forms associated with
them and that with the original process. First, we introduce a special
process given by a generator $\mathcal{L}^*$ with
%
%e2.1 #&#
%
\begin{eqnarray}
\label{eq:*LP} \Lambda^{*}(x_i,x_{i+1})&=&(x_i+x_{i+1})^m,
\nonumber
\\[-8pt]
\\[-8pt]
P^{*}(x_i,x_{i+1}, d\a)&=&
\frac{ \{ \a(1-\a)\}^{\ga-1}}{B(\ga,\ga)} \,d\a=\mu^{\ga}(d\a),
\nonumber
\end{eqnarray}
where $B(\ga,\ga)=\int_0^1\{ \a(1-\a)\}^{\ga-1} \,d\a$ is the
normalizing factor and $m \ge0$ and $\ga>0$ are the constants given
in the assumption.

We can rewrite the generator $\mathcal{L}^{*} (=\mathcal{L}^{*,m,\ga
}) $ given by (\ref{eq:*LP}) as
\[
\mathcal{L}^* f (x)= \sum_{i=1}^{N-1}
(x_i+x_{i+1})^m \bigl\{ E_{i,i+1}f (x)
-f(x)\bigr\}=\sum_{i=1}^{N-1}
(x_i+x_{i+1})^m D_{i,i+1}f (x),
\]
where $E_{i,j}f=E_{ \nu^{\ga}_{\E,N}}[f | \F_{i,j}]$,
$D_{i,j}f=E_{i,j}f-f$ and $\F_{i,j}$ is the $\sigma$-algebra
generated by variables $\{x_k \}_{k \neq i,j}$. Here, we follow the
notation used in \cite{Ca08} (see also \cite{S}). Note that $x_i
+x_j$ is measurable with respect to $\F_{i,j}$.
Using the above expression, we can easily check that $ \nu_{\E
,N}^{\ga}=\nu_{\E,N}$ is a reversible measure for the process. The
associated Dirichlet form is given by
\begin{eqnarray*}
\D^{*}(f)&=&\D^{*,m}_{\E,N}(f) := \int
\nu_{\E,N}(dx) \bigl[- \mathcal{L}^*f\bigr](x) f(x)
\\
&=& \sum_{i=1}^{N-1} E_{\nu_{\E,N}}
\bigl[(x_i+x_{i+1})^m ( E_{i,i+1}f -f )
^2 \bigr]
\\
&=& \sum_{i=1}^{N-1}
E_{\nu_{\E,N}} \bigl[(x_i+x_{i+1})^m (
D_{i,i+1}f ) ^2 \bigr]
\end{eqnarray*}
for all $f \in L^2( \nu_{\E,N})$. We use notation $\D^{*}$ or $\D
^{*,m}$ when there is no confusion.

We denote the spectral gap of $\mathcal{L}^{*}$ in $L^2( \nu^{\ga
}_{\E,N})$ by $\lambda^{*,m}(\E,N)$. Here, we abbreviate $\ga$.
Note that $\lambda^{*,m}(\E,2)=2^m\E^m$ since $E_{\nu_{\E,2}}
[(x_1+x_2)^m ( D_{1,2}f ) ^2 ]=(2\E)^m E_{\nu_{\E,2}} [ ( D_{1,2}f )
^2 ]=(2\E)^m \operatorname{Var}[f^2]$. Namely, this special model satisfies
assumption (\ref{assmeq:2particles}) of Theorem~\ref{thm:main}.
Moreover, the model is of mechanical form with $\La_s(s)=s^m$, $\La
_r(\b)=1$ and $P(\b,d\a)=\mu^{\ga}(d\a)$.

%re2.1 #&#
%
\begin{remark}
The model defined by the generator $\mathcal{L}^{*,0,1}$ (namely, the
above process with parameters $m=0$ and $\ga=1$) was studied by Kipnis
et al. in \cite{KMP} as an exactly solvable model which describes the
heat flow.
\end{remark}

We consider this special model because of the following guess: Under
assumption (\ref{assmeq:2particles}), the state of each pair of sites
achieves the equilibrium (with respect to the state of this pair of
sites) at least with the rate proportional to the $m$th power of the
sum of their energies under the dynamics given by $\mathcal{L}$.
Namely, the spectral gap of $\mathcal{L}$ can be bounded from below up
to constant by that of the process where the state of any pair of sites
achieves the equilibrium exactly with the rate proportional to the
$m$th power of the sum of their energies.
The next proposition shows that the guess truly holds.

%pr2.1 #&#
%
\begin{proposition}\label{prop:general}
Under the assumption of Theorem~\ref{thm:main}, for any $\E> 0$ and
$N \ge2$,
%
%e2.2 #&#
%
\begin{equation}
\lambda(\E,N) \ge\frac{\tilde{C}}{2^m} \lambda^{*,m}(\E,N)
\end{equation}
holds.
\end{proposition}

\begin{pf}
Define an operator $\mathcal{L}_0$ on $L^2(\nu_{\E,2})$ acting on
$f$ as
\[
\mathcal{L}_0 f (z_1,z_2)=
\La(z_1,z_2) \int P(z_1,z_2,d\a)
\bigl[f(T_{1,2,\a}z)-f(z)\bigr],
\]
where $z =(z_1,z_2) \in\R_+^2$. For $N \ge3$, $x \in\R_+^N$, $1
\le i <j \le N$ and $f\dvtx \R_+^N \to\R$, define $f^{i,j}_{x}
\dvtx \R_+^2
\to\R$ as
\[
f^{i,j}_{x}(p,q)=f(x_1,x_2, \ldots,
x_{i-1}, p, x_{i+1}, \ldots, x_{j-1}, q,
x_{j+1}, \ldots, x_N).
\]
Note that the function $f^{i,j}_x$ does not depend on $x_i$ nor $x_j$.
Then we can rewrite our generator as follows:
\[
\mathcal{L}f(x)=\sum_{i=1}^{N-1}
\mathcal{L}_{i,i+1} f(x),
\]
where $\mathcal{L}_{i,i+1} f (x) = (\mathcal{L}_0 f^{i,i+1}_x )
(x_i,x_{i+1})$. Then we have
\begin{eqnarray*}
%\label{eq:2-Nrelation}
&&E_{ \nu_{\E,N}}\bigl[f (-\mathcal{L}_{i,i+1}) f|
\F_{i,i+1}\bigr]
\\
&&\qquad= E_{ \nu
_{\E,N}}\bigl[f^{i,i+1}_x
(x_i, x_{i+1}) \bigl( \bigl(-\mathcal{L}_0
f^{i,i+1}_x\bigr) (x_i, x_{i+1})
\bigr)| \F_{i,i+1}\bigr]
\\
&&\qquad= E_{\nu_{{(x_i+x_{i+1})}/{2},2}}\bigl[f^{i,i+1}_x (-
\mathcal{L}_0 )f^{i,i+1}_x\bigr].
\end{eqnarray*}
Here, we use the \textit{noninterference} property of $ \nu_{\E,N}$,
which is mentioned in \cite{Ca08}. Namely, the conditional
distribution with respect to the pair $(x_i,x_{i+1})$ of $\nu_{\E,N}$
on the configuration space with a fixed value $x_i+x_{i+1}$ is $\nu
_{{(x_i+x_{i+1})}/{2},2}$ for any $\E>0$ and $N \ge2$.

Now, since $\mathcal{L}=\mathcal{L}_0$ for $N=2$, by the definition
of the spectral gap, we have for any $\E> 0$ and $g \in L^2(\nu_{\E,2})$,
\[
%C\E^m E_{\pi_{e,2}}[ \{ g - E_{\pi_{2, e}}[g] \} ^2 ] \le
\lambda(\E,2) E_{\nu_{\E, 2}}\bigl[ \bigl( g - E_{\nu_{\E, 2}}[g]
\bigr) ^2 \bigr] \le E_{\nu_{\E, 2}}\bigl[ g (-\mathcal{L}_0)
g \bigr].
\]
On the other hand, since $E_{i,i+1}$ is the integral operator with
respect to $\nu_{{(x_i+x_{i+1})}/{2},2}$ we have
\[
E_{\nu_{{(x_i+x_{i+1})}/{2},2}}\bigl[ \bigl( f^{i,i+1}_x -
E_{\nu
_{{(x_i+x_{i+1})}/{2},2}}\bigl[f^{i,i+1}_x \bigr] \bigr)
^2 \bigr]= E_{ \nu_{\E
,N}}\bigl[(D_{i,i+1}f)^2 |
\F_{i,i+1}\bigr].
\]
Combining the above equations, we have for any $\E>0$ and $N \ge2$
\[
\lambda\biggl(\frac{x_i+x_{i+1}}{2},2 \biggr) E_{ \nu_{\E
,N}}
\bigl[(D_{i,i+1}f)^2 |\F_{i,i+1}\bigr] \le
E_{ \nu_{\E,N}}\bigl[f (-\mathcal{L}_{i,i+1}) f| \F_{i,i+1}
\bigr].
\]
Then, by assumption (\ref{assmeq:2particles}),
%\begin{eqnarray*}
%& E_{ \nu_{\E,N}}[ (x_i+x_{i+1})^m E_{\nu_{x_i+x_{i+1},2}}[ \{
%f^{i,i+1}_x - E_{\nu_{x_i+x_{i+1},2}}[f^{i,i+1}_x] \} ^2 |
%\F_{i,i+1}]] \\
%& = E_{ \nu_{\E,N}}[ (x_i+x_{i+1})^m \{ f^{i,i+1}_x - E_{
%\nu_{x_i+x_{i+1},2}}[f^{i,i+1}_x] \} ^2 ]= E_{ \nu_{
%\E,N}}[(x_i+x_{i+1})^m \{ f- E_{i,i+1}f \} ^2 ],
%\end{eqnarray*}
we have
\begin{eqnarray*}
&& \frac{\tilde{C}}{2^m}E_{ \nu_{\E,N}}\bigl[(x_i+x_{i+1})^m
( D_{i,i+1}f ) ^2 \bigr]
\\
&&\qquad= E_{ \nu_{\E,N}}\biggl[ \tilde{C} \biggl( \frac{x_i+x_{i+1}}{2}
\biggr)^m E_{ \nu_{\E,N}}\bigl[(D_{i,i+1}f)^2 |
\F_{i,i+1} \bigr] \biggr]
\\
&&\qquad\le E_{ \nu_{\E,N}}\biggl[ \lambda\biggl(\frac
{x_i+x_{i+1}}{2},2 \biggr)
E_{ \nu_{\E,N}}\bigl[(D_{i,i+1}f)^2 |\F_{i,i+1}\bigr]
\biggr]
\\
&&\qquad\le E_{ \nu_{\E,N}}\bigl[ f (-\mathcal{L}_{i,i+1}) f \bigr].
\end{eqnarray*}
Finally, by summing up the terms for $1 \le i \le N-1$, we have $\frac
{\tilde{C}}{2^m} \D^{*}(f) \le\mathcal{D}(f)$ and complete the proof.
\end{pf}

%re2.2 #&#
%
\begin{remark}
The similar idea to the proof of Proposition~\ref{prop:general} was
already used in \cite{OS,S}.
\end{remark}

Hereafter, we only work on the process $\mathcal{L}^*$. The following
scaling relation is simple but the key of the rest of the paper.

%le2.1 #&#
%
\begin{lemma}\label{lem:scaling}
For any $\E>0$ and $N \ge2$,
%
%e2.3 #&#
%
\begin{equation}
\label{eq:scaling} \lambda^{*,m}(\E,N)=\E^m\lambda^{*,m}(1,N).
\end{equation}
\end{lemma}

\begin{pf}
Recall that for any $\E> 0$, $\nu_{1,N}$ is the image of $ \nu_{\E
,N}$ under the map $S \dvtx x \to\frac{1}{\E}x$, the unitary change of
scale from $\mathcal{S}_{\E,N}$ to $\mathcal{S}_{1, N}$. Therefore,
for any $f \in L^2(\nu_{\E,N})$, let $f_{\E}(x)=f(\E x)$, then
$f_{\E} \in L^2(\nu_{1,N})$ and
%
%e2.4 #&#
%
\begin{equation}
E_{\nu_{\E,N}} \bigl[f \bigl(-\mathcal{L}^*\bigr)f\bigr]= \E^m
E_{\nu_{1,N}} \bigl[f_{\E} \bigl(-\mathcal{L}^*
\bigr)f_{\E}\bigr]
\end{equation}
holds. Then the statement follows immediately from the definition of
the spectral gap.
\end{pf}

To estimate $\lambda^{*,m}(\E,N)$, we introduce a long-range version
of the process $\mathcal{L}^*$, which is defined by the generator
$\mathcal{L}^*_{\mathrm{LR}}\ (=\mathcal{L}^{*,m, \ga}_{\mathrm
{LR}} ) $ as
\[
\mathcal{L}_{\mathrm{LR}}^* f (x)= \frac{1}{N}\sum
_{i<j} (x_i+x_j)^m \bigl
\{ E_{i,j}f (x) -f(x)\bigr\}=\frac{1}{N}\sum
_{i<j} (x_i+x_j)^m
D_{i,j}f (x).
\]
It is easy to see that $ \nu_{\E,N}^{\ga}$ is a reversible measure
of $\mathcal{L}_{\mathrm{LR}}^*$ and the associated Dirichlet form is
given by
\begin{eqnarray*}
\D^{*}_{\mathrm{LR}} (f)&=&\D^{*,m}_{\mathrm{LR}, \E,N} (f) :=
\int\nu_{\E,N}(dx) \bigl[- \mathcal{L}_{\mathrm{LR}}^*f\bigr](x) f(x)
\\
& = &\frac{1}{N} \sum_{i<j} E_{\nu_{\E,N}}
\bigl[(x_i+x_j)^m (E_{i,j}f -f
)^2 \bigr]
\\
&=& \frac{1}{N} \sum_{i<j}
E_{\nu_{\E,N}} \bigl[(x_i+x_j)^m (
D_{i,j}f ) ^2 \bigr]
\end{eqnarray*}
for all $f \in L^2( \nu_{\E,N})$. We use notation $\D^{*}_{\mathrm
{LR}}$ or
$\D^{*,m}_{\mathrm{LR}}$ when there is no confusion and denote the
spectral gap
of $\mathcal{L}^{*}_{\mathrm{LR}}$ in $L^2( \nu_{\E,N})$ by
$\lambda
^{*,m}_{\mathrm{LR}}(\E,N)$.

Comparison techniques between the spectral gap of a nearest-neighbor
interaction process and that of its long-range version are known for
general interacting particle systems, or a class of continuous spin
systems with uniformly positive rate function (e.g., \cite{NS,S}).
However, to apply them to our process, we need to combine their ideas
cleverly because of the nonuniformly positive rate function. In fact,
unlike general comparison theorems, the spectral gap of $\mathcal
{L}^*$ is bounded from below by the spectral gap of $\mathcal
{L}^*_{\mathrm{LR}}$ multiplied by $N^{-2}$ and \textit{the spectral
gap of
$3$-site system}. Denote $\kappa_m=\lambda^{*,m}(\frac{1}{3},3)$ and
$\tilde{\kappa}_m=\lambda^{*,m}_{\mathrm{LR}}(\frac{1}{3},3)$.
Recall that
$\kappa_m$ and $\tilde{\kappa}_m$ depend also on $\gamma$. Our
comparison theorem is precisely given in the following way.

%th2 #&#
%
\begin{theorem}\label{thm:compare}
For any $m \ge0$, there exists a positive constant $C=C(m,\ga)$ such that
\[
\lambda^{*,m}(\E,N) \ge C\kappa_m N^{-2}
\lambda^{*,m}_{\mathrm{LR}}(\E,N)
\]
for all $\E>0$ and $N \ge2$.
\end{theorem}

We give a proof of this theorem in the next section.

Once we have Theorem~\ref{thm:compare} and the scaling relation (\ref
{eq:scaling}), then all we have to show is that $\kappa_m >0$ and
$\lambda^{*,m}_{\mathrm{LR}}(1,N)$ is uniformly positive in $N$.
$\kappa_m >0$
follows from Corollary~\ref{cor:kappa} below. The main work of this
paper is to give a uniform lower bound for the spectral gap of
$\mathcal{L}^*_{\mathrm{LR}}$ in the size of the system with a fixed mean
energy $\E$. It is already done for the case $m=0$ in \cite{GF} (also
in \cite{S} by a different proof) as follows.

%th3 #&#
%
\begin{theorem}[(\cite{GF,S})]\label{thm:0}
For any $\E>0$ and $N \ge2$,
%
%e2.5 #&#
%
\begin{equation}
\label{eq:gamnondeenerate} \lambda^{*,0}_{\mathrm{LR}}(\E,N)=\frac
{\ga N+1}{N(2\ga+1)}.
\end{equation}
In particular, $\inf_N \lambda^{*,0}_{\mathrm{LR}}(\E,N)=\frac{\ga
}{2\ga
+1} >0$ for any $\E> 0$.
\end{theorem}

To obtain a uniform bound for positive $m$, we prove a comparison
theorem between the spectral gap of the process with $m \ge1$ and with
$m=0$. To do this, we use the convexity of the function $x^m$ as a
function of $x \in\R_+$, which is true only for $m \ge1$.

%th4 #&#
%
\begin{theorem}\label{thm:convex}
For any $m\ge1$, $\E>0$ and $N \ge2$,
\[
\lambda^{*,m}_{\mathrm{LR}}(\E,N) \ge\frac{ \E^m \kappa_m}{2}
\lambda
^{*,0}_{\mathrm{LR}}(\E, N).
\]
\end{theorem}

In this estimate, the spectral gap of $3$-site system appears again. A
proof of the Theorem~\ref{thm:convex} is given in Section~\ref
{sec:long-range}.

The analysis of the case $0 < m <1$ is more difficult and complicated.
First, we give a new type of comparison theorem between $\lambda
^{*,m}_{\mathrm{LR}}(\E,N)$ and $\lambda^{*,2m}_{\mathrm{LR}}(\E,N)$:

%pr2.2 #&#
%
\begin{proposition}\label{prop:compm2m}
For any $m \ge0$, if $\tilde{\kappa}_m \ge\frac{1}{3}$, then
%
%e2.6 #&#
%
\begin{equation}
\label{eq:compm2m} \lambda^{*,m}_{\mathrm{LR}}(\E, N) \ge\sqrt{
\biggl( (3\tilde{\kappa}_m-1) \biggl(1-\frac{2}{N}\biggr)+
\frac{1}{N} \biggr) \lambda^{*,2m}_{\mathrm{LR}}(\E,N) }
\end{equation}
holds for all $\E>0$ and $N \ge2$.
\end{proposition}

We give a proof of this proposition in Section~\ref{sec:long-range}.
Obviously, to use inequality (\ref{eq:compm2m}), we need the following
key lemma and its corollary.

%le2.2 #&#
%
\begin{lemma}\label{lem:kappa1}
For any $\ga>0$,
\[
\tilde{\kappa}_1 >\tfrac{1}{3}.
\]
\end{lemma}

Here, we emphasize that $\tilde{\kappa}_m$ depends on $\ga$. A proof
of this lemma is given in the \hyperref[app]{Appendix}. The next corollary is shown
easily from this lemma.

%co2.1 #&#
%
\begin{corollary}\label{cor:kappa}
For any $m \ge0$ and $\ga>0$,
\[
3\tilde{\kappa}_m \ge\kappa_m >0.
\]
Moreover, for any $m \le1$,
\[
\tilde{\kappa}_m > \tfrac{1}{3}.
\]
\end{corollary}

\begin{pf}
Recall that $\kappa_m=\lambda^{*,m}(\frac{1}{3},3)$ and $\tilde
{\kappa}_m=\lambda^{*,m}_{\mathrm{LR}}(\frac{1}{3},3)$.
By the explicit expressions,
\begin{eqnarray*}
\D^{*}(f) & =& \bigl( E_{\nu} \bigl[(x_1+x_2)^m
( D_{1,2}f ) ^2 \bigr] + E_{\nu
}
\bigl[(x_2+x_3)^m ( D_{2,3}f )
^2 \bigr] \bigr),
\\
\D^{*}_{\mathrm{LR}}(f) & =&\tfrac{1}{3} \bigl(
E_{\nu} \bigl[(x_1+x_2)^m (
D_{1,2}f ) ^2 \bigr] + E_{\nu}
\bigl[(x_2+x_3)^m ( D_{2,3}f )
^2 \bigr]
\\
&&\hspace*{122pt}{} + E_{\nu} \bigl[(x_1+x_3)^m (
D_{1,3 }f ) ^2 \bigr] \bigr)
\end{eqnarray*}
for all $f \in L^2(\nu)$ where $\nu=\nu_{{1}/{3},3}^{\ga}$,
the inequality $3 \tilde{\kappa}_m \ge\kappa_m$ follows directly.
Note that $\nu$ does not depend on $m$.

Next, we show that $\kappa_1 > 0$. By the definition, for any $f \in
L^2(\nu)$ satisfying $E_{\nu}[f]=0$, we have
\begin{eqnarray*}
&& 3 \tilde{\kappa}_1 E_{\nu}\bigl[f^2\bigr]
\\
&&\qquad\le E_{\nu} \bigl[(x_1+x_2) (
D_{1,2}f ) ^2 \bigr] + E_{\nu}
\bigl[(x_2+x_3) ( D_{2,3}f ) ^2
\bigr]
\\
&&\qquad\quad{}+ E_{\nu} \bigl[(x_1+x_3) (
D_{1,3 }f ) ^2 \bigr].
\end{eqnarray*}
Noting $ E_{\nu}[ (x_i+x_j)(D_{i,j}f)^2 ]= E_{\nu}[(x_i+x_j) f^2 ] -
E_{\nu}[(x_i+x_j)(E_{i,j}f)^2 ]$, we have
\begin{eqnarray*}
&& E_{\nu}\bigl[ (x_1+x_2)
(E_{1,2}f)^2 \bigr] + E_{\nu}\bigl[(x_2+x_3)
(E_{2,3}f)^2 \bigr] + E_{\nu}
\bigl[(x_1+x_3) (E_{1,3}f)^2 \bigr]
\\
&&\qquad\le(2 -3\tilde{\kappa}_1 ) E_{\nu}
\bigl[f^2\bigr]
\end{eqnarray*}
since for any $x= (x_1,x_2,x_3) \in\mathcal{S}_{{1}/{3},3}$,
$x_1+x_2+x_3=1$. Therefore,
\begin{eqnarray*}
&&E_{\nu}\bigl[ (x_1+x_2) (E_{1,2}f)^2
\bigr] + E_{\nu}\bigl[(x_2+x_3)
(E_{2,3}f)^2 \bigr]
\\
&&\qquad\le\bigl(1 - (3\tilde{
\kappa}_1 -1) \bigr) E_{\nu}\bigl[f^2\bigr]
\\
&&\qquad\le E_{\nu}\bigl[(x_1+x_2)f^2
\bigr] + E_{\nu}\bigl[(x_2+x_3)f^2
\bigr] - (3\tilde{\kappa}_1 -1)E_{\nu}\bigl[f^2
\bigr]
\end{eqnarray*}
which implies $(3\tilde{\kappa}_1 -1)E_{\nu}[f^2] \le E_{\nu}
[(x_1+x_2) ( D_{1,2}f ) ^2 ] + E_{\nu} [(x_2+x_3) ( D_{2,3}f ) ^2 ]$
and hence $\kappa_1 \ge3\tilde{\kappa}_1 -1$. Then, by Lemma~\ref
{lem:kappa1}, $\kappa_1 >0$ follows.

Now, since $(x_i+x_j)^m$ is decreasing in $m$ for any fixed $x \in
\mathcal{S}_{{1}/{3},3}$ and $1 \le i < j \le3$, we have $\kappa
_m$ and $\tilde{\kappa}_m$ are both decreasing in $m$. Therefore, for
any $m \le1$, $\kappa_m >0$ and $\tilde{\kappa}_m > \frac{1}{3}$ holds.

On the other hand, for $m >1$, by H\"older's inequality,
\begin{eqnarray*}
\hspace*{-5pt}&&E_{\nu}\bigl[(x_1+x_2) (D_{1,2}f)^2
\bigr] + E_{\nu}\bigl[(x_2+x_3)
(D_{2,3}f)^2 \bigr]
\\
\hspace*{-5pt}&&\qquad\le E_{\nu}\bigl[(x_1+x_2)^m
(D_{1,2}f)^2 \bigr]^{{1}/{m}}E_{\nu
}
\bigl[(D_{1,2}f)^2 \bigr]^{{1}/{m'}}
\\
\hspace*{-5pt}&&\qquad\quad{} + E_{\nu}\bigl[(x_2+x_3)^m
(D_{2,3}f)^2 \bigr]^{{1}/{m}}E\bigl[(D_{2,3}f)^2
\bigr]^{{1}/{m'}}
\\
\hspace*{-5pt}&&\qquad\le\bigl\{ E_{\nu}\bigl[(x_1+x_2)^m
(D_{1,2}f)^2 \bigr] + E_{\nu}\bigl[(x_2+x_3)^m
(D_{2,3}f)^2 \bigr] \bigr\}^{{1}/{m}}
\\
\hspace*{-5pt}&&\qquad\quad{}\times\bigl\{ E_{\nu}\bigl[ (D_{1,2}f)^2
\bigr] + E_{\nu}\bigl[(D_{2,3}f)^2 \bigr] \bigr
\}^{{1}/{m'}}
\\
\hspace*{-5pt}&&\qquad\le\bigl\{ E_{\nu}\bigl[(x_1+x_2)^m
(D_{1,2}f)^2 \bigr] + E_{\nu}\bigl[(x_2+x_3)^m
(D_{2,3}f)^2 \bigr] \bigr\}^{{1}/{m}} \bigl\{ 2
E_{\nu}\bigl[f^2 \bigr] \bigr\}^{{1}/{m'}},
\end{eqnarray*}
where $\frac{1}{m}+\frac{1}{m'}=1$. For the second inequality, we use
the inequality $a^{1/m}b^{1/m'}+c^{1/m}d^{1/m'} \le
(a+c)^{1/m}(b+d)^{1/m'}$ for any nonnegative numbers $a,b,c$ and $d$,
which is obtained by H\"older's inequality for a two-point space
equipped with the counting measure.

Then, combining the above inequality with the fact that $\kappa
_1=\lambda^*(\frac{1}{3},3)$, we have
\[
\kappa_1 E_{\nu}\bigl[f^2\bigr] \le
2^{{1}/{m'}} \bigl( \D^{*,m}(f) \bigr)^{{1}/{m}} \bigl(
E_{\nu}\bigl[f^2 \bigr] \bigr)^{{1}/{m'}}
\]
which implies $\kappa_m \ge\frac{\kappa_1 ^ {m'}}{2} >0$.
\end{pf}

%re2.3 #&#
%
\begin{remark}
For the following proof of Theorem~\ref{thm:main}, the sharp estimate
$\tilde{\kappa}_1 > \frac{1}{3}$ is essential for the case $0 <
m<1$, but the weaker condition $\tilde{\kappa}_1 >0$ is sufficient
for the case $m \ge1$. To prove $\tilde{\kappa}_1 >0$, we can avoid
the complicated argument in the \hyperref[app]{Appendix} and instead
use a simpler
argument, for example, the one used in Section~4.2 of \cite{BCTV}.
\end{remark}

\begin{pf*}{Proof of Theorem~\ref{thm:main}}
Combining Theorem~\ref{thm:0}, Theorem~\ref{thm:convex} and
Corollary~\ref{cor:kappa}, for any $m \ge1$, there exists a positive constant
$C=C(m,\ga)$ such that
%
%e2.7 #&#
%
\begin{equation}
\label{ineq:m1} \lambda^{*,m}_{\mathrm{LR}}(\E,N) \ge C \E^m
\end{equation}
for all $\E>0$ and $N \ge2$. Then, applying Proposition~\ref
{prop:compm2m} and Corollary~\ref{cor:kappa}, for any $\frac{1}{2}
\le m <1$,
\[
\lambda^{*,m}_{\mathrm{LR}}(\E,N) \ge\E^m \sqrt{ \biggl(
(3\tilde{\kappa}_m-1) \biggl(1-\frac{2}{N}\biggr)+
\frac{1}{N} \biggr) C(2m, \ga)} \ge C(m,\ga) \E^m,
\]
where $C(2m,\ga)$ is the constant in the inequality (\ref{ineq:m1}) and
\[
C(m,\ga)=\sqrt{C(2m, \ga) } \inf_{N \ge2} \sqrt{ \biggl( (3
\tilde{\kappa}_m-1) \biggl(1-\frac{2}{N}\biggr)+
\frac{1}{N} \biggr) } >0.
\]
Repeating the same argument, we have for any $\frac{1}{2^{k+1}} \le m
<\frac{1}{2^k}$ with\vspace*{1pt} some $k \in\N$, there exists a positive
constant $C=C(m,\ga)$ such that
\[
\lambda^{*,m}_{\mathrm{LR}}(\E,N) \ge C \E^m
\]
for all $\E>0$ and $N \ge2$. Therefore, it holds for any $m >0$ and
also for $m=0$ by Theorem~\ref{thm:0}.

Now, combining this inequality with Proposition~\ref{prop:general} and
Theorem~\ref{thm:compare}, we complete the proof of Theorem~\ref{thm:main}.
\end{pf*}

%re2.4 #&#
%
\begin{remark}
We can also consider the process generated by $\mathcal{L}_{\mathrm
{LR}}^{*,m}$
with negative $m$. However, for this case, the statement ``$\lambda
^{*,m}_{\mathrm{LR}}(\E,N) \ge C\E^m$ for\vspace*{1pt} some positive constant
$C$'' or the equivalent statement ``$\lambda^{*,m}_{\mathrm{LR}}(1,N)
\ge C$ for some positive constant $C$'' turns out to be false.
Actually, fix $\E=1$ and consider the function $f_N(x)={\mathbf
{1}}_{\{x_1 > {N}/{2}\}} \in L^2(\nu_{1,N})$, then
\begin{eqnarray*}
E\bigl[(x_i+x_j)^m(D_{i,j}f_N)^2
\bigr] & =&0 \qquad\mbox{for } i,j \neq1,
\\
E\bigl[(x_1+x_j)^m(D_{1,j}f_N)^2
\bigr] & =&E\bigl[(x_1+x_2)^m(D_{1,2}f_N)^2
\bigr] \qquad\mbox{for } j \ge2
\end{eqnarray*}
and
\begin{eqnarray*}
&& E\bigl[ (x_1+x_2)^m (D_{1,2}f_N)^2
\bigr]
\\
&&\qquad=E\bigl[{\mathbf{1}}_{\{x_1 +x_2 > {N}/{2}\}
}(x_1+x_2)^m(D_{1,2}f_N)^2
\bigr]
\\
&&\qquad\quad{}+E\bigl[{\mathbf{1}}_{\{x_1 +x_2 \le{N}/{2}\}
}(x_1+x_2)^m(D_{1,2}f_N)^2
\bigr]
\\
&&\qquad=E\bigl[{\mathbf{1}}_{\{x_1 +x_2 > {N}/{2}\}}(x_1+x_2)^m
(D_{1,2}f_N)^2\bigr]
\\
&&\qquad\le\biggl(\frac{N}{2} \biggr)^m E
\bigl[(D_{1,2}f_N)^2\bigr] \le\biggl(
\frac
{N}{2} \biggr)^m \operatorname{Var}(f_N),
\end{eqnarray*}
where $\operatorname{Var}(f_N)$ is the variance of $f_N$. Namely,
$\lambda
^{*,m}_{\mathrm{LR}}(1,N) \le2^{-m} N ^{m}$ which means there is no uniform
spectral gap for the case where $m$ is negative.
\end{remark}

%A key assumption of our result is that the process to be reversible
%with respect to a product of stationary Gamma distribution. In
%\cite{GKS}, it is shown that the product of Gamma distribution
%naturally occurs in models originating from mechanical systems. In
%fact, all of examples studied in \cite{GKS} or in Section~4 of this
%paper satisfy this assumption. We note that the precise estimates of
%the spectral gap and logarithmic Sobolev constants of canonical Gibbs
%measures associated to Gamma distribution with parameter $\ga\ge1$
%were studied in \cite{BW}, but for $\ga< 1$ is still open.

%The basic idea to show the spectral gap estimate for $X(t)$ is to
%compare the spectral gap of its generator to a suitably chosen
%reference process introduced in the next section.

%s3 #&#
\section{Reduction to the long-range model}\label{sec3}
%\label{special}

%\begin{remark}
%As a consequence of Theorem~2.12 of \cite{GKS}, $\lambda^{*,0}(N) \ge
%\frac{\ga}{2\ga+1}\sin^2(\frac{\pi}{N+2})$ is shown.
% However, to show the spectral gap estimate for the case $m \ge1$, we
%need to study the long-range version of the model which will be
%defined precisely in the next subsection. As a by-product of the
%result, we also have the result that $\lambda^{*}(N) = O(N^{-2})$.
%\end{remark}
In this section, we estimate the spectral gap $\lambda^{*,m}(\E,N)$
from below by $\lambda^{*,m}_{\mathrm{LR}}(\E,N)$ by comparing the associated
Dirichlet forms.

We first give simple but useful lemmas.

%le3.1 #&#
%
\begin{lemma}
%Since for any $x \in\mathcal{S}_{\frac{1}{N}, N}$, $\sum_{i=1}^N
%x_i=1$, followings hold:
%
%(i) for each $N \in\N$, $\la^{*,m}(\frac{1}{N},{N})$ is decreasing in
%$m$.\\
For any $x \in\R_+^N$,
\[
\Biggl(\sum_{i=1}^N x_i
\Biggr)^m \lambda^{*,m}\biggl(\frac{1}{N},{N}\biggr)=
\lambda^{*,m}\biggl(\frac{\sum_{i=1}^N x_i}{N},N\biggr).
\]

In particular, applying the equality for $N=3$, we have
%
%e3.1 #&#
%
\begin{equation}
\label{eq:scaling2} \kappa_m (a+b+c)^m =\lambda^{*,m}
\biggl(\frac{a+b+c}{3},3\biggr)
\end{equation}
for any $a,b,c >0$.
\end{lemma}

\begin{pf}
It follows from Lemma~\ref{lem:scaling} directly.
\end{pf}

%le3.2 #&#
%
\begin{lemma}
For any $m \ge0$, $\kappa_m \le\frac{3}{2}$.
\end{lemma}

\begin{pf}
Since $\kappa_m$ is decreasing in $m$, $\kappa_m \le\kappa_0$ and
by Corollary~\ref{cor:kappa} and Theorem~\ref{thm:0}, $\kappa_0 \le
3\tilde{\kappa}_0 = \frac{3\ga+1}{2\ga+1} \le\frac{3}{2}$.
\end{pf}

%\subsection{Reduction to the long-range model}\

%\[
%\mathcal{L}_{LR}^* f (x)= \frac{1}{N}\sum_{i<j} (x_i+x_j)^m \{
%E_{i,j}f (x) -f(x)\}=\frac{1}{N}\sum_{i<j} (x_i+x_j)^m D_{i,j}f (x).
%\]
%It is easy to see that $ \nu_{\E,N}^{\ga}$ is a reversible measure of $
%\mathcal{L}_{LR}^*$ and the associated Dirichlet form is given by
%\begin{eqnarray*}
%\D^{*,LR}_{\E,N}(f) & := \int\nu_{\E,N}(dx) [-
%\mathcal{L}_{LR}^*f](x) f(x) \\
%& = \sum_{i<j} E_{\nu_{\E,N}} [(x_i+x_j)^m (E_{i,j}f -f )^2 ] =
%\sum_{i<j} E_{\nu_{\E,N}} [(x_i+x_j)^m ( D_{i,j}f ) ^2 ]
%\end{eqnarray*}
%for all $f \in L^2( \nu_{\E,N})$.

To compare $\D^{*}_{\mathrm{LR}}(f)$ and $\D^{*}(f)$, we introduce operators
$\pi_{i,j}\dvtx \mathcal{S}_{\E,N} \to\mathcal{S}_{\E,N}$ which
exchange the energies of sites $i$ and $j$:
\[
(\pi_{i,j} x )_k = %
\cases{ x_k &
\quad$\mbox{if }k \neq i,j$,
\cr
x_j & \quad$\mbox{if }k=i$,
\cr
x_i & \quad$\mbox{if }k=j$, } %
\]
and $\pi_{i,i}x=x$.

Before going to the main result in this section, we prepare the
following key lemma.

%le3.3 #&#
%
\begin{lemma}\label{lem:moving}
There exists a universal positive constant $C$ such that for any $m \ge
0$, $\E> 0$, $N \ge2$ and $1 \le i<j \le N$,
%
%e3.2 #&#
%
\begin{eqnarray}
\label{eq:moving!!} &&\kappa_m E_{ \nu_{\E,N}}\bigl[x_i^m
( f \circ\pi_{i,j} -f )^2 \bigr]
\nonumber
\\[-8pt]
\\[-8pt]
&&\qquad\le C|j-i| \sum
_{k=i}^{j-1}E_{ \nu_{\E,N}}\bigl[
(x_k+x_{k+1})^m (D_{k,k+1}f)^2
\bigr]\nonumber
\end{eqnarray}
holds for all $f \in L^2( \nu_{\E,N})$.
\end{lemma}

\begin{pf}
Following the strategy used in the study of the spectral gap for
multi-species exclusion processes \cite{NS}, we express the exchange
$\pi_{i,j}$ with rate $x_i^m$ by a sequence of neighboring exchange
$\pi_{k,k+1}$ with rate $x_k^m$ or $x_{k+1}^m$, and $\pi_{k,k+2}$
with rate $x_k^m$ or $x_{k+2}^m$. More precisely, for $i<j$, we denote
$K=j-i$ and define a sequence of sites $n_0=i,n_1,n_2, \ldots,
n_{4K-3}$ by
\[
n_k = %
\cases{ i+k & \quad$\mbox{if }0 \le k \le K $,
\cr
j-2-l & \quad$\mbox{if }k=K+2l+1, 0 \le l \le K-2$,
\cr
j-l & \quad$\mbox{if
}k=K+2l, 1 \le l \le K-1$,
\cr
i+k-3K+3 & \quad$\mbox{if }3K-1 \le k \le4K-3$. }
\]
We define operators $S_k\dvtx \mathcal{S}_{\E,N} \to\mathcal
{S}_{\E
,N}$ for $ 0 \le k \le4K-3$ by $S_0=\mathit{Id}$ and $S_{k+1} = \pi
_{n_k,n_{k+1}} \circ S_k $ for $ 0 \le k \le4K-4$. By the
construction, $S_{4K-3}=\pi_{i,j}$ and for all $ 0 \le k \le4K-3$,
$(S_kx)_{n_k}=x_i$ and $|n_k -n_{k+1}|=1$ or $2$. We emphasize that the
use of transitions $\pi_{k,k+2}$ is necessary since due to that, we
have the crucial relation $(S_kx)_{n_k}=x_i$. Also, it is precisely
because of these extra transitions that we are forced to have the
factor $\kappa_m$ in (\ref{eq:moving!!}).

Then, by the Schwarz inequality,
\begin{eqnarray*}
&&E_{ \nu_{\E,N}} \bigl[x_i^m \bigl\{f(
\pi_{i,j}x )-f( x )\bigr\}^2 \bigr]
\\
&&\qquad \le(4K-3) \sum
_{k=0}^{4K-4} E_{ \nu_{\E,N}}
\bigl[x_i^m \bigl\{f(S_{k+1}x )-f(
S_kx )\bigr\}^2\bigr]
\\
&&\qquad= (4K-3) \sum_{k=0}^{4K-4}
E_{ \nu_{\E,N}}\bigl[x_{n_k}^m \bigl\{f(\pi
_{n_k,n_{k+1}}x )-f( x )\bigr\}^2 \bigr]
\\
&&\qquad\le(4K-3) \Biggl\{ 3 \sum_{k=i}^{j-1}
E_{ \nu_{\E,N}}\bigl[x_k^m \bigl\{f(\pi
_{k,k+1}x )-f( x )\bigr\}^2 \bigr]
\\
&&\qquad\quad\hspace*{45pt} {} + \sum_{k=i}^{j-2} E_{ \nu_{\E,N}}
\bigl[x_{k+2}^m \bigl\{f(\pi_{k,k+2}x )-f( x )\bigr
\}^2 \bigr] \Biggr\}.
\end{eqnarray*}
Here, we use the fact that $i \le n_k \le j$ for all $ 0 \le k \le
4K-4$ and for each $l$ satisfying $i \le l <j$, $\sharp\{ k; \{
n_k,n_{k+1}\} =\{l,l+1\} \} \le3$ and $\sharp\{ k; \{n_k,n_{k+1}\} =\{
l,l+2\} \} \le1$ by the construction. We also use the invariance of
$\nu_{\E,N}$ under the permutation of coordinates, which we will use
repeatedly without notice.

Now, since $E_{k,k+1}f(\pi_{k,k+1}x)=E_{k,k+1}f(x)$, we obtain that
%
%\[
%
\begin{eqnarray*}
&&E_{ \nu_{\E,N}}\bigl[ x_k^m \bigl\{f(
\pi_{k,k+1}x )-f( x )\bigr\}^2 \bigr]
\\
&&\qquad= E_{ \nu_{\E,N}}\bigl[ x_k^m \bigl\{f(
\pi_{k,k+1}x )-(E_{k,k+1}f) (\pi_{k,k+1}x)+(E_{k,k+1}f)
(x)- f( x )\bigr\}^2 \bigr]
\\
&&\qquad\le2 E_{ \nu_{\E,N}}\bigl[ x_{k+1}^m\bigl\{f(x
)-(E_{k,k+1}f) (x) \bigr\}^2\bigr]
\\
&&\qquad\quad{}+ 2E_{ \nu_{\E,N}}\bigl[
x_k^m \bigl\{(E_{k,k+1}f) (x)- f( x )\bigr
\}^2 \bigr]
\\
&&\qquad\le4 E_{ \nu_{\E,N}}\bigl[ (x_k+x_{k+1})^m
(D_{k,k+1}f)^2\bigr].
\end{eqnarray*}
%
%\]
%
By the same argument,
\[
%\begin{array}{cc}
E_{ \nu_{\E,N}}\bigl[ x_{k+2}^m \bigl\{f(
\pi_{k,k+2}x )-f( x )\bigr\}^2 \bigr] \le4 E_{ \nu_{\E,N}}
\bigl[ (x_k+x_{k+2})^m (D_{k,k+2}f)^2
\bigr]. %\end{array}
\]
Then, by the definition of the spectral gap and the scaling relation
(\ref{eq:scaling2}), we have
\begin{eqnarray*}
&&\kappa_m E_{\nu_{\E,N}}\bigl[ (x_k+x_{k+1}+x_{k+2})^m
(D_{k,k+2}f)^2| \F_{k,k+1,k+2}\bigr]
\\
&&\qquad= \lambda^{*,m}\biggl(\frac{x_k+x_{k+1}+x_{k+2}}{3},3\biggr)
E_{\nu_{\E
,N}}\bigl[(D_{k,k+2}f)^2| \F_{k,k+1,k+2}
\bigr]
\\
&&\qquad= \lambda^{*,m}\biggl(\frac{x_k+x_{k+1}+x_{k+2}}{3},3\biggr)
E_{\nu_{\E
,N}}\bigl[(E_{k,k+2}f - f )^2|
\F_{k,k+1,k+2}\bigr]
\\
&&\qquad\le\lambda^{*,m}\biggl(\frac
{x_k+x_{k+1}+x_{k+2}}{3},3\biggr)
E_{\nu_{\E
,N}}\bigl[(E_{k,k+1,k+2}f - f )^2|
\F_{k,k+1,k+2}\bigr]
\\
&&\qquad\le E_{\nu_{\E,N}}\bigl[ (x_k+x_{k+1})^m
(D_{k,k+1}f)^2| \F_{k,k+1,k+2}\bigr]
\\
&&\qquad\quad{} + E_{ \nu_{\E,N}}\bigl[ (x_{k+1}+x_{k+2})^m
(D_{k+1,k+2}f)^2| \F_{k,k+1,k+2}\bigr],
\end{eqnarray*}
where $E_{i,j,k}f=E_{\nu}[f| \F_{i,j,k}]$ and $\F_{i,j,k}$ is the
$\sigma$-algebra generated by variables $\{x_l \}_{l \neq i,j,k}$.
At the first inequality, we use the relation that
\begin{eqnarray*}
&&E_{\nu_{\E,N}}\bigl[ (E_{i,j,k}f-f)^2 | \F_{i,j,k}
\bigr]
\\
&&\qquad= E_{\nu_{\E,N}}\bigl[ (E_{i,j,k}f-E_{i,j}f)^2
| \F_{i,j,k}\bigr] + E_{\nu_{\E,N}}\bigl[ (E_{i,j}f-f)^2
| \F_{i,j,k}\bigr].
\end{eqnarray*}
Then, by taking the expectation, we have
%
%e3.3 #&#
%
\begin{eqnarray}
\label{eq:2-estimate} %
%\begin{array}{ll}
&& \kappa_m E_{ \nu_{\E,N}}
\bigl[ (x_k+x_{k+1}+x_{k+2})^m
(D_{k,k+2}f)^2\bigr]
\nonumber
\\
%[-8pt]
%\\[-8pt]
&&\qquad\le E_{ \nu_{\E,N}}\bigl[ (x_k+x_{k+1})^m
(D_{k,k+1}f)^2\bigr]
\\
&&\qquad\quad{}+ E_{ \nu
_{\E,N}}\bigl[
(x_{k+1}+x_{k+2})^m (D_{k+1,k+2}f)^2
\bigr].
\nonumber
%\end{array}
%
\end{eqnarray}
Then, combing the inequalities, we have
%
%\[
%
\begin{eqnarray*}
&& \kappa_m E_{ \nu_{\E,N}}\bigl[x_i^m ( f
\circ\pi_{i,j} -f )^2 \bigr]
\\
&&\qquad\le\kappa_m (4K-3) \Biggl\{ 3 \sum
_{k=i}^{j-1} E_{ \nu_{\E,N}}\bigl[x_k^m
\bigl\{ f(\pi_{k,k+1}x )-f( x )\bigr\}^2 \bigr]
\\
&&\hspace*{91pt}{} + \sum_{k=i}^{j-2} E_{ \nu_{\E,N}}
\bigl[x_{k+2}^m \bigl\{f(\pi_{k,k+2}x )-f( x )\bigr
\}^2 \bigr] \Biggr\}
\\
&&\qquad\le\kappa_m (4K-3) 12 \sum_{k=i}^{j-1}
E_{ \nu_{\E
,N}}\bigl[(x_k+x_{k+1})^m
(D_{k,k+1}f)^2 \bigr]
\\
&&\qquad\quad{} + 8 (4K-3) \sum_{k=i}^{j-1}
E_{ \nu_{\E,N}}\bigl[ (x_k+x_{k+1})^m
(D_{k,k+1}f)^2\bigr]
\\
&&\qquad\le|j-i| (48 \kappa_m + 32) \sum
_{k=i}^{j-1} E_{ \nu_{\E
,N}}\bigl[(x_k+x_{k+1})^m
(D_{k,k+1}f)^2 \bigr]
\\
&&\qquad\le104 |j-i| \sum_{k=i}^{j-1}
E_{ \nu_{\E,N}}\bigl[(x_k+x_{k+1})^m
(D_{k,k+1}f)^2 \bigr],
\end{eqnarray*}
%
%\]
%
where we use $\kappa_m \le\frac{3}{2}$ in the last inequality.
\end{pf}

%re3.1 #&#
%
\begin{remark}
The argument used for the estimate (\ref{eq:2-estimate}) is exactly
the same as the one used in Section~4.2 of \cite{BCTV}. The method is
also used in the proof of Theorem~\ref{thm:convex} in the next section.
\end{remark}

Next is the main result in this section, which allows us to compare the
Dirichlet forms associated with the nearest neighbor interaction model
and the long-range interaction model.

%pr3.1 #&#
%
\begin{proposition}\label{prop:moving}
There exists a positive constant $C$ depending only on $m$ such that
for any $m \ge0$, $\E>0$, $N \ge2$ and $1 \le i < j \le N$,\vspace*{-1pt}
\[
\kappa_m E_{ \nu_{\E,N}}\bigl[ (x_i+x_j)^m
(D_{i,j}f)^2 \bigr] \le C |j-i| \sum
_{k=i}^{j-1} E_{ \nu_{\E,N}}\bigl[
(x_k+x_{k+1})^m (D_{k,k+1}f)^2
\bigr]\vspace*{-1pt}
\]
holds for all $f \in L^2( \nu_{\E,N})$.
\end{proposition}

\begin{pf}
Since if $i=j-1$, taking $C=\frac{3}{2}$, the statement is obvious
(with the fact $\kappa_m\le\frac{3}{2}$), we assume $i <j -1$.
First, we recall that\vspace*{-1pt}
%
%e3.4 #&#
%
\begin{eqnarray}
\label{eq:LRproof0} &&E_{ \nu_{\E,N}} \bigl[ (x_i+x_j)^m
(D_{i,j}f)^2 \bigr]
\nonumber
\\[-1pt]
&&\qquad= E_{ \nu_{\E,N}}\biggl[
(x_i+x_j)^m \biggl\{ \int
^1_0 P^{*}(d\a) f(T_{i,j,\a}
x)-f(x) \biggr\} ^2 \biggr]
\\[-1pt]
%[-8pt]
%\\[-8pt]
&&\qquad = \frac{1}{2} E_{ \nu_{\E,N}}\biggl[ (x_i+x_j)^m
\int^1_0 P^{*}(d\a) \bigl\{
f(T_{i,j,\a} x)-f(x) \bigr\}^2 \biggr].
\nonumber
\end{eqnarray}
Since $ (x_i+x_j)^m \le2^m(x_i^m+x_j^m)$, the last term of (\ref
{eq:LRproof0}) is bounded from above by
%
%e3.5 #&#
%
\begin{eqnarray}
\label{eq:LRproof1} &&2^{m-1} E_{ \nu_{\E,N}}\biggl[ x_i^m
\int^1_0 P^{*}(d\a) \bigl\{
f(T_{i,j,\a
} x)-f(x) \bigr\}^2 \biggr]
\nonumber
\\[-9pt]
\\[-9pt]
&&\qquad{} +2^{m-1} E_{ \nu_{\E,N}}\biggl[ x_j^m
\int^1_0 P^{*}(d\a) \bigl\{
f(T_{i,j,\a} x)-f(x) \bigr\}^2 \biggr].
\nonumber
\end{eqnarray}
%
%Therefore, to finish the proof, we only need to show that for all $i<
%j$,
%\begin{eqnarray}\label{eq:particlelemma}
% E_{ \nu_{\E,N}}[ x_i^m \int^1_0 P^{*,d}(d\a) \{ f(T_{i,j,\a} x)-f(x)
%\}^2 ] \le C_1 N \sum_{k=i}^{j-1} E_{ \nu_{\E,N}}[(x_k+x_{k+1})^m
%(D_{k,k+1}f)^2 ].
%\end{eqnarray}
%The second term of (\ref{eq:LRproof1}) is estimated in the same manner.

%If $j=i+1$, then (\ref{eq:particlelemma}) is obvious, so from now on
%we assume that $i+1 <j$.

%Now, we introduce operators $\pi_{i,j}: \mathcal{S}_{\E,N} \to
%\mathcal{S}_{\E,N}$ for $i \neq j$ as
%\[
% (\pi_{i,j} x )_k = \cases{
% x_k & \mbox{if $k \neq i,j$,} \\
% x_j & \mbox{if $k=i$,} \\
% x_i & \mbox{if $k=j$,}
%}
%\]
%and $\pi_{i,i}x=x$.
%$\sigma^{i,j}:= \pi_{j-1,j} \circ\pi_{j-2,j-1} \cdots\circ
%\pi_{i,i+1}$ and $\tilde{\sigma}^{i,j}:= \pi_{i,i+1} \circ
%\pi_{i+1,i+2} \cdots\circ\pi_{j-1,j}$.
Since the second term of (\ref{eq:LRproof1}) can be estimated in the
same manner as the first term, we only estimate the first term of (\ref
{eq:LRproof1}). Rewrite the term $f(T_{i,j,\a} x)-f(x)$ as
%
%\[
%
\begin{eqnarray*}
&&f(T_{i,j,\a} x)-f(x)
\\
&&\qquad= \bigl\{f\bigl(\pi_{i,j-1}\bigl(
T_{j-1,j,\a}( \pi_{i,j-1}x )\bigr)\bigr)-f\bigl( T_{j-1,j,\a}(
\pi_{i,j-1}x )\bigr)\bigr\}
\\
&&\qquad\quad{}+ \bigl\{f\bigl(T_{j-1,j,\a}( \pi_{i,j-1}x )\bigr)-f(
\pi_{i,j-1}x )\bigr\} + \bigl\{f( \pi_{i,j-1}x )-f(x)\bigr\}.
\end{eqnarray*}
%
%\]
%
Then, using the Schwarz inequality, we can bound the first term of (\ref
{eq:LRproof1}) from above by
%
%e3.6 #&#
%
\begin{eqnarray}
\label{schwarz} %
%\begin{array}{cc}
&& E_{ \nu_{\E,N}}\biggl[
x_i^m \int^1_0
P^{*}(d\a) \bigl\{ f\bigl(\pi_{i,j-1}\bigl( T_{j-1,j,\a}(
\pi_{i,j-1}x )\bigr)\bigr)-f\bigl(T_{j-1,j,\a}( \pi_{i,j-1}x
)\bigr) \bigr\} ^2\biggr]
\nonumber
\\
&&\qquad{} + E_{ \nu_{\E,N}}\biggl[ x_i^m \int
^1_0 P^{*}(d\a) \bigl\{ f
\bigl(T_{j-1,j,\a}( \pi_{i,j-1}x )\bigr)-f( \pi_{i,j-1}x )
\bigr\} ^2 \biggr]
\\
&&\qquad{} + E_{ \nu_{\E,N}}\biggl[ x_i^m \int
^1_0 P^{*}(d\a) \bigl\{ f(
\pi_{i,j-1}x )-f(x)\bigr\}^2 \biggr]
\nonumber
%\end{array}
%
\end{eqnarray}
up to constant depending only on $m$. We estimate three terms of (\ref
{schwarz}) separately.

The last term of (\ref{schwarz}) is equal to
\[
E_{ \nu_{\E,N}}\bigl[x_i^m \bigl\{f(
\pi_{i,j-1}x )-f( x )\bigr\}^2 \bigr]
\]
and, therefore, we can apply Lemma~\ref{lem:moving}.

%Since $E_{k,k+1}f(\pi_{k,k+1}x)=E_{k,k+1}f(x)$ and $ \nu_{\E,N}$ is
%invariant under the change of variable $y= \pi_{k,k+1}x$, we obtain
%that
%\[
%\begin{array}{cc}
%&E_{ \nu_{\E,N}}[ x_k^m \{f(\pi_{k,k+1}x )-f( x )\}^2 ] \\
%&= E_{ \nu_{\E,N}}[ x_k^m \{f(\pi_{k,k+1}x )-(E_{k,k+1}f)(
%\pi_{k,k+1}x)+(E_{k,k+1}f)(x)- f( x )\}^2 ] \\
%& \le2 E_{ \nu_{\E,N}}[ x_{k+1}^m\{f(x )-(E_{k,k+1}f)(x) \}^2] + 2E_{
%\nu_{\E,N}}[ x_k^m \{(E_{k,k+1}f)(x)- f( x )\}^2 ] \\
%& \le4 E_{ \nu_{\E,N}}[ (x_k+x_{k+1})^m (D_{k,k+1}f)^2].
%\end{array}
%\]

By the change of variable, the second term of (\ref{schwarz}) is
rewritten as
\[
E_{ \nu_{\E,N}} \biggl[ x_{j-1}^m \int
^1_0 P^{*}(d\a) \bigl\{
f(T_{j-1,j,\a} x )-f( x)\bigr\}^2 \biggr]
\]
which is obviously bounded from above by
\begin{eqnarray*}
&&E_{ \nu_{\E,N}} \biggl[(x_{j-1}+x_j)^m \int
^1_0 P^{*}(d\a) \bigl\{
f(T_{j-1,j,\a} x )-f( x) \bigr\}^2 \biggr]
\\
&&\qquad= 2 E_{ \nu_{\E,N}}\bigl[(x_{j-1}+x_j)^m
(D_{j-1,j}f)^2 \bigr].
\end{eqnarray*}

Finally, we study the first term of (\ref{schwarz}). By the same way
as the second term, the term is rewritten as
%
%\[
%
\begin{eqnarray*}
&&E_{ \nu_{\E,N}} \biggl[ x_{j-1}^m \int
^1_0 P^{*}(d\a) \bigl[f\bigl(
\pi_{i,j-1}( T_{j-1,j,\a} x)\bigr)-f(T_{j-1,j,\a}x )
\bigr]^2 \biggr]
\\
&&\qquad{} = E_{ \nu_{\E,N}} \bigl[ x_{j-1}^m
E_{ \nu_{\E,N}} \bigl[(f \circ\pi_{i,j-1}-f )^2 |
\F_{j-1,j}\bigr] \bigr]
\end{eqnarray*}
%
%\]
%
and since $x_{j-1}+x_j$ is measurable with respect to $\F_{j-1,j}$,
the last expression is bounded from above by
%
%\[
%
\begin{eqnarray*}
&& E_{ \nu_{\E,N}} \bigl[ (x_{j-1}+x_j)^m
E_{ \nu_{\E,N}} \bigl[(f \circ\pi_{i,j-1}-f )^2 |
\F_{j-1,j}\bigr] \bigr]
\\
&&\qquad= E_{ \nu_{\E,N}} \bigl[ (x_{j-1}+x_j)^m
(f \circ\pi_{i,j-1}-f )^2\bigr].
\end{eqnarray*}
%
%\]
%
Then, using the trivial inequality again, we conclude that the last
term is bounded by
%
%e3.7 #&#
%
\begin{equation}
\label{eq:LRproof2}\quad2^m E_{ \nu_{\E,N}} \bigl[ x_{j-1}^m
(f \circ\pi_{i,j-1}-f )^2 \bigr] + 2^m
E_{ \nu_{\E,N}}\bigl[ x_j^m (f \circ
\pi_{i,j-1}-f )^2 \bigr].
\end{equation}

Though the first term of (\ref{eq:LRproof2}) is directly estimated by
Lemma~\ref{lem:moving}, we need to treat the second term carefully.
Precisely, since $\pi_{i,j-1}=\pi_{i,j} \circ\pi_{i,j-1} \circ
\pi_{j-1,j}$, we have
%
%\[%\label{eq:firstterm}
%
\begin{eqnarray*}
E_{ \nu_{\E,N}} \bigl[ x_j^m \bigl(f (
\pi_{i,j-1}x)-f(x) \bigr)^2 \bigr] &\le&3 E_{ \nu
_{\E,N}}
\bigl[ x_j^m \bigl(f (\pi_{j-1,j}x)-f(x)
\bigr)^2 \bigr]
\\
&&{}+ 3E_{ \nu_{\E,N}} \bigl[ x_{j-1}^m \bigl(f (
\pi_{i,j-1}x)-f(x) \bigr)^2 \bigr]
\\
&&{} + 3 E_{
\nu_{\E,N}}
\bigl[ x_{i}^m \bigl(f (\pi_{i,j}x)-f(x)
\bigr)^2 \bigr].
\end{eqnarray*}
%
%\]
%
Then we can apply Lemma~\ref{lem:moving} to complete the
proof.
\end{pf}

%We denote the spectral gap of $\mathcal{L}^{*,m,\ga}_{LR}$ in $L^2(
%\nu_{\E,N})$ by $\la^{*,m}_{LR}(\E,N)$ and $\la^{*,m}_{LR}(1,N)$ by $
%\la^{*,m}_{LR}(N)$. By the exactly same argument with Lemma
%\ref{lem:scaling}, we obtain the relation that
%\[
%\la^{*,m}_{LR}(\E,N)= \E^m\la^{*,m}_{LR}(N)
%\]
%for all $\E>0$ and $N \in\N$.

%Since the Dirichlet forms associated with $\mathcal{L}^{*,m,\ga}$ and $
%\mathcal{L}^{*,m,\ga}_{LR}$, denoted by $D^*$ and $D^{*}_{LR}$, are
%\[
%5D^*(f)=\sum_{i=1}^{N-1} E_{ \nu_{\E,N}}[ (x_i+x_{i+1})^m
%(D_{i,i+1}f)^2 ], D^*_{LR} (f)=\frac{1}{N} \sum_{i<j} E_{ \nu_{
%\E,N}}[ (x_i+x_j)^m (D_{i,j}f)^2 ]
%\]
%respectively,

\begin{pf*}{Proof of Theorem~\ref{thm:compare}}
Noting the explicit expressions of $D^{*} (f)$ and $D^{*}_{\mathrm
{LR}}(f)$, by
Proposition~\ref{prop:moving}, we have
\[
\kappa_m D^{*}_{\mathrm{LR}} (f) \le\frac{C}{N}
\sum_{i<j} |j-i| \sum_{k=i}^{j-1}
E_{ \nu_{\E,N}}\bigl[ (x_k+x_{k+1})^m
(D_{k,k+1}f)^2\bigr] \le C' N^2
D^{*}(f)
\]
for any $f \in L^2(\nu_{\E,N})$ with positive constants $C$ and $C'$
depending only on $m$. Then, by the definition of the spectral gap, the
proof is complete.
\end{pf*}

%s4 #&#
\section{Spectral gap for the long-range model}
\label{sec:long-range}

In this section, to show that $\inf_N \lambda^{*,m}_{\mathrm
{LR}}(1,N) >0$, we
give a proof of Theorem~\ref{thm:convex} and Proposition~\ref{prop:compm2m}.

\begin{pf*}{Proof of Theorem~\ref{thm:convex}}
By the definition of the spectral gap, it is sufficient to show that
for all $\E> 0$, $N \ge2$ and $f \in L^2 ( \nu_{\E,N})$,
\[
\kappa_m \frac{1}{N} \sum_{i<j}
E_{ \nu_{\E,N}}\bigl[(D_{i,j}f)^2\bigr] \le2
\E^{-m} \frac{1}{N} \sum_{i<j}
E_{ \nu_{\E,N}}\bigl[(x_i+x_j)^m(D_{i,j}f)^2
\bigr]
\]
holds.

For $m \ge1$, since $x^m$ is a convex function, we have $\E^m \le
\frac{1}{N} \sum_{k=1}^N x_k^m$ and, therefore,
\[
\frac{1}{N} \sum_{i<j} E_{ \nu_{\E,N}}
\bigl[(D_{i,j}f)^2\bigr] \le\frac
{1}{N} \sum
_{i<j} \sum_{k=1}^N
\frac{\E^{-m}}{N} E_{ \nu_{\E
,N}}\bigl[x_k^m(D_{i,j}f)^2
\bigr].
\]
If $k \neq i, j$, we have
\begin{eqnarray*}
E_{ \nu_{\E,N}} \bigl[x_k^m(D_{i,j}f)^2
\bigr] &=& E_{ \nu_{\E
,N}}\bigl[x_k^m(E_{i,j}f-f)^2
\bigr]
\\
& \le& E_{ \nu_{\E,N}}\bigl[(x_i+x_j+x_k)^m(E_{i,j}f-f)^2
\bigr]
\\
& =& E_{ \nu_{\E,N}}\bigl[(x_i+x_j+x_k)^m
E_{ \nu_{\E,N}}\bigl[ (E_{i,j}f-f)^2 | \F_{i,j,k}
\bigr] \bigr]
\\
& \le& E_{ \nu_{\E,N}}\bigl[(x_i+x_j+x_k)^m
E_{ \nu_{\E,N}}\bigl[ (E_{i,j,k}f-f)^2 | \F_{i,j,k}
\bigr] \bigr],
\end{eqnarray*}
where we use the relation
\begin{eqnarray*}
E_{ \nu_{\E,N}}\bigl[ (E_{i,j,k}f-f)^2 | \F_{i,j,k}
\bigr] &=& E_{ \nu_{\E,N}}\bigl[ (E_{i,j,k}f-E_{i,j}f)^2
| \F_{i,j,k}\bigr]
\\
&&{}+ E_{ \nu_{\E,N}}\bigl[ (E_{i,j}f-f)^2
| \F_{i,j,k}\bigr]
\end{eqnarray*}
again.

Then, by the definition of the spectral gap for $3$-site system and the
scaling relation (\ref{eq:scaling2}),
\begin{eqnarray*}
&&\kappa_m E_{ \nu_{\E,N}} \bigl[(x_i+x_j+x_k)^m
E\bigl[ (f-E_{i,j,k}f)^2 | \F_{i,j,k}\bigr] \bigr]
\\
&&\qquad\le E_{ \nu_{\E,N}}\bigl[ (x_i+x_k)^m
(D_{i,k}f)^2 + (x_k+x_j)^m
(D_{k,j}f)^2\bigr].
\end{eqnarray*}
%
%Also,
%\begin{eqnarray*}
%E_{\nu_{N,N}} & [(x_i+x_j+x_k)^{\a-1}(E_{i,j}f-E_{i,j.k}f)^2] = E_{
%\nu_{N,N}}[(x_i+x_j+x_k)^{\a-1} E[ (E_{i,j}f-E_{i,j.k}f)^2 |
%\F_{i,j,k}] ] \\
%& \le E_{\nu_{N,N}} [(x_i+x_j+x_k)^{\a-1}(f-E_{i,j.k}f)^2] \\
%&\le\frac{1}{3\la^*(3,1)}E[ (x_i+x_j)^{\a-1} (D_{i,j}f)^2 +
%(x_i+x_k)^{\a-1} (D_{i,k}f)^2 + (x_j+x_k)^{\a-1} (D_{j,k}f)^2].
%\end{eqnarray*}

Therefore, noting $x_i^m+x_j^m \le(x_i+x_j)^m$ for $m \ge1$ and
$\kappa_m \le\frac{3}{2}$, by summing terms, we obtain
\begin{eqnarray*}
&&\frac{\kappa_m}{N} \sum_{i<j} E_{ \nu_{\E,N}}
\bigl[(D_{i,j}f)^2\bigr]
\\
&&\qquad\le\frac{ \E^{-m}}{N} \sum
_{i<j} \frac{1}{N} \biggl( \frac{3}{2} +
2(N-2) \biggr) E_{ \nu_{\E,N}}\bigl[(x_i+x_j)^m(D_{i,j}f)^2
\bigr]
\\
&&\qquad\le\frac{2\E^{-m}}{N} \sum_{i<j}
E_{ \nu_{\E
,N}}\bigl[(x_i+x_j)^m(D_{i,j}f)^2
\bigr].
\end{eqnarray*}
%
%To obtain the last inequality, we use the fact that $1 \le\frac{4}{
%\la^{*,m}(\frac{1}{3},3)}$ which can be justified directly from
%Theorem~\ref{thm:3-stie}.
\upqed
\end{pf*}

\begin{pf*}{Proof of Proposition~\ref{prop:compm2m}}
Here, we use the idea developed by Caputo in \cite{Ca08} and
generalize it. First, recall a well-known equivalent characterization
of the spectral gap of a generator $\mathcal{L}$ in $L^2(\nu)$ as the
largest constant $\lambda$ such that the inequality
%
%e4.1 #&#
%
\begin{equation}
\label{eq:gap} E_{\nu}\bigl[(\mathcal{L}f)^2\bigr] \ge
\lambda E_{\nu}\bigl[f(-\mathcal{L})f\bigr]
\end{equation}
holds for all $f \in L^2(\nu)$. Then, by Schwarz's inequality, we have
\begin{eqnarray*}
\lambda&=& \inf\biggl\{ \frac{ E_{\nu}[(\mathcal{L}f)^2] }{E_{\nu
}[f(-\mathcal{L})f]} \Big| E_{\nu}[f] =0, f \in
L^2(\nu) \biggr\}
\\
&\ge&\inf\biggl\{ \sqrt{ \frac{ E_{\nu}[(\mathcal{L}f)^2] }{
E_{\nu
}[f^2]} } \Big| E_{\nu}[f] =0, f \in
L^2(\nu) \biggr\}.
\end{eqnarray*}
Now, we have
\[
E_{\nu_{\E,N}}\bigl[\bigl(\mathcal{L}^{*,m}_{\mathrm{LR}}f
\bigr)^2\bigr]=\frac{1}{N^2}\sum_{b,b'}E_{\nu_{\E,N}}[h_b
h_{b'} D_bf D_{b'}f],
\]
where the sum runs over all $\bigl({N\atop2}\bigr)$ unordered pairs
$b$ and
$b'$, and $h_b(x)=(x_i+x_j)^m$ if $b=\{i,j\}$. We write $b \sim b'$
when two unordered pairs have at least one common vertex (including the
case $b=b'$). Otherwise, we write $b \nsim b'$. We observe that if $b
\nsim b'$, then $E_b$ and $E_{b'}$ commute. Moreover, $h_b$ and
$h_{b'}$ are both measurable with respect to $\F_{b}$ and $\F_{b'}$
where $\F_b=\F_{i,j}$ for $b=\{i,j\}$. Therefore, using $D_b^2=-D_b$
and self-adjointness of $D_b$ and $D_{b'}$, for $ b \nsim b'$
\begin{eqnarray*}
E_{\nu_{\E,N}}[h_b h_{b'} D_bf
D_{b'}f] & =& - E_{\nu_{\E,N}}\bigl[h_b
h_{b'} (D_{b'}D_bf) (D_{b'}f)\bigr]
\\
& =&E_{\nu_{\E,N}}\bigl[h_b h_{b'}
(D_{b'}D_bf) ^2\bigr] \ge0.
\end{eqnarray*}
Therefore, it follows that
\[
E_{\nu_{\E,N}}\bigl[\bigl(\mathcal{L}^{*,m}_{\mathrm{LR}}f
\bigr)^2\bigr] \ge\frac{1}{N^2}\sum_{b,b': b \sim b'}E_{\nu_{\E,N}}[h_b
h_{b'} D_bf D_{b'}f].
\]
Now, we denote unordered triples $\{i,j,k\}$ of distinct vertices by
$T$ (triangles). We say that $b \in T$ if $b=\{i,j\}$ and $i,j \in T$.
Clearly, if $b \sim b'$ and $b \neq b'$ there is only one triangle $T$
such that $b,b' \in T$. We may therefore write
\begin{eqnarray*}
&&\sum_{b,b': b \sim b'} E_{\nu_{\E,N}}[h_b
h_{b'} D_bf D_{b'}f]
\\
&&\qquad= \sum_{b,b': b \sim b', b \neq b'}E_{\nu_{\E,N}}[h_b
h_{b'} D_bf D_{b'}f] +\sum
_{b}E_{\nu_{\E,N}}\bigl[h_b ^2
(D_bf)^2\bigr]
\\
&&\qquad= \sum_{T} \sum
_{b,b' \in T} E_{\nu_{\E,N}}[h_b h_{b'}
D_bf D_{b'}f] -(N-3)\sum_{b}E_{\nu_{\E,N}}
\bigl[h_b ^2 (D_bf)^2\bigr]
\end{eqnarray*}
since for every $b$ there are exactly $N-2$ triangles $T$ such that $b
\in T$.

Let us now apply inequality (\ref{eq:gap}) for $\mathcal
{L}_{\mathrm{LR}}^{m,*}$ to a fixed triangle $T$. Let $\F_T$ denote
the $\sigma
$-algebra generated by $\{x_l, l \notin T\}$. Then
\begin{eqnarray*}
&&\frac{1}{3}\sum_{b,b' \in T} E_{\nu_{\E,N}}[h_b
h_{b'} D_bf D_{b'}f | \F_T]
\\
&&\qquad\ge
\lambda^{*,m}_{\mathrm{LR}}\biggl(\frac{\sum_{i\in
T}x_i}{3},3\biggr) \sum
_{b}E_{\nu_{\E,N}}\bigl[h_b
(D_bf)^2 | \F_T\bigr]
\\
&&\qquad = \tilde{\kappa}_m\biggl(\sum_{i\in T}x_i
\biggr)^m \sum_{b}E_{\nu_{\E
,N}}
\bigl[h_b (D_bf)^2 | \F_T\bigr]
\\
&&\qquad= \tilde{\kappa}_m\sum_{b}E_{\nu_{\E
,N}}
\biggl[ \biggl(\sum_{i\in T}x_i
\biggr)^m h_b (D_bf)^2 \Big|
\F_T\biggr]
\\
&&\qquad \ge \tilde{\kappa}_m\sum_{b}E_{\nu_{\E,N}}
\bigl[h_b ^2 (D_bf)^2 |
\F_T\bigr],
\end{eqnarray*}
where we use $(\sum_{i\in T}x_i)^m $ is measurable with respect to $\F
_T$ and $ (\sum_{i\in T}x_i)^m \ge h_b$ for any $b \in T$. Taking $\nu
_{\E,N}$-expectation to remove the conditioning on $\F_T$, we obtain that
\[
E_{\nu_{\E,N}}\bigl[\bigl(\mathcal{L}^{*,m}_{\mathrm{LR}}f
\bigr)^2\bigr] \ge\frac{1}{N^2} \bigl( (3\tilde{
\kappa}_m -1) (N-2)+1 \bigr)\sum_{b}E_{\nu_{\E,N}}
\bigl[h_b ^2 (D_bf)^2\bigr].
\]
On the other hand, since
\[
E_{\nu_{\E,N}}\bigl[f\bigl(-\mathcal{L}^{2m,*}_{\mathrm{LR}}\bigr)f
\bigr]=\frac{1}{N}\sum_{b} E_{\nu_{\E,N}}
\bigl[h_b^2 (D_bf)^2\bigr],
\]
we have
\[
%\label{eq:gap2m}
\frac{1}{N}\sum_{b}
E_{\nu_{\E,N}}\bigl[h_b^2 (D_bf)^2
\bigr] \ge\lambda^{*,2m}_{\mathrm{LR}}(\E,N) E_{\nu_{\E,N}}
\bigl[f^2\bigr].
\]
Therefore, combining the above inequalities, we complete the proof.
\end{pf*}

%s5 #&#
\section{Examples}
\label{sec:example}

In this section, we present two interesting classes of stochastic
energy exchange models for which we can apply Theorem~\ref{thm:main}.

%\subsection{The rarely interacting billiard lattice}
%Here, we will give some examples where the.

%s5.1 #&#
\subsection{The rarely interacting billiard lattice}

As mentioned in the \hyperref[intro]{Intro}-\break\hyperref[intro]{duction}, the main motivation of the article
\cite{GKS} was to study the models studied in \cite{GG08,GG09}.
Gaspard and Gilbert argued that in the limit of rare
collisions, the dynamics of a billiard lattice becomes a Markov jump
process. The limiting process is actually in the class considered in
this paper. As shown in \cite{GKS}, the process studied in \cite
{GG09} has the generator of the mechanical form with
\begin{eqnarray*}
\Lambda_s(s)&=&s^{1/2}, \qquad\Lambda_r(
\beta)=\frac{\sqrt{2\pi}}{6} \frac{{1}/{2}+\b\vee(1-\b)}{\sqrt
{\b\vee(1-\b)}},
\\
P(\beta, d\a)&=&\frac{3}{2}\frac{1 \wedge\sqrt{{(\a\wedge (1-\a
))}/{(\b\wedge(1-\b))}}}{{1}/{2}+\b\vee(1-\b)}\,d\a.
\end{eqnarray*}
The symbol $\vee$ denotes the maximum and $\wedge$ denotes the minimum.
This process is reversible with respect to the product
Gamma-distribution with $\ga=\frac{3}{2}$. Moreover, it is shown in
\cite{GKS} that this measure is also reversible for the process given
by the generator corresponding to any other function $\La_s$ (while
keeping $\La_r$ and $P$ unchanged). Therefore, we consider the
generator given by $\La_s(s)=s^m$ for $m \ge0$, and denote the
spectral gap on $\mathcal{S}_{\E,N}$ of the process by $\lambda
_{GG3}^m(\E,N)$ where $3$ represents the dimension of the original
mechanical model.

Here, we also consider the process obtained from the two-dimensional
billiard lattice studied in \cite{GG08}.
Changing equations (3) and (5) in \cite{GG08} to our notation yields that
\begin{eqnarray*}
\Lambda_s(s)&=&s^{1/2}, \qquad\Lambda_r(
\beta)=\sqrt{\frac{8 (\b\vee
(1-\b))}{\pi^3}} \bigl( 2 E\bigl(\b^*\bigr) - \bigl(1- \b^*
\bigr)
K \bigl(\b^*\bigr) \bigr),
\\
P(\beta, d\a)&=&\frac{\tilde{P}(\beta, \a)}{\Lambda_r(\beta
)}d\a,
\end{eqnarray*}
where $\beta^*=\frac{\b}{1-\b} \wedge\frac{1-\b}{\b}$,
\[
\tilde{P}(\beta, \a)= \sqrt{\frac{2}{\pi^3}} \times%
\cases{
\sqrt{\displaystyle\frac{1}{1-\b}}K\biggl(\sqrt{\displaystyle
\frac{\a}{1-\b}}\biggr) & \quad$\mbox{if
} 0 \le\a\le\bigl(\b\wedge(1-\b)\bigr)$,\vspace*{2pt}
\cr
\sqrt{\displaystyle\frac{1}{1-\a}}K\biggl(
\sqrt{\displaystyle\frac{\b}{1-\a}}\biggr) & \quad$\mbox{if }
\b\le\a\le(1-\b)$,\vspace*{2pt}
\cr
\sqrt{
\displaystyle\frac{1}{\a}}K\biggl(\sqrt{\displaystyle\frac{1-\b
}{\a}}\biggr) & \quad$\mbox{if } (1-
\b) \le\a\le\b$,\vspace*{2pt}
\cr
\sqrt{\displaystyle\frac{1}{\b}}K\biggl(\sqrt{\displaystyle\frac
{1-\a}{\b}}
\biggr) & \quad$\mbox{if } \bigl(\b\vee(1-\b) \bigr) \le\a\le
1$, } %
\]
and
\[
K(t)=\int_0^{{\pi}/{2}} \frac{1}{\sqrt{1-t^2 \sin^2\theta}} \,
d\theta,
\qquad E(t)=\int_0^{{\pi}/{2}} \sqrt{1-t^2
\sin^2\theta} \,d\theta.
\]
Since the underlying mechanical model has a two-dimensional
configuration space for each of the constituent particles, this process
is reversible with respect to the product Gamma-distribution with $\ga
=1$. In the same manner as before, this measure is also reversible for
the process given by the generator corresponding to any other function
$\La_s$ (while keeping $\La_r$ and $P$ unchanged). So, we consider
the generator given by $\La_s(s)=s^m$ for $m \ge0$, and denote the
spectral gap on $\mathcal{S}_{\E,N}$ of the process by $\lambda
_{GG2}^m(\E,N)$.

Since these processes are of the mechanical form (\ref
{eq:productform}), $\lambda_{GG3}^m(\E, 2)=\break \La_s(2\E)\* \tilde
{C}_{GG3} = (2\E)^m\tilde{C}_{GG3} $ and $\lambda_{GG2}^m(\E,
2)=\La_s(2\E) \tilde{C}_{GG2} = (2\E)^m\tilde{C}_{GG2} $
where 
$\tilde{C}_{GG3}=\lambda_{GG3}^0(1,2)$ and $\tilde{C}_{GG2}=\lambda
_{GG2}^0(1,2)$.
%Here, $\ga$, $\La_r$ and $P$ are those associated to respective models.

%le5.1 #&#
%
\begin{lemma}
%For any $m \ge0$,
%\[
%\la_{GG3}^m(\E, 2)=(2\E)^m\la_{GG3}^0(1,2), \la_{GG2}^m(\E, 2)=(2\E)^m
%\la_{GG2}^0(1,2)
%\]
%and
%
\[
\tilde{C}_{GG3} >0,\qquad\tilde{C}_{GG2} >0
\]
hold.
\end{lemma}

\begin{pf}
%We have the scaling property $\la_{GG3}^m(\E, 2)=(2\E)^m\la_{GG3}^m(
%\frac{1}{2},2)$ by the same argument of the proof of Lemma
%\ref{lem:scaling}. Then, since $x_1+x_2=1$ for any $x \in\mathcal{S}_{
%\frac{1}{2},2}$, $\la_{GG3}^m(\frac{1}{2},2)=\la_{GG3}^0(
%\frac{1}{2},2)$ for any $m \ge0$. $\la_{GG2}^m(\E, 2)=\E^m
%\la_{GG2}^0(1,2)$ is shown in the same way.
%
The fact $\lambda_{GG3}^0(1,2) >0$ is shown in \cite{GKS} since the
case $m=0$ satisfies the condition assumed in Lemma~5.1 of \cite{GKS}.
To show $\tilde{C}_{GG2} >0$, we write down the explicit Dirichlet
form associated to the two-dimensional model:
\[
\tilde{C}_{GG2} = \inf\biggl\{ \frac{ \int_0^1 d\b\int_0^1
\tilde
{P}(\b,\a) \,d\a[f(\a)-f(\b)]^2}{\int_0^1 d\b\int_0^1 d\a[f(\a
)-f(\b)]^2} \Big| f \in
L^2\bigl([0,1]\bigr) \biggr\}.
\]
Then, since $\tilde{P}(\b,\a) \ge\sqrt{\frac{2}{\pi^3}} K(0) =
\sqrt{\frac{1}{2\pi}}$ for all $0 \le\a, \b\le1$, we have
$\tilde{C}_{GG2} \ge\sqrt{\frac{1}{2\pi}}$.
\end{pf}

With this result, we can apply Corollary~\ref{cor:mechanicalcase} to
these models directly and obtain the following corollary.

%co5.1 #&#
%
\begin{corollary}
For any $m \ge0$, there exists positive constants $C$ and $C'$
independent of $\E$ and $N$ such that
\[
\lambda_{GG3}^m(\E,N) \ge C\E^m
\frac{1}{N^2},\qquad\lambda_{GG2}^m(\E,N) \ge
C' \E^m\frac{1}{N^2}.
\]
\end{corollary}

%s5.2 #&#
\subsection{Stick processes}

The class of stick processes studied in \cite{FIiiiS} is another
interesting example in the class we considered. The model was first
introduced as the microscopic model which scales to the porous medium
equations. The generator of the model is described by the rate function
and the probability kernel of the mechanical form as
\[
\Lambda_s(s)=s^m,\qquad\Lambda_r(\beta)=
\beta^m+(1-\beta)^m,\qquad P(\beta, d\a)=
\frac{m|\b-\a|^{m-1}}{\Lambda_r(\beta)}\,d\a,
\]
where $m$ is a positive parameter. $\a-1$ in \cite{FIiiiS} is associated
to $m$ here. The process is reversible with respect to a product
Gamma-distribution with $\ga=1$.
% Since the model is of gradient type, Feng et al. showed the
%hydrodynamic limit for the process without any knowledge of the
%spectral gap estimates. Under the diffusive space-time scaling, they
%obtained the porous medium equation
%\[
%\partial_t \E(t,u) = \Delta(\E^{m+1}(t,u)) ) = (m+1)\partial_u (
%\E^m(t,u) \partial_u \E(t,u))
%\]
%as the hydrodynamic equation.

Denote the spectral gap for the stick process with parameter $m$ on
$\mathcal{S}_{\E,N}$ by $\lambda_{st}^m (\E,N)$. By the definition,
\[
\lambda^m_{st}(1,2)=\inf_{f} \biggl
\{ \frac{ m \int^1_0 \int^1_0 \{
f(t)-f(s)\}^2 |t-s|^{m-1} \,ds\,dt}{\int^1_0 \int^1_0 \{f(t)-f(s)\}
^2\,ds\,dt} \Big| f \in L^2(\nu_{1,2}) \biggr\}.
\]
Therefore, it is obvious that $\lambda^1_{st}(1,2)=1$ and $\lambda
^m_{st}(1,2)>0$ for $0 < m \le1$ as $|t-s|^m \ge|t-s|$ for any $0 \le
t,s \le1$ and $0 < m \le1$.

On the other hand, for $m >1$, we need to show $\lambda^m_{st}(1,2)
>0$ more carefully. Let $k=m-1 > 0$. For $f$ satisfying $E_{\nu
_{1,2}}[f]=\int_0^1 f(t)\,dt=0$, we have
\begin{eqnarray*}
&&\int^1_0 \int^1_0
\bigl\{f(t)-f(s)\bigr\}^2 |t-s|^k \,ds\,dt
\\
&&\qquad= \frac{2}{k+1} \int^1_0
f(t)^2 \bigl(t^{k+1}+(1-t)^{k+1}\bigr) \,dt - 2
\int^1_0 \int^1_0
f(t)f(s) |t-s|^k \,ds\,dt.
\end{eqnarray*}
Then, for any $ a >0$,
\begin{eqnarray*}
&&\biggl| \int^1_0 \int^1_0
f(t)f(s) |t-s|^k \,ds\,dt \biggr|
\\
&&\qquad= \biggl| \int^1_0
\int^1_0 f(t)f(s) \bigl( |t-s|^k -
a^k \bigr) \,ds\,dt \biggr|
\\
&&\qquad \le \int^1_0 \int^1_0
\bigl| f(t) \bigr| \bigl|f(s) \bigr| \bigl| |t-s|^k - a^k \bigr|
\,ds\,dt
\\
&&\qquad \le \int^1_0 \int^1_0
\frac{1}{2} \bigl(\bigl| f(t) \bigr|^2+ \bigl|f(s) \bigr|^2 \bigr
) \bigl|
|t-s|^k - a^k\bigr | \,ds\,dt
\\
&&\qquad = \int^1_0 f(t)^2 \int
^1_0 \bigl| |t-s|^k - a^k \bigr|
\,ds\,dt.
\end{eqnarray*}

By simple calculations,
\begin{eqnarray*}
&&\int^1_0 \bigl| |t-s|^k -
a^k \bigr| \,ds %= \int^1_0 | |s-t|^k - a^k | ds =
%\int^{1-t}_{-t} | |q|^k -a^k | dq \\
\\
&&\qquad= \int^{1-t}_0
\bigl| q^k -a^k \bigr| \,dq + \int^t_0
\bigl| q^k -a^k \bigr| \,dq
\\
&&\qquad = \int^{1-t}_{(1-t) \wedge a } \bigl(q^k
-a^k\bigr) \,dq - \int^{(1-t) \wedge a
}_0
\bigl(q^k -a^k\bigr) \,dq
\\
&&\qquad\quad{} + \int^{t}_{t \wedge a } \bigl(q^k
-a^k\bigr) \,dq - \int^{t \wedge a }_0
\bigl(q^k -a^k\bigr) \,dq
\\
&&\qquad = \frac{1}{k+1} \bigl( (1-t)^{k+1} +t^{k+1} - 2 \bigl(
(1-t) \wedge a \bigr)^{k+1} - 2 ( t \wedge a )^{k+1} \bigr)
\\
&&\qquad\quad {} - a^k \bigl( 1 - 2 \bigl( (1-t) \wedge a \bigr) -
2 ( t \wedge
a ) \bigr)
\\
&&\qquad \le \frac{1}{k+1} \bigl( (1-t)^{k+1} +t^{k+1} \bigr) -
a^k (1 -4a). %\\
% \le\frac{(1-t)^{k+1} +t^{k+1}}{k+1} \Big( 1 - (k+1) a^k (1 -4a)
%\Big).
\end{eqnarray*}
Therefore,
\begin{eqnarray*}
&& \int^1_0 \int^1_0
\bigl\{f(t)-f(s)\bigr\}^2 |t-s|^k \,ds\,dt
\\
&&\qquad\ge\frac{2}{k+1} \int^1_0
f(t)^2 \bigl(t^{k+1}+(1-t)^{k+1}\bigr) \,dt
\\
&&\qquad\quad{} - 2 \int^1_0
f(t)^2 \biggl( \frac{(1-t)^{k+1} +t^{k+1} }{k+1} - a^k (1 -4a) \biggr
)\,
dt
\\
&&\qquad= 2 \int^1_0 f(t)^2
a^k (1 -4a)\, dt.
\end{eqnarray*}
Namely, for any $0<a <\frac{1}{4}$, we have $\lambda_{st}^m(1,2) \ge
a^{m-1}(1-4a) >0$.

With this result, we can apply Corollary~\ref{cor:mechanicalcase} to
the stick process directly and obtain the following corollary.

%co5.2 #&#
%
\begin{corollary}
For any $m > 0$, there exists a positive constant $C$ independent of
$\E$ and $N$ such that
\[
\lambda_{st}^m(\E,N) \ge C\E^m
\frac{1}{N^2}.
\]
\end{corollary}

\begin{appendix}\label{app}
%s6 #&#
\section*{Appendix: Spectral gap for $3$-site system}

In this Appendix, we give a proof of Lemma~\ref
{lem:kappa1}. From now
on, we fix $\nu=\nu_{1/3,3}^{\ga}$ and denote by $E$ the integration
with respect to $\nu$. For each $n \in\N$, let $P_n$ be the set of
polynomials of three variables of degree less than or equal to $n$ and
$\tilde{P}_n$ be the set of polynomials of one variable of degree less
than or equal to $n$.

Since $P_n$ is dense in $L^2(\nu)$, we have
\begin{eqnarray*}
\tilde{\kappa}_1 &=& \inf\biggl\{\frac{D^{*,1}_{\mathrm
{LR}}(f)}{E[f^2]}; E[f]=0, f \in
L^2(\nu) \biggr\}
\\
&=& \inf_{n \in\N} \inf\biggl\{
\frac
{D^{*,1}_{\mathrm{LR}}(f)}{E[f^2]}; E[f]=0, f \in P_n \biggr\}.
\end{eqnarray*}
Then, since $D^{*,1}_{\mathrm{LR}}(f)=\frac{1}{3}\sum_{i=1}^3
E[(1-x_i)(f-E[f|x_i])^2]$ where $E[f|x_i]=\break E[f|\G_i]$ and $\G_i$ is
the $\sigma$-algebra generated by $x_i$,
\[
\tilde{\kappa}_1 = \frac{1}{3} \inf_{n \in\N}
\inf\biggl\{ \frac{2
E[f^2] - \sum_{i=1}^3 E[(1-x_i)E[f|x_i]^2]}{E[f^2]}; E[f]=0, f \in P_n
\biggr\}.
\]

Therefore, to show $\tilde{\kappa}_1 > \frac{1}{3}$, we only need to
show that
%
%e6.1 #&#
%
\begin{equation}
\label{eq:statement1} \sup_{n \in\N} \sup\biggl\{ \frac{\sum_{i=1}^3
E[(1-x_i)E[f|x_i]^2]}{E[f^2]};
E[f]=0, f \in P_n \biggr\} < 1.
\end{equation}

Now, we construct a set of special functions which generates $P_n$.

First, for each $n \in\N$, let $J_n \in\tilde{P}_n$ be
%
%e6.2 #&#
%
\begin{equation}
\quad J_n(u)=\frac{\Gamma(n+\gamma)}{n!\Gamma(n+3\ga-1)}\sum_{m=0}^n
(-1)^m \binom{n} {m} \frac{\Gamma(n+m+3\gamma-1)}{\Gamma(m+\ga)}u^m.
\end{equation}
$\{J_n \}_{n \in\N} $ are orthogonal polynomials called the Jacobi
polynomials with parameters $(\ga-1,2\ga-1)$ on the interval $[0,1]$.
We choose the parameter since
$\{J_n \}_{n \in\N} $ are orthogonal with respect to the marginal of
$x_1$ under $\nu$, or precisely the beta distribution of parameters
$(\ga, 2\ga)$. By the construction, for $1 \le i \le3$,
%
%e6.3 #&#
%
\begin{equation}
\quad E\bigl[J_n(x_i)\bigr]=0 \qquad(n \in\N),\qquad E
\bigl[J_n(x_i)J_m(x_i)\bigr] = 0
\qquad(n \neq m).
\end{equation}

%le6.1 #&#
\setcounter{lemma}{0}
\begin{lemma}
For any $n \in\N$,
\[
E\bigl[J_n(x_i)|x_j\bigr]=
\nu_n J_n (x_j) \qquad\mbox{for } i \neq j,
\]
where $\nu_n =(-1)^n \frac{\Gamma(2\ga)\Gamma(n+\ga)}{\Gamma(\ga
)\Gamma(n+2\ga)}$.
\end{lemma}

\begin{pf}
We first remark that for any $f \in\tilde{P}_n$, $E[f(x_i)|x_j] \in
\tilde{P}_n$ as a function of $x_j$. Moreover, as shown in the proof
of Theorem~3.2 in \cite{S}, there exists a set of polynomials $\psi
_n$ which satisfies $E[\psi_n (x_i)|x_j]=\nu_n \psi_n (x_j)$ for $i
\neq j$ and $\psi_n \in\tilde{P}_n$ where $\nu_n =(-1)^n \frac
{\Gamma(2\ga)\Gamma(n+\ga)}{\Gamma(\ga)\Gamma(n+2\ga)}$. Then, since
%
%e6.4 #&#
%
\begin{equation}
\nu_n E\bigl[\psi_n (x_1)
\psi_m(x_1)\bigr] = E\bigl[ \psi_n
(x_2) \psi_m(x_1) \bigr]= \nu_m E
\bigl[\psi_n (x_2) \psi_m(x_2)
\bigr]
\end{equation}
and $\nu_n \neq\nu_m$ for $n \neq m$, $\{\psi_n\}_{n \in\N}$ are
orthogonal polynomials with respect to the marginal of $x_1$ under $\nu
$, which implies $J_n=c_n \psi_n$ for some $c_n \neq0$ and $
E[J_n(x_i)|x_j]=\nu_n J_n (x_j)$ for $i \neq j$.
\end{pf}

Next, we consider following polynomials $F_n, G_n, H_n \in P_n$:
%
%e6.5 #&#
%e6.6 #&#
%e6.7 #&#
%
\begin{eqnarray}
F_n(x_1,x_2,x_3) &
=&J_n(x_1)+J_n(x_2)+J_n(x_3),
\\
G_n(x_1,x_2,x_3) &
=&J_n(x_1)-J_n(x_3),
\\
H_n(x_1,x_2,x_3) &
=&J_n(x_1)-2J_n(x_2)+J_n(x_3).
\end{eqnarray}

For any $n \in\N$, let $Q_n$ denote a subspace of $P_n$ generated by
$F_0 := 1$ and $\{F_k,G_k, H_k \}_{1 \le k \le n}$ and $Q_n^{\perp}$
be the orthogonal complement of $Q_n$ of $P_n$ equipped with the inner
product induced from $L^2(\nu)$.

%pr6.1 #&#
%
\begin{proposition}
For any $n \in\N$ and $f \in Q_n^{\perp}$,
%
%e6.8 #&#
%
\begin{equation}
E[f|x_i]=0,\qquad1 \le\forall i \le3.
\end{equation}
\end{proposition}

\begin{pf}
For any $f \in P_n$, by the explicit expression of the integration over
two variables, it is not hard to show that $E[f|x_i] \in\tilde{P}_n$.
The same property was pointed out in \cite{CCL}. Therefore,
$E[f|x_i]=\sum_{k=1}^{n} t_k J_k(x_i)+t_0$ with some constants $t_k$.
On the other hand, since $1, J_k(x_i) \in Q_n$ for $1 \le k \le n$ and
by the assumption $f \in Q_n^{\perp}$, we have $t_k=0$ for $0 \le k
\le n$.
\end{pf}

For any $f \in P_n$, we can write $f= \sum_{i=0}^n a_i F_i + \sum
_{i=1}^n b_i G_i + \sum_{i=1}^n c_i H_i + K$ with some $K \in
Q_n^{\perp}$ and constants $a_i$, $b_i$ and $c_i$. In particular, if
$E[f]=0$, then $a_0=0$. Moreover, since $F_1=0$ on $\mathcal
{S}_{{1}/{3},3}$, we take $a_1=0$. Then, for any $f \in P_n$ satisfying
$E[f]=0$,
\begin{eqnarray*}
\hspace*{-4pt}&& \sum_{i=1}^3 E\bigl[(1-x_i)E[f|x_i]^2
\bigr]
\\
\hspace*{-4pt}&&\qquad= \sum_{i=1}^3 E
\Biggl[(1-x_i) \Biggl(E\Biggl[\sum_{k=2}^n
a_k F_k + \sum_{k=1}^n
b_k G_k + \sum_{k=1}^n
c_k H_k |x_i\Biggr] \Biggr) ^2
\Biggr]
\\
\hspace*{-4pt}&&\qquad= E\Biggl[(1-x_1) \Biggl( \sum
_{k=2}^n a_k (1+2\nu_k)
J_k (x_1) + \sum_{k=1}^n
(b_k + c_k) (1-\nu_k) J_k(x_1)
\Biggr) ^2\Biggr]
\\
\hspace*{-4pt}&&\qquad\quad{} + E\Biggl[(1-x_2) \Biggl( \sum
_{k=2}^n a_k (1+2\nu_k)
J_k (x_2) - \sum_{k=1}^n
2 c_k (1-\nu_k) J_k(x_2) \Biggr)
^2\Biggr]
\\
\hspace*{-4pt}&&\qquad\quad{} + E\Biggl[(1-x_3) \Biggl( \sum
_{k=2}^n a_k (1+2\nu_k)
J_k (x_3) + \!\sum_{k=1}^n
(-b_k + c_k) (1-\nu_k) J_k(x_3)
\Biggr) ^2\Biggr]
\\
\hspace*{-4pt}&&\qquad= 3 E\Biggl[(1-x_1) \Biggl( \sum
_{k=2}^n a_k (1+2\nu_k)
J_k (x_1) \Biggr) ^2\Biggr]
\\
\hspace*{-4pt}&&\qquad\quad{}+ 2 E
\Biggl[(1-x_1) \Biggl( \sum_{k=1}^n
b_k (1-\nu_k) J_k(x_1) \Biggr)
^2\Biggr]
\\
\hspace*{-4pt}&&\qquad\quad{} + 6 E\Biggl[(1-x_1) \Biggl( \sum
_{k=1}^n c_k (1-\nu_k)
J_k(x_1) \Biggr) ^2\Biggr].
\end{eqnarray*}
On the other hand,
\begin{eqnarray*}
E\bigl[f^2\bigr] & =& E\Biggl[ \Biggl(\sum
_{k=2}^n a_k F_k + \sum
_{k=1}^n b_k G_k +
\sum_{k=1}^n c_k
H_k +K \Biggr) ^2\Biggr]
\\
& =& 3 \sum_{k=2}^n a_k^2
(1+2\nu_k) E\bigl[ J_k (x_1)^2
\bigr] + 2 \sum_{k=1}^n
b_k^2 (1- \nu_k) E\bigl[ J_k
(x_1)^2\bigr]
\\
&&{} + 6 \sum_{k=1}^n
c_k^2 (1- \nu_k) E\bigl[ J_k
(x_1)^2\bigr] + E \bigl[K^2\bigr].
\end{eqnarray*}

Therefore, to show (\ref{eq:statement1}) we only need to show that
%
%e6.9 #&#
%
\begin{equation}
\label{eq1} \sup_{n \in\N} \sup_{a=(a_k)} \biggl
\{ \frac{ E_{\mu}[(1-u) ( \sum_{k=2}^n a_k (1+2\nu_k) J_k (u) ) ^2]
}{\sum_{k=2}^n a_k^2
(1+2\nu_k) E_{\mu}[ J_k (u)^2]} \biggr\} < 1
\end{equation}
and
%
%e6.10 #&#
%
\begin{equation}
\label{eq2} \sup_{n \in\N} \sup_{b=(b_k)} \biggl
\{ \frac{ E_{\mu}[(1-u) ( \sum_{k=1}^n b_k (1-\nu_k) J_k(u) ) ^2]
}{\sum_{k=1}^n b_k^2 (1- \nu
_k) E_{\mu}[ J_k (u)^2]} \biggr\} < 1,
\end{equation}
where $\mu$ is the beta distribution with parameters $(\ga, 2\ga)$.

Since $\{ J_{n}\}$ is a series of orthogonal polynomials, we have
\begin{eqnarray*}
&& E_{\mu}\Biggl[(1-u) \Biggl( \sum_{k=2}^n
a_k (1+2\nu_k) J_k (u) \Biggr) ^2
\Biggr]
\\
&&\qquad= \sum_{k=2}^n
a_k^2 (1+2\nu_k)^2
E_{\mu}\bigl[(1-u) J_k (u) ^2\bigr]
\\
&&\qquad\quad{} -2 \sum_{k=2}^{n-1}
a_k a_{k+1} (1+2\nu_k) (1+2\nu_{k+1})
E_{\mu
}\bigl[u J_k (u) J_{k+1} (u)\bigr].
\end{eqnarray*}

Define $J_{n,m} \in\R$ for $n \in\N$ and $ 1 \le m \le n$ as
\[
J_{n}(u)=\sum_{m=0}^n
J_{n,m} u^m.
\]

Then we have
\begin{eqnarray*}
E_{\mu}\bigl[ u J_k (u) ^2\bigr] & =&
E_{\mu}\bigl[ u^{k+1} J_{k,k} J_k(u)
\bigr] + E_{\mu
}\bigl[ u^{k} J_{k,k-1}
J_k(u) \bigr]
\\
& =& \frac{J_{k,k}}{J_{k+1,k+1}} E_{\mu}\bigl[ \bigl( J_{k+1}(u) -
J_{k+1,k} u^k \bigr) J_k(u) \bigr] +
\frac{J_{k,k-1}}{J_{k,k}} E_{\mu}\bigl[ J_k(u)^2 \bigr]
\\
& =& - \frac{J_{k+1,k}}{J_{k+1,k+1}} E_{\mu}\bigl[J_k(u)
^2 \bigr] + \frac
{J_{k,k-1}}{J_{k,k}} E_{\mu}\bigl[
J_k(u)^2 \bigr]
\end{eqnarray*}
and
\[
E_{\mu}\bigl[u J_k (u) J_{k+1} (u)\bigr] =
E_{\mu}\bigl[ u^{k+1} J_{k,k} J_{k+1}(u)
\bigr] = \frac{J_{k,k}}{J_{k+1,k+1}} E_{\mu}\bigl[J_{k+1}(u)
^2 \bigr].
\]
Therefore, we have
%
%e6.11 #&#
%e6.12 #&#
%e6.13 #&#
%
\begin{eqnarray*}
&&\quad\quad E_{\mu}\Biggl[(1-u) \Biggl( \sum_{k=2}^n
a_k (1+2\nu_k) J_k (u) \Biggr) ^2
\Biggr]
\\
&&\quad\quad\qquad= \sum_{k=2}^n
a_k^2 (1+2\nu_k)^2 \biggl( 1 +
\frac
{J_{k+1,k}}{J_{k+1,k+1}} - \frac{J_{k,k-1}}{J_{k,k}} \biggr) E_{\mu
}
\bigl[J_k (u) ^2\bigr]
\\
&&\quad\quad\qquad\quad{} -2 \sum_{k=2}^{n-1}
a_k a_{k+1} (1+2\nu_k) (1+2\nu_{k+1})
\frac
{J_{k,k}}{J_{k+1,k+1}} E_{\mu}\bigl[J_{k+1}(u) ^2
\bigr].
\end{eqnarray*}

In the same manner, we have
\begin{eqnarray*}
&& E_{\mu}\Biggl[(1-u) \Biggl( \sum_{k=1}^n
b_k (1-\nu_k) J_k(u) \Biggr) ^2
\Biggr]
\\
&&\qquad= \sum_{k=1}^n
b_k^2 (1-\nu_k)^2 \biggl( 1 +
\frac
{J_{k+1,k}}{J_{k+1,k+1}} - \frac{J_{k,k-1}}{J_{k,k}} \biggr) E_{\mu
}
\bigl[J_k (u) ^2\bigr]
\\
&&\qquad\quad{} -2 \sum_{k=1}^{n-1}
b_k b_{k+1} (1-\nu_k) (1-\nu_{k+1})
\frac
{J_{k,k}}{J_{k+1,k+1}} E_{\mu}\bigl[J_{k+1}(u) ^2
\bigr].
\end{eqnarray*}

Now, we change variables as $\tilde{a}_k=a_k \sqrt{(1 +2\nu_k)
E_{\mu}[ J_k (u)^2]}$ and $\tilde{b}_k=\break b_k \sqrt{(1 -\nu_k)
E_{\mu
}[ J_k (u)^2]}$.
Note that $1 +2\nu_k > 0$ for $k \ge2$ and $1 -\nu_k >0$ for $k \ge
1 $.

Then conditions (\ref{eq1}) and (\ref{eq2}) can be rewritten as
%
%e6.14 #&#
%
\begin{eqnarray}
\label{eq:statment2}
&&\quad\sup_{n \in\N} \sup_{\tilde{a}=(\tilde{a}_k) } \Biggl
\{ \Biggl( \sum_{k=2}^n \tilde{a}_k^2 (1+2\nu_k) p_k
\nonumber
\\[-8pt]
\\[-8pt]
&&\quad\hspace*{55pt}{}-2 \sum_{k=2}^{n-1} \tilde{a}_k\tilde
{a}_{k+1} \sqrt{ (1+2\nu
_k)(1+2\nu_{k+1}) } q_k \Biggr)\bigg/{\sum_{k=2}^n \tilde{a}_k ^2
} \Biggr\} < 1\nonumber
\end{eqnarray}
and
%
%e6.15 #&#
%
\begin{eqnarray}
\label{eq:b}
&&\quad\sup_{n \in\N} \sup_{\tilde{b}=(\tilde{b}_k) }
\Biggl
\{ \Biggl( \sum_{k=1}^n \tilde{b}_k^2 (1-\nu_k) p_k
\nonumber
\\[-8pt]
\\[-8pt]
&&\quad\hspace*{55pt}{}-2 \sum_{k=1}^{n-1} \tilde{b}_k\tilde
{b}_{k+1} \sqrt{ (1-\nu
_k)(1-\nu_{k+1}) } q_k \Biggr)\bigg/{\sum_{k=1}^n \tilde{b}_k^2 }
\Biggr\} < 1,\nonumber
\end{eqnarray}
where $p_k= 1 + \frac{J_{k+1,k}}{J_{k+1,k+1}} - \frac
{J_{k,k-1}}{J_{k,k}} $ and $q_k=\frac{J_{k,k}}{J_{k+1,k+1}} \sqrt
{\frac{E_{\mu}[J_{k+1}(u) ^2 ] }{E_{\mu}[J_{k}(u) ^2 ]}}$. Note that
$|q_k|=-q_k$ for all $k \in\N$.

%By the definition,
%\begin{eqnarray*}
%q_k & = - \frac{\Gamma(2k+3\ga-1)}{k!\Gamma(k+3\ga-1)} \frac{(k+1)!
%\Gamma(k+3\ga)}{\Gamma(2k+3\ga+1)} \sqrt{\frac{\Gamma(k+1+\gamma)(2k+3
%\ga-1)k!\Gamma(k+3\ga-1)}{(2k+3\ga+1)(k+1)!\Gamma(k+3\ga) \Gamma(k+
%\gamma)}} \\
%& = - \frac{(k+1)(k+3\ga-1)}{(2k+3\ga)(2k+3\ga-1)} \sqrt{\frac{(k+
%\gamma)(2k+3\ga-1)}{(2k+3\ga+1)(k+1)(k+3\ga-1)}} \\
%& = - \frac{1}{(2k+3\ga)} \sqrt{\frac{(k+3\ga-1)(k+1) (k+\gamma)}{(2k+3
%\ga+1)(2k+3\ga-1)}} \\
%& = - \kappa_k\frac{1}{\sqrt{(2k+3\ga)(2k+3\ga-1)} } \frac{1}{
%\sqrt{(2k+3\ga+1)(2k+3\ga)}}
% \end{eqnarray*}
%where $\kappa_k= \sqrt{(k+3\ga-1)(k+1) (k+\gamma)}$.

Since for any sequence of positive numbers $\{ \a_k\}_{k \ge2}$
\[
\sum_{k=2}^{n-1} \biggl(
\tilde{a}_k \sqrt{ \frac{(1+2 \nu_k)
|q_k|}{\a_k}} - \tilde{a}_{k+1}
\sqrt{ (1+2 \nu_{k+1})|q_k| \a_k }
\biggr)^2 \ge0
\]
we have
\begin{eqnarray*}
&& - 2 \sum_{k=2}^{n-1}
\tilde{a}_k\tilde{a}_{k+1} \sqrt{ (1 +2\nu_k)
(1+2\nu_{k+1}) } q_k
\\
&&\qquad= 2 \sum_{k=2}^{n-1}
\tilde{a}_k\tilde{a}_{k+1} \sqrt{ (1 +2\nu_k)
(1+2\nu_{k+1}) } |q_k|
\\
&&\qquad\le\sum_{k=2}^{n-1} \biggl(
\tilde{a}_k^2 \frac{(1+2 \nu_k)
|q_k| }{\a_k} + \tilde{a}_{k+1}^2
(1+2\nu_{k+1})|q_k| \a_k \biggr).
\end{eqnarray*}
Namely,
\begin{eqnarray*}
&& \sup_{n \ge2} \sup_{\tilde{a} } \biggl\{
\frac{ \sum_{k=2}^n \tilde{a}_k^2 (1+2\nu_k) p_k
-2 \sum_{k=2}^{n-1} \tilde{a}_k\tilde{a}_{k+1} \sqrt{ (1+2\nu
_k)(1+2\nu_{k+1}) } q_k }{\sum_{k=2}^n \tilde{a}_k^2 } \biggr\}
\\
&&\qquad\le\sup_{n \ge2} \sup_{\tilde{a} } \biggl\{
\frac{ \sum_{k=2}^n \tilde{a}_k^2 (1+2 \nu_k) ( p_k + {|q_k| }/{\a
_k} + |q_{k-1}| \a_{k-1} )
}{\sum_{k=2}^n \tilde{a}_k^2 } \biggr\},
\end{eqnarray*}
where $\a_1=0$ for convention. Therefore, we can conclude (\ref
{eq:statment2}) if we succeed to show the following proposition.

%pr6.2 #&#
%
\begin{proposition}\label{prop:a}
There exists a sequence of positive numbers $\{ \a_n \}_{n=2}^{\infty
}$ which satisfies
\[
\sup_{n \ge2 } \biggl\{ (1+2\nu_n) \biggl(
p_n + \frac{|q_n| }{\a_n} + |q_{n-1}| \a_{n-1}
\biggr) \biggr\} < 1,
\]
where $\a_1=0$ for convention.
\end{proposition}

In the same manner, to show (\ref{eq:b}), we only need to show the
following proposition.

%pr6.3 #&#
%
\begin{proposition}\label{prop:b}
There exists a sequence of positive numbers $\{ \b_n \}_{n=1}^{\infty
}$ which satisfies
\[
\sup_{n \ge1 } \biggl\{ (1-\nu_n) \biggl(
p_n + \frac{|q_n| }{\b_n} + |q_{n-1}| \b_{n-1}
\biggr) \biggr\} < 1,
\]
where $\b_0=0$ for convention.
\end{proposition}

%s6.1 #&#
\subsection{Some properties of constants}

To prove the desired propositions, we first study some properties of
constants $\nu_n$, $p_n$ and $q_n$. Hereafter, to emphasize the fact
that $\nu_n$, $p_n$ and $q_n$ depend not only on $n$ but also $\ga$,
we denote them by $\nu_n(\ga)$, $p_n(\ga)$ and $q_n(\ga)$.

%le6.2 #&#
%
\begin{lemma}\label{lem:propertynu}
For each fixed $\ga>0$, $|\nu_n(\ga)|$ is decreasing as a function
of $n$ for $n \ge1$. Moreover, for each fixed $n \in\N$, $|\nu
_n(\ga)|$ is decreasing as a function of $\gamma$ for $\ga>0$.
\end{lemma}

\begin{pf}
Since $ |\nu_n(\ga)|= \prod_{k=0}^{n-1} \frac{\ga+k}{2\ga+k}$, it
is obvious.
\end{pf}

%le6.3 #&#
%
\begin{lemma}\label{lem:propertyp}
$p_n (\ga) >0$ for any $n \in\N$ and $\ga>0$. Moreover, for each
fixed $\ga< \frac{2}{3}$, $p_n(\ga)$ is increasing as a function of
$n$ for $n \ge1$ and $p_n (\ga) <\frac{1}{2}$ for any $n \in\N$.
If $\ga=\frac{2}{3}$, then $p_n(\ga)= \frac{1}{2}$ for all $n \in
\N$. For each fixed $\ga> \frac{2}{3}$, $p_n(\ga)$ is decreasing as
a function of $n$ for $n \ge1$.
\end{lemma}

\begin{pf}
By the definition,
\begin{eqnarray*}
p_n(\ga)&=&1+\frac{-(n+1)(n+\gamma)}{2n+3\gamma}-\frac{-n(n+\gamma
-1)}{2n+3\gamma-2}
\\
%& =\frac{(2n+3\gamma)(2n+3\gamma-2)-(n+1)(n+\gamma)(2n+3\gamma-2)+n(n+
%\gamma-1)(2n+3\ga) }{(2n+3\gamma)(2n+3\gamma-2)} \\
&=& \frac{2n(n-1)+(6n-4)\ga+6\ga^2 }{(2n+3\gamma)(2n+3\gamma
-2)}=\frac{1}{2}+
\frac{-\ga+3/2\ga^2}{(2n+3\gamma)(2n+3\gamma-2)}.
\end{eqnarray*}
\upqed
\end{pf}

%le6.4 #&#
%
\begin{lemma}\label{lem:propertyp2}
For each fixed $n \in\N$, $p_n(\ga)$ is increasing as a function of
$\gamma$ for $\ga\ge\frac{1}{3}$.
\end{lemma}

\begin{pf}
By the definition,
\[
\frac{d}{d\ga} p_n(\ga) = \frac{d}{d\ga} \biggl(
\frac{6\ga
^2+(6n-4)\ga+2n(n-1)}{(2n+3\ga)(2n+3\ga-2)} \biggr)
\]
and the numerator of the derivative is
%& = \frac{(12\ga+6n-4)(2n+3\ga)(2n+3\ga-2)-3(6\ga^2+(6n-4)
%\ga+2n(n-1))(4n+6\ga-2) }{(2n+3\ga)^2(2n+3\ga-2)^2}
%\end{eqnarray*}
%Now, by the calculations,
%
\begin{eqnarray*}
&&(12\ga+6n-4) (2n+3\ga) (2n+3\ga-2)
\\
&&\quad{}-3\bigl(6\ga^2+(6n-4)\ga+2n(n-1)\bigr) (4n+6\ga-2)
\\
&&\qquad= 2n \bigl( 9\ga^2+6\ga(n-1)-2(n-1) \bigr) > 0
\end{eqnarray*}
for $\gamma\ge\frac{1}{3}$.
\end{pf}

%le6.5 #&#
%
\begin{lemma}\label{lem:propertyq1}
For each fixed $\frac{2}{3} \le\ga\le2$, $ |q_n (\ga)|$ is
decreasing as a function of $n$ for $n \ge2$. For each fixed $2 < \ga
\le\frac{7}{3}$, $|q_n(\ga)|$ is decreasing as a function of $n$ for
$n \ge3$. For each fixed $\ga< \frac{2}{3}$, $|q_n (\ga)|$ is
decreasing as a function of $n$ for $n \ge2$.
\end{lemma}

\begin{pf}
First, note that by the definition
\begin{eqnarray*}
q_n(\ga) & = &- \frac{\Gamma(2n+3\ga-1)}{n!\Gamma(n+3\ga-1)}
\frac
{(n+1)!\Gamma(n+3\ga)}{\Gamma(2n+3\ga+1)}
\\
&&\ \ \ {} \times\sqrt{\frac{\Gamma(n+1+\gamma)(2n+3\ga-1)n!\Gamma
(n+3\ga
-1)}{(2n+3\ga+1)(n+1)!\Gamma(n+3\ga) \Gamma(n+\gamma)}}
\\
& =& - \frac{(n+1)(n+3\ga-1)}{(2n+3\ga)(2n+3\ga-1)} \sqrt{\frac
{(n+\gamma)(2n+3\ga-1)}{(2n+3\ga+1)(n+1)(n+3\ga-1)}}
\\
& =& - \frac{1}{(2n+3\ga)} \sqrt{\frac{(n+3\ga-1)(n+1) (n+\gamma
)}{(2n+3\ga+1)(2n+3\ga-1)}}. %\\
%& = - \kappa_k\frac{1}{\sqrt{(2k+3\ga)(2k+3\ga-1)} } \frac{1}{
%\sqrt{(2k+3\ga+1)(2k+3\ga)}}
\end{eqnarray*}

For each $n \in\N$, we have
\[
\frac{|q_{n+1}(\ga)|^2}{|q_{n}(\ga)|^2} =\frac{(2n+3\ga)^2(2n+3\ga
-1)}{(2n+3\ga+2)^2(2n+3\ga+3)}\frac{(n+2)(n+3\ga)(n+\ga
+1)}{(n+1)(n+3\ga-1)(n+\ga)}.
\]
For any $\ga>0$,
$(2n+3\ga+3)(n+\ga)- (2n+3\ga)(n+\ga+1) = n$, so
\[
\frac{(2n+3\ga)}{(2n+3\ga+3)}\frac{(n+\ga+1)}{(n+\ga)} < 1.
\]
On the other hand, for $2/3 \le\ga$, since $(2n+3\ga+2)(n+3\ga-1)-
(2n+3\ga)(n+3\ga) = 3\ga-2$,
\[
\frac{(2n+3\ga)}{(2n+3\ga+2)}\frac{(n+3\ga)}{(n+3\ga-1)} \le1.
\]
In the same manner, $(2n+3\ga+2)(n+1)- (2n+3\ga-1)(n+2) = n-3\ga+4$,
then if $\ga\le\frac{n+4}{3}$, then
\[
\frac{(2n+3\ga-1)}{(2n+3\ga+2)}\frac{(n+2)}{(n+1)} \le1.
\]
Next, we assume that $\ga< 2/3$. As in the same way,
since $(2n+3\ga+2)(n+3\ga-1)- (2n+3\ga-1)(n+3\ga) = n+6\ga-2$, for
any $n \ge2$,
\[
\frac{(2n+3\ga-1)}{(2n+3\ga+2)}\frac{(n+3\ga)}{(n+3\ga-1)} < 1
\]
and $(2n+3\ga+2)(n+1)- (2n+3\ga)(n+2) = 2-3\ga$,
\[
\frac{(2n+3\ga)}{(2n+3\ga+2)}\frac{(n+2)}{(n+1)} < 1.
\]
\upqed
\end{pf}

%le6.6 #&#
%
\begin{lemma}\label{lem:propertyq2}
For each fixed $n \ge3$, $|q_n(\ga)|$ is decreasing as a function of
$\gamma$ for $\ga>0$ and $|q_2 (\ga) |$ is decreasing as a function
of $\gamma$ for $\ga\ge\frac{1}{10}$.
\end{lemma}

\begin{pf}
Instead of $|q_n(\ga)|$ itself, we will consider the derivative of
$|q_n(\ga)|^2$.
By the definition,
\[
\frac{d}{d\ga} \bigl|q_n(\ga)\bigr|^2= \sqrt{n+1}
\frac{d}{d\ga} \biggl(\frac{(n+\ga)(n+3\ga-1)}{(2n+3\ga)^4 -
(2n+3\ga)^2} \biggr).
\]
The numerator of the derivative is
\begin{eqnarray*}
&& (4n+6\ga-1) \bigl((2n+3\ga)^4 - (2n+3\ga)^2\bigr)
\\
&&\quad{} - (n+\ga) (n+3\ga-1) \bigl( 12(2n+3\ga)^3-6(2n+3\ga)
\bigr)
\\
&&\qquad= - (2n+3\ga) \bigl( \bigl(18\ga^2+(24n-9)\ga+
\bigl(4n^2-10n\bigr)\bigr) (2n+3\ga)^2
\\
&&\hspace*{223pt}{}+ 3
\ga+2n^2+4n \bigr)
\end{eqnarray*}
which is negative for any $\ga>0$ if $n \ge3$, and at least for $\ga
\ge\frac{1}{10}$ if $n =2$.
\end{pf}

%le6.7 #&#
%
\begin{lemma}\label{lem:sqrtestimate}
For any $\ga\ge\frac{1}{5}$ and $n \in\N$, $|q_n(\ga)| \le\frac
{1}{4\sqrt{n+\ga}}$.
\end{lemma}

\begin{pf}
For any positive numbers $a,b$, $\sqrt{ab} \le\frac{1}{2} (a+b)$. Therefore,
\[
\bigl|q_n(\ga)\bigr|= \frac{\sqrt{n+\ga}}{(2n+3\ga)}\frac{\sqrt
{(n+1)(n+3\ga-1)}}{\sqrt{(2n+3\ga)^2-1}} \le
\frac{\sqrt{n+\ga
}}{2\sqrt{(2n+3\ga)^2-1}}.
\]
Now, for any $\ga\ge\frac{1}{5}$,
\begin{eqnarray*}
(2n+3\ga)^2-1&=&4n^2+12n\ga+9\ga^2-1 \ge4(n+
\ga)^2 + (5\ga-1) (\ga+1)
\\
&\ge& 4(n+\ga)^2.
\end{eqnarray*}
\upqed
\end{pf}

%s6.2 #&#
\subsection{Proof of Proposition \texorpdfstring{\protect\ref{prop:a}}{A.2} for \texorpdfstring{$\gamma\ge\frac{2}{3}$}{$gamma>=\frac{2}{3}$}}

Here, we give a proof of Proposition~\ref{prop:a} for the case $\ga
\ge\frac{2}{3}$.

%le6.8 #&#
%
\begin{lemma}\label{lem:4case1}
For any fixed $\ga\ge2$,
$\frac{1}{4\sqrt{4+\ga}}+\frac{1}{4\sqrt{3+\ga}} < \frac
{1}{1+2\nu_4(\ga)}- p_4(\ga)$.
%and for $2/3 \le\ga\le7/3$, $|q_3|+|q_2| < \frac{1}{1-\nu_3}- p_3$.
\end{lemma}

\begin{pf}
By the definition, for any $\ga\ge2$,
\begin{eqnarray*}
&&\frac{1}{1+2\nu_4(\ga)}- p_4(\ga)
\\
&&\qquad = \frac{2(2\ga+1)(2\ga
+3)}{2(2\ga+1)(2\ga+3)+(\ga+2)(\ga+3)}-
\frac{6\ga^2+20\ga
+24}{(3\ga+8)(3\ga+6)}
\\
&&\qquad = \frac{2 \ga(46 + 67 \ga+ 29 \ga^2 + 3 \ga^3)}{3 (\ga
+2) (3 \ga
+8) (3\ga+4) (\ga+1)}
\\
&&\qquad = \frac{2}{9}+\frac{2 (-64 - 30 \ga+ 43 \ga^2 + 24 \ga
^3)}{9 (\ga
+2) (3 \ga+8) (3\ga+4) (\ga+1)} >\frac{2}{9}.
\end{eqnarray*}
On the other hand, for any $\ga\ge2$,\vspace*{-1pt}
\[
\frac{1}{4\sqrt{4+\ga}}+\frac{1}{4\sqrt{3+\ga}} < \frac
{1}{4\sqrt{6}}+
\frac{1}{4\sqrt{5}} <\frac{2}{9},
\]
and the lemma follows.\vspace*{-1pt}
\end{pf}

%le6.9 #&#
%
\begin{lemma}\label{lem:6}
For\vspace*{-1pt} any fixed $\frac{2}{3} \le\ga\le2$,
$|q_6 (\ga)|+|q_5(\ga)| < \frac{1}{1+2\nu_6 (\ga)}- p_6(\ga)$.
%and for $2/3 \le\ga\le7/3$, $|q_3|+|q_2| < \frac{1}{1-\nu_3}- p_3$.
\end{lemma}

\begin{pf}
By Lemmas~\ref{lem:propertynu}, \ref{lem:propertyp2} and~\ref
{lem:propertyq2},\vspace*{-1pt}
\[
\sup_{{2}/{3} \le\ga\le2}\bigl(\bigl|q_6 (\ga)\bigr|+\bigl|q_5(
\ga)\bigr|\bigr)=|q_6|\biggl(\frac{2}{3}\biggr)+|q_5|
\biggl(\frac{2}{3}\biggr)
\]
and\vspace*{-1pt}
\[
\inf_{{2}/{3} \le\ga\le2}\biggl(\frac{1}{1+2\nu_6 (\ga)}-
p_6(\ga)
\biggr) \ge\frac{1}{1+2\nu_6({2}/{3})}-p_6(2).
\]
Then the exact calculation shows\vspace*{-1pt}
\[
\bigg|q_6\biggl(\frac{2}{3}\biggr)\bigg|+\bigg|q_5\biggl(
\frac{2}{3}\biggr)\bigg| < \frac{1}{1+2\nu_6(
{2}/{3})}-p_6(2).
\]
\upqed
\end{pf}

%le6.10 #&#
%
\begin{lemma}\label{lem:smalla}
For $\frac{2}{3} \le\ga\le2$, let $\a_2(\ga)=|q_2(\ga)|
(\frac{1}{1+2\nu_2(\ga)}-p_2(\ga) )^{-1}$,
{\spaceskip=0.155em plus 0.05em minus 0.02em $\a_3(\ga)=\frac
{{1}/{(1+2\nu_4(\ga))}-p_4(\ga) }{2|q_3(\ga)|}$, $\a_4(\ga)=2
|q_4(\ga)| (\frac{1}{1+2\nu_4(\ga)}-p_4(\ga) )^{-1}$ and
\mbox{$\a_5(\ga)=1$}.} Then\vspace*{-1pt}
%
%e6.16 #&#
%e6.17 #&#
%
\begin{eqnarray}
\label{eq:simple1} \bigl(1+2\nu_2(\ga)\bigr) \biggl(p_2(\ga)+
\frac{|q_2(\ga)|}{\a_2(\ga)} \biggr)&=& 1,
\\
\label{eq:simple2} \bigl(1+2\nu_4(\ga)\bigr) \biggl(p_4(
\ga)+\frac{|q_4(\ga)|}{\a_4(\ga)}+\bigl|q_3(\ga)\bigr|\a_3(\ga
)\biggr)&=& 1
\end{eqnarray}
and\vspace*{-1pt}
%
%e6.18 #&#
%
\begin{equation}
\label{eq:simple3} \quad\max_{n=3,5} \biggl\{\bigl(1+2\nu_n(
\ga)\bigr) \biggl(p_n(\ga)+\frac{|q_n(\ga)|}{\a
_n(\ga)}+\bigl|q_{n-1}(\ga)\bigr|
\a_{n-1}(\ga)\biggr) \biggr\} < 1
\end{equation}
hold.\vspace*{-1pt}
\end{lemma}

\begin{pf}
Equations (\ref{eq:simple1}) and (\ref{eq:simple2}) hold by the
choice of $\{\a_n(\ga)\}_{n=2}^4$. Therefore, we only need to show
(\ref{eq:simple3}).
For $n=3$,
\begin{eqnarray*}
&& \bigl(1+2\nu_3(\ga)\bigr) \biggl(p_3(\ga)+
\frac{|q_3(\ga)|}{\a_3(\ga)}+\bigl|q_2(\ga)\bigr|\a_2(\ga
)\biggr)
\\
&&\qquad= \frac{3\ga}{4\ga+2}\biggl(p_3(\ga)+2\bigl|q_3(
\ga)\bigr|^2 \biggl(\frac
{1}{1+2\nu_4(\ga)}-p_4(\ga)
\biggr)^{-1}
\\
&&\hspace*{106pt}{} +\bigl|q_2(\ga)\bigr|^2 \biggl(\frac
{1}{1+2\nu_2(\ga)}-p_2(
\ga) \biggr)^{-1}\biggr),
\\
&&\frac{3\ga}{4\ga+2}p_3(\ga) =\frac{3\ga}{4\ga+2}
\frac{6\ga
^2+14\ga+12}{(3\ga+6)(3\ga+4)}=\frac{\ga(3\ga^2+7\ga+6)}{6 \ga
^3 + 23 \ga^2 +26 \ga+8} < \frac{1}{2},
\\
&&\frac{6\ga}{4\ga+2}\bigl|q_3(\ga)\bigr|^2 \biggl(
\frac{1}{1+2\nu_4(\ga
)}-p_4(\ga) \biggr)^{-1}
\\
&&\qquad\le\frac{3\ga}{2\ga+1} \frac{3 (\ga+2) (3 \ga+8) (3\ga+4)
(\ga+1)}{2 \ga(46 + 67 \ga+ 29 \ga^2 + 3 \ga^3)}\biggl|q_3\biggl(
\frac
{2}{3}\biggr)\biggr|^2
\\
&&\qquad= \biggl( \frac{27}{4} + \frac{351}{124 (1 + 2 \ga)} -
\frac{9
(1052 + 1388 \ga+ 471 \ga^2)}{ 62 (46 + 67 \ga+ 29 \ga^2 + 3 \ga
^3)} \biggr) \frac{11}{756}
\\
&&\qquad< \biggl(\frac{27}{4}+\frac{351}{124}\biggr)
\frac{11}{756} = \frac{121}{868},
\\
&&\frac{3\ga}{4\ga+2}\bigl|q_2(\ga)\bigr|^2 \biggl(
\frac{1}{1+2\nu_2(\ga
)}-p_2(\ga) \biggr)^{-1}
\\
&&\qquad\le\frac{3\ga}{4\ga+2} \frac{(3\ga+4)(3\ga+2)}{3\ga}
\biggl|q_2\biggl(
\frac{2}{3}\biggr)\biggr|^2
\\
&&\qquad= \biggl(\frac{27}{8} +
\frac{9 \ga}{4} + \frac
{5}{8 (1 + 2 \ga)}\biggr)\frac{ 2 }{105} <
\frac{17}{105}
\end{eqnarray*}
for the last inequality we use the fact that $\ga\le2$.
%\begin{eqnarray*}
%&= \frac{3\ga}{4\ga+2}(\frac{6\ga^2+14\ga+12}{(3\ga+6)(3\ga+4)}+2|q_2(
%\ga)|^2 (\frac{2}{9}+\frac{48}{23\ga})+|q_2(\ga)|^2\frac{(3\ga+4)(3
%\ga+2)}{3\ga})
%\end{eqnarray*}
For $n=5$,\vspace*{-1pt}
\begin{eqnarray*}
&& \bigl(1+2\nu_5(\ga)\bigr) \biggl(p_5(\ga)+
\frac{|q_5(\ga)|}{\a_5(\ga)}+\bigl|q_4(\ga)\bigr|\a_4(\ga
)\biggr)
\\[-1pt]
&&\qquad= \bigl(1+2\nu_5(\ga)\bigr) \biggl(p_5(
\ga)+\bigl|q_5(\ga)\bigr|+2\bigl|q_4(\ga)\bigr|^2 \biggl(
\frac
{1}{1+2\nu_4(\ga)}-p_4(\ga) \biggr)^{-1}\biggr).
\end{eqnarray*}
Here,\vspace*{-1pt}
\begin{eqnarray*}
&& \bigl(1+2\nu_5(\ga)\bigr)p_5(\ga) \le\biggl(1+2
\nu_5\biggl(\frac{2}{3}\biggr)\biggr)p_5(2) =
\frac{15}{26} \frac{29}{56} < \frac{1}{3},
\\[-1pt]
&& \bigl(1+2\nu_5(\ga)\bigr)\bigl|q_5(\ga)\bigr|
\le\bigl|q_5(\ga)\bigr| \le\biggl|q_5\biggl(\frac
{2}{3}
\biggr)\biggr|=\sqrt{\frac{17}{1716}} <\frac{1}{10},
\\[-1pt]
&& \bigl(1+2\nu_5(\ga)\bigr)2\bigl|q_4(\ga)\bigr|^2
\biggl(\frac{1}{1+2\nu_4(\ga
)}-p_4(\ga) \biggr)^{-1}
\\
&&\qquad\le\frac{25}{18} 2 \bigl(1+2\nu_3(\ga)
\bigr)\bigl|q_3(\ga)\bigr|^2 \biggl(\frac
{1}{1+2\nu_4(\ga)}-p_4(
\ga) \biggr)^{-1} \le\frac{25}{18}\frac
{121}{868} <
\frac{2}{9}.\qquad
\end{eqnarray*}
\upqed
\end{pf}

%le6.11 #&#
%
\begin{lemma}\label{lem:biga}
For $ \ga\ge2$, let $\a_2(\ga)=|q_2(\ga)| (\frac{1}{1+2\nu
_2(\ga)}-p_2(\ga) )^{-1}$, and $\a_3(\ga)=1$. Then
%
%e6.19 #&#
%
\begin{equation}
\label{eq:simple4} \bigl(1+2\nu_2(\ga)\bigr) \biggl(p_2(\ga)+
\frac{|q_2(\ga)|}{\a_2(\ga)} \biggr)= 1
\end{equation}
and
%
%e6.20 #&#
%
\begin{equation}
\label{eq:simple5} \biggl\{\bigl(1+2\nu_3(\ga)\bigr)
\biggl(p_3(\ga)+\frac{|q_3(\ga)|}{\a_3(\ga
)}+\bigl|q_2(\ga)\bigr|
\a_2(\ga)\biggr) \biggr\} < 1
\end{equation}
hold.
\end{lemma}

\begin{pf}
Equation (\ref{eq:simple4}) holds by the choice of $\a_2(\ga)$.
Therefore, we only need to show (\ref{eq:simple5}). Note that
\begin{eqnarray*}
&& \bigl(1+2\nu_3(\ga)\bigr) \biggl(p_3(\ga)+
\frac{|q_3(\ga)|}{\a_3(\ga)}+\bigl|q_2(\ga)\bigr|\a_2(\ga
)\biggr)
\\
&&\qquad= \frac{3\ga}{4\ga+2}\biggl(p_3(\ga)+\bigl|q_3(\ga
)\bigr|
+\bigl|q_2(\ga)\bigr|^2 \biggl(\frac{1}{1+2\nu_2(\ga)}-p_2(
\ga) \biggr)^{-1}\biggr).
\end{eqnarray*}
Then, since
\begin{eqnarray*}
\hspace*{-3pt}&&\frac{3\ga}{4\ga+2}p_3(\ga) =\frac{3\ga}{4\ga+2}
\frac{6\ga
^2+14\ga+12}{(3\ga+6)(3\ga+4)}=\frac{\ga(3\ga^2+7\ga+6)}{6 \ga
^3 + 23 \ga^2 +26 \ga+8} < \frac{1}{2},
\\
\hspace*{-3pt}&&\frac{3\ga}{4\ga+2}\bigl|q_3(\ga)\bigr| \le\frac{3}{4}
\bigl|q_3(2)\bigr| = \frac
{1}{3} \sqrt{\frac{10}{143}} <
\frac{1}{9},
\\
\hspace*{-3pt}&& \frac{3\ga}{4\ga+2}\bigl|q_2(\ga)\bigr|^2 \biggl(
\frac{1}{1+2\nu_2(\ga
)}-p_2(\ga) \biggr)^{-1}
\\
\hspace*{-3pt}&&\qquad\le\frac{3\ga}{4\ga+2} \frac{(3\ga+4)(3\ga+2)}{3\ga}
\bigl |q_2(\ga
)\bigr |^2 = \frac{(3\ga+1)(3\ga+2)(3\ga+6)}{(4\ga+2)(3\ga+3)(3\ga
+4)(3\ga+5)}
\\
\hspace*{-3pt}&&\qquad< \frac{1}{4\ga+2} \le\frac{1}{10},
\end{eqnarray*}
the lemma follows.
%for the last inequality we use the fact that $\ga\le2$.
\end{pf}

\begin{pf*}{Proof of Proposition~\ref{prop:a} for $\ga\ge\frac{2}{3}$}
First, assume that $\ga\ge2$. Take $\a_2(\ga)= \frac{|q_2(\ga
)|}{{1}/{(1+2\nu_2(\ga))}-p_2(\ga)}+ \epsilon(\ga) $ where
$\epsilon(\ga) >0$ will be specified later, and $\a_n(\ga)=1$ for
$n \ge3$. By Lemmas~\ref{lem:propertynu}, \ref{lem:propertyp}, \ref
{lem:sqrtestimate} and~\ref{lem:4case1},
\begin{eqnarray*}
&&\sup_{n \ge4} \biggl\{ \bigl(1+2\nu_n(\ga)\bigr)
\biggl(p_n(\ga)+\frac{|q_n(\ga
)|}{\a_n(\ga)}+\bigl|q_{n-1}(\ga)\bigr|
\a_{n-1}(\ga)\biggr) \biggr\}
\\
&&\qquad= \sup_{n \ge4} \bigl\{ \bigl(1+2\nu_n(\ga)
\bigr) \bigl(p_n(\ga)+\bigl|q_n(\ga)\bigr|+\bigl|q_{n-1}(\ga
)\bigr|
\bigr) \bigr\}
\\
&&\qquad\le\sup_{n \ge4} \biggl\{ \bigl(1+2\bigl|\nu_n(
\ga)\bigr|\bigr) \biggl(p_n+\frac{1}{4\sqrt{n+\ga}}+\frac
{1}{4\sqrt{n-1+\ga}}\biggr)
\biggr\}
\\
&&\qquad= \bigl(1+2\nu_4(\ga)\bigr) \biggl(p_4(\ga)+
\frac{1}{4\sqrt{4+\ga}}+\frac
{1}{4\sqrt{3+\ga}}\biggr) <1
\end{eqnarray*}
holds. On the other hand, by Lemma~\ref{lem:biga}, for sufficiently
small $\epsilon(\ga)>0$
\[
\bigl(1+2\nu_2(\ga)\bigr) \biggl(p_2(\ga)+
\frac{|q_2(\ga)|}{\a_2(\ga)} \biggr) <1
\]
and
\[
\bigl(1+2\nu_3(\ga)\bigr) \biggl(p_3(\ga)+
\frac{|q_3(\ga)|}{\a_3(\ga)}+ \bigl|q_2(\ga)\bigr|\a_2(\ga
)\biggr) <1
\]
hold. Therefore, the proof is complete. The same argument works for
the case $\frac{2}{3} \le\ga\le2$ with Lemmas~\ref{lem:smalla}
and~\ref{lem:4case1}.
\end{pf*}

%%%%%%%%%%%%%%%%%%%%%%%%%%%%%%%%%%%%%%%%%%%%%%%%%%%%%%%%%%%%%%%%%%%%%%%%%%%%%%%%%%%%%%%%%%%%%%%%%%%%%%%%%%%%%%%%%%%%%%%%%%%%%%%%%%%%%%%%%%%%%%

%s6.3 #&#
\subsection{Proof of Proposition \texorpdfstring{\protect\ref{prop:b}}{A.3} for \texorpdfstring{$\gamma\ge\frac{2}{3}$}{$gamma>=\frac{2}{3}$}}

Here, we give a proof of Proposition~\ref{prop:b} for the case $\ga
\ge\frac{2}{3}$.

%le6.12 #&#
%
\begin{lemma}\label{lem:4case2}
For any fixed $\ga\ge2$,
$\frac{1}{4\sqrt{3+\ga}}+\frac{1}{4\sqrt{2+\ga}} < \frac
{1}{1-\nu_3(\ga)}- p_3(\ga)$.
%and for $2/3 \le\ga\le7/3$, $|q_3|+|q_2| < \frac{1}{1-\nu_3}- p_3$.
\end{lemma}

\begin{pf}
By the definition,
\begin{eqnarray*}
\frac{1}{1-\nu_3(\ga)}- p_3(\ga) &=&\frac{8+40\ga+38\ga^2+6\ga
^3}{48+132\ga+108\ga^2+27\ga^3}
\\
& =&\frac{2}{9}+\frac{2(-4+16\ga+21\ga^2)}{9(16+44\ga+36\ga
^2+9\ga^3)}.
\end{eqnarray*}
Therefore, for any $\ga\ge3$,
\[
\frac{1}{4\sqrt{3+\ga}}+\frac{1}{4\sqrt{2+\ga}} <\frac{2}{9} <
\frac{1}{1-\nu_3(\ga)}- p_3(\ga).
\]
On the other hand,
\[
\frac{1}{1-\nu_3(\ga)}- p_3(\ga) =\frac{1}{4}+
\frac{-16+28\ga
+44\ga^2-3\ga^3}{12(16+44\ga+36\ga^2+9\ga^3)}
\]
and since $-16+28\ga+44\ga^2-3\ga^3 > 0$ for any $2 \le\ga\le3$,
we have
\[
\frac{1}{4\sqrt{3+\ga}}+\frac{1}{4\sqrt{2+\ga}} <\frac{1}{4} <
\frac{1}{1-\nu_3(\ga)}- p_3(\ga)
\]
for $2 \le\ga\le3$.
\end{pf}

%le6.13 #&#
%
\begin{lemma}\label{lem:10}
For any fixed $\frac{2}{3} \le\ga\le2$,
$|q_3 (\ga)|+|q_2(\ga)| < \frac{1}{1-\nu_3 (\ga)}- p_3(\ga)$.
%and for $2/3 \le\ga\le7/3$, $|q_3|+|q_2| < \frac{1}{1-\nu_3}- p_3$.
\end{lemma}

\begin{pf}
By Lemmas~\ref{lem:propertynu}, \ref{lem:propertyp2} and~\ref{lem:propertyq2},
\[
\sup_{{2}/{3} \le\ga\le2}\bigl(\bigl|q_3 (\ga)\bigr|+\bigl
|q_2(\ga
)\bigr|\bigr)=|q_3|\biggl(\frac{2}{3}\biggr)+|q_2|
\biggl(\frac{2}{3}\biggr) =\frac{1}{6} \sqrt{\frac
{11}{21}}+
\frac{1}{3} \sqrt{\frac{6}{35}} < \frac{13}{50}
\]
and
\[
\frac{1}{1-\nu_3 (\ga)}- p_3(\ga)= \frac{13}{50} +
\frac
{-224+284\ga+496\ga^2-51\ga^3}{150(16+44\ga+36\ga^2+9\ga^3)}.
\]
Then, since $-224+284\ga+496\ga^2-51\ga^3 >0$ for $\frac{2}{3} \le
\ga\le2$,
\[
\big|q_3(\ga)\big|+\big|q_2(\ga)\big| < \frac{1}{1-\nu_3(\ga)}-p_3(
\ga).
\]
\upqed
\end{pf}

%le6.14 #&#
%
\begin{lemma}\label{lem:bb}
Let $0 < \epsilon(\ga) < \frac{2}{1+3\ga}$ and $\b_1(\ga)=\frac
{|q_1(\ga)|3(3\ga+2)}{2-\epsilon(\ga) }$ and $\beta_2(\ga
)=1$. Then
\[
\bigl(1-\nu_1(\ga)\bigr) \biggl( p_1(\ga) +
\frac{|q_1(\ga)|}{ \beta_1(\ga
)} \biggr) <1
\]
and
%
%e6.21 #&#
%
\begin{equation}
\label{eq:estimate2} \bigl(1-\nu_2(\ga)\bigr) \biggl( p_2(\ga)
+ \frac{|q_2(\ga)|}{ \beta_2(\ga
)} + \bigl|q_1(\ga)\bigr| \beta_1(\ga) \biggr)
<1
\end{equation}
hold.
\end{lemma}

\begin{pf}
By the definition,
\begin{eqnarray*}
\bigl(1-\nu_1(\ga)\bigr) \biggl( p_1(\ga) +
\frac{|q_1(\ga)|}{ \beta_1(\ga
)} \biggr) &= &\frac{3}{2} \biggl( \frac{2(3\ga+1)}{3(3\ga+2)}+
\frac
{2-\epsilon(\ga)}{3(3\ga+2) } \biggr)
\\
& =& \frac{(3\ga+1)}{(3\ga+2)} + \frac{2-\epsilon(\ga)}{2(2+3\ga
)} < 1.
\end{eqnarray*}
On the other hand,
\begin{eqnarray*}
1-\nu_2(\ga) & =&\frac{3\ga+1}{2(2\ga+1)},
\\
p_2(\ga) & =& \frac{2(3\ga^2+4\ga+2)}{(4+3\ga)(2+3\ga)}= \frac
{2\ga}{2+3\ga} +
\frac{4}{(4+3\ga)(2+3\ga)}, %\end{eqnarray*}
%
%\begin{eqnarray*}
%& (1-\nu_2) \Big( p_2 + \frac{|q_2|}{\beta_2} + |q_1| \beta_{1} \Big) \
%\
%& = \frac{3\ga+1}{2(2\ga+1)} \Big(\frac{2(3\ga^2+4\ga+2)}{(4+3\ga)(2+3
%\ga)} + \frac{\sqrt{3(3\ga+1)(2+\ga)}}{\sqrt{(4+3\ga)(2+3\ga)}(3+3
%\ga)} + \frac{|q_1|^2 3(3\ga+2)}{2-\epsilon} \Big)
% \end{eqnarray*}
\\
\bigl|q_1(\ga)\bigr| \beta_1 (\ga) & =&\frac{|q_1(\ga)|^2 3(3\ga
+2)}{2-\epsilon(\ga)} =
\frac{18\ga(\ga+1)(3\ga+2)}{(3+3\ga
)(2+3\ga)^2 (1+3\ga) (2-\epsilon(\ga)) }
\\
& =& \frac{6\ga}{(2+3\ga) (1+3\ga) (2-\epsilon(\ga))} < \frac
{1}{2+3\ga}
\end{eqnarray*}
and
\[
\bigl|q_2 (\ga) \bigr| = \frac{\sqrt{3(\ga+2)(3\ga+1)}}{\sqrt
{(3+3\ga
)(4+3\ga)^2 (5+3\ga) } } < \frac{\sqrt{3} (3+4\ga)}{2 (4+3\ga
)(2+3\ga)}.
\]
Therefore,
\begin{eqnarray*}
&& \bigl(1-\nu_2(\ga)\bigr) \biggl( p_2(\ga) +
\frac{|q_2 (\ga) |}{ \beta
_2(\ga) } + \bigl|q_{1}(\ga)\bigr| \beta_{1}(\ga) \biggr)
\\
%& < \frac{3\ga+1}{2(2\ga+1)} \Big(\frac{2(3\ga^2+4\ga+2)}{(4+3\ga)(2+3
%\ga)} + \frac{\sqrt{3(3\ga+1)(2+\ga)}}{(4+3\ga)(3+3\ga)} +
%\frac{1}{(2+3\ga)} \Big) \\
&&\qquad< \frac{3\ga+1}{2(2\ga+1)} \biggl(
\frac{2\ga+1}{2+3\ga} + \frac
{4}{(4+3\ga)(2+3\ga)} + \frac{\sqrt{3} (3+4\ga)}{2 (4+3\ga
)(2+3\ga)} \biggr)
\\
&&\qquad= \frac{3\ga+1}{2(2+3\ga)} + \frac{3\ga+1}{2(2\ga+1)}
\frac{8+
\sqrt{3} (3+4\ga)}{2 (4+3\ga)(2+3\ga)}.
\end{eqnarray*}
Now, to show (\ref{eq:estimate2}), we only need to show that
\[
\frac{(8+ \sqrt{3} (3+4\ga) )(3\ga+1)}{4 (2\ga+1) (4+3\ga)(2+3\ga
)} \le1- \frac{3\ga+1}{2(2+3\ga)}=\frac{3\ga+3}{2(2+3\ga)}
\]
which is equivalent to
\[
(8+ 3 \sqrt{3} +4\sqrt{3} \ga) (3\ga+1) \le2 (3\ga+3) (2\ga+1)
(4+3\ga).
\]
Then, by comparing coefficients of both sides, we complete the proof.
\end{pf}

\begin{pf*}{Proof of Proposition~\ref{prop:b} for $\ga\ge\frac{2}{3}$}
First, assume that $\ga\ge2$. Take $\beta_1(\ga)$ as in Lemma~\ref
{lem:bb} and $\b_n(\ga)=1$ for $n \ge2$. By Lemmas~\ref
{lem:propertynu}, \ref{lem:propertyp}, \ref{lem:sqrtestimate}
and~\ref{lem:4case2},
\begin{eqnarray*}
&&\sup_{n \ge3} \biggl\{ \bigl(1-\nu_n(\ga)\bigr)
\biggl(p_n(\ga)+\frac{|q_n(\ga)|}{\b
_n(\ga)}+\bigl|q_{n-1}(\ga)\bigr|
\b_{n-1}(\ga)\biggr) \biggr\}
\\
&&\qquad= \sup_{n \ge3} \bigl\{ \bigl(1- \nu_n(\ga)
\bigr) \bigl(p_n(\ga)+\bigl|q_n(\ga)\bigr|+\bigl|q_{n-1}(\ga
)\bigr|
\bigr) \bigr\}
\\
&&\qquad\le\sup_{n \ge3} \biggl\{ \bigl(1+\bigl|\nu_n(
\ga)\bigr|\bigr) \biggl(p_n(\ga)+\frac{1}{4\sqrt{n+\ga}}+\frac
{1}{4\sqrt{n-1+\ga}}
\biggr) \biggr\}
\\
&&\qquad= \bigl(1-\nu_3(\ga)\bigr) \biggl(p_3(\ga)+
\frac{1}{4\sqrt{3+\ga}}+\frac
{1}{4\sqrt{2+\ga}}\biggr) <1
\end{eqnarray*}
holds. Therefore, with Lemma~\ref{lem:bb}, we have
\[
\sup_{n \ge1} \biggl\{ \bigl(1-\nu_n(\ga)\bigr)
\biggl(p_n(\ga)+\frac{|q_n(\ga)|}{\b
_n(\ga)}+\bigl|q_{n-1}(\ga)\bigr|
\b_{n-1}(\ga)\biggr) \biggr\} <1.
\]
The same argument works for the case $\frac{2}{3} \le\ga\le2$ with
Lemmas~\ref{lem:propertyq1}, \ref{lem:10} and~\ref{lem:bb}.
\end{pf*}

%%%%%%%%%%%%%%%%%%%%%%%%%%%%%%%%%%%%%%%%%%%%%%%%%%%%%%%%%%%%%%%%%%%%%%%%%%%%%%%%%%%%%%%%%%%%%%%%%%%%%%%%%%%%%%%%%%%%%%%%%%%%%%%%%%%%%%%%%%%%%%

%s6.4 #&#
\subsection{Proof of Propositions \texorpdfstring{\protect\ref{prop:a}}{A.2} and \texorpdfstring{\protect\ref{prop:b}}{A.3} for
\texorpdfstring{$\gamma< \frac{2}{3}$}{$gamma< \frac{2}{3}$}}

Here, we give a proof of Propositions~\ref{prop:a} and~\ref{prop:b}
for the case $\ga< \frac{2}{3}$.

For any $\ga>0 $, $(\frac{1}{1+2|\nu_n(\ga)|}-\frac{1}{2})$ is
positive for any $n \ge2$ and increasing for $n \ge2$. Therefore, by
Lemma~\ref{lem:propertyq1}, $|q_n(\ga)|(\frac{1}{1+2|\nu_n(\ga
)|}-\frac{1}{2})^{-1}$ is decreasing for $n \ge2$ and
\[
\lim_{n \to\infty} \bigl|q_n(\ga)\bigr| \biggl(\frac{1}{1+2
|\nu_n(\ga
)|}-
\frac{1}{2} \biggr)^{-1} =0
\]
holds. Therefore, there exists $n_0=n_0(\ga) \in\N$ satisfying for
any $n \ge n_0$
\[
\bigl|q_n(\ga)\bigr| \biggl(\frac{1}{1+2 |\nu_n(\ga)|}-\frac{1}{2}
\biggr)^{-1} < \frac{1}{2}.
\]

Then it is obvious that
\[
\sup_{n \ge n_0+1} \biggl\{ \bigl(1+2 \nu_n(\ga)\bigr)
\biggl( \frac{1}{2} + \bigl|q_n(\ga)\bigr| + \bigl|q_{n-1}(\ga
)\bigr|
\biggr) \biggr\} < 1
\]
and
\[
\sup_{n \ge n_0+1} \biggl\{ (1-\nu_n) \biggl(
\frac{1}{2} + \bigl|q_n(\ga)\bigr| + \bigl|q_{n-1}(\ga)\bigr|
\biggr)
\biggr\} < 1.
\]

%In this case, $p_n < \frac{1}{2}$ for any $n \in\N$, so we only need
%to show that
%\begin{eqnarray*}
%\sup_{n \ge2} \{ (1+2 \nu_n) \Big( \frac{1}{2} + \frac{|q_n|}{\a_n} +
%|q_{n-1}| \a_{n-1} \Big) \} < 1
%\end{eqnarray*}
%with $\a_1 =0$ and
%\begin{eqnarray*}
%\sup_{n \in\N} \{ (1-\nu_n) \Big( \frac{1}{2} + \frac{|q_n|}{\b_n} +
%|q_{n-1}| \b_{n-1} \Big) \} < 1
%\end{eqnarray*}
%with $\b_0=0$.

%Fix an even number $n_1=n_1(\ga)$ satisfying $n_1 \ge n_0$.

Define $\a_n(\ga)$ as follows:
%
%e6.22 #&#
%
\begin{equation}
\label{defan} \a_n(\ga)= %
\cases{ \max\biggl\{
\bigl|q_2(\ga)\bigr|\biggl(\displaystyle\frac{1}{1+2\nu_2(\ga
)}-\frac{1}{2}
\biggr)^{-1}, 1 \biggr\}\vspace*{2pt}
\cr
\qquad\mbox{if } n = 2,\vspace*{2pt}
\cr
\max\biggl\{
2\bigl|q_n(\ga)\bigr|\biggl(\displaystyle\frac{1}{1+2\nu_n(\ga
)}-\frac{1}{2}
\biggr)^{-1}, 1 \biggr\}\vspace*{2pt}
\cr
\qquad\mbox{if } n \ge4 \mbox{ and } n \mbox{
is even},\vspace*{2pt}
\cr
\biggl(\max\biggl\{ 2\bigl|q_n(\ga)\bigr|\biggl(
\displaystyle\frac{1}{1+2\nu_{n+1}(\ga)}-\frac
{1}{2}\biggr)^{-1}, 1 \biggr\}
\biggr)^{-1}\vspace*{2pt}
\cr
\qquad\mbox{if } n \ge3 \mbox{ and } n \mbox{ is odd} }
\end{equation}
and $\a_1=0$. Obviously, $\a_n (\ga) =1$ for any $n \ge n_0= n_0(\ga)$.

%le6.15 #&#
%
\begin{lemma}\label{lem:mainforsmall}
For any $n \ge2$,
\[
\bigl(1+2 \nu_n(\ga)\bigr) \biggl( \frac{1}{2} +
\frac{|q_n(\ga)|}{\a_n(\ga
)} + \bigl|q_{n-1}(\ga)\bigr| \a_{n-1}(\ga) \biggr) \le
1.
\]
\end{lemma}

\begin{pf}
In this proof, we denote $\nu_n(\ga)$, $p_n(\ga)$ and $q_n(\ga)$ by
$\nu_n$, $p_n$ and $q_n$ when there is no confusion.

By the definition,
\[
(1+2 \nu_2) \biggl( \frac{1}{2} + \frac{|q_2|}{\a_2} \biggr)
\le(1+2 \nu_2) \biggl( \frac{1}{2} + {\biggl( \frac{1}{1+2\nu
_2}-\frac {1}{2}\biggr) |q_2|}/{|q_2|} \biggr) =1
\]
and for any even number $n \ge4 $,
\begin{eqnarray*}
&&(1+2 \nu_n) \biggl( \frac{1}{2} + \frac{|q_n|}{\a_n} +
|q_{n-1}| \a_{n-1} \biggr)
\\
&&\qquad\le(1+2 \nu_n) \biggl( \frac{1}{2} + {\biggl( \frac
{1}{1+2\nu _n}-\frac{1}{2}\biggr) |q_n|}/{\bigl(2|q_n|\bigr)}
\\
&&\hspace*{82pt}{}+ \biggl(\frac{1}{1+2\nu_n}-\frac
{1}{2}
\biggr) |q_{n-1}|/ {\bigl(2|q_{n-1}|\bigr)} \biggr) =1.
\end{eqnarray*}
Next, for $n=3$,
\begin{eqnarray*}
&& (1+2 \nu_3) \biggl( \frac{1}{2} + \frac{|q_3|}{\a_3} +
|q_2|\a_2 \biggr)
\\
&&\qquad\le(1+2 \nu_3) \biggl( \frac{1}{2} + \max\biggl\{
2|q_3|^2\biggl(\frac
{1}{1+2\nu_{4}}-\frac{1}{2}
\biggr)^{-1}, \frac{1}{2}\biggl(\frac{1}{1+2\nu
_{4}}-
\frac{1}{2}\biggr) \biggr\}
\\
&&\hspace*{102pt}{} + \max\biggl\{ |q_2|^2\biggl(\frac{1}{1+2\nu_{2}}-
\frac{1}{2}\biggr)^{-1}, \biggl(\frac{1}{1+2\nu_{2}}-
\frac{1}{2}\biggr) \biggr\} \biggr)
\end{eqnarray*}
and for any odd number $n \ge5$,
\begin{eqnarray*}
\hspace*{-4pt}&& (1+2 \nu_n) \biggl( \frac{1}{2} + \frac{|q_n|}{\a
_n} +
|q_{n-1}|\a_{n-1} \biggr)
\\
\hspace*{-4pt}&&\qquad\le(1+2 \nu_n) \biggl( \frac{1}{2} + \max
\biggl\{
2|q_n|^2\biggl(\frac
{1}{1+2\nu_{n+1}}-\frac{1}{2}
\biggr)^{-1}, \frac{1}{2}\biggl(\frac{1}{1+2\nu
_{n+1}}-
\frac{1}{2}\biggr) \biggr\}
\\
\hspace*{-4pt}&&\hspace*{81pt}{} + \max\biggl\{ 2|q_{n-1}|^2\biggl(
\frac{1}{1+2\nu_{n-1}}-\frac
{1}{2}\biggr)^{-1}, \frac{1}{2}
\biggl(\frac{1}{1+2\nu_{n-1}}-\frac{1}{2}\biggr) \biggr\} \biggr).
\end{eqnarray*}
To conclude the proof, we will show that
\begin{eqnarray*}
\max\biggl\{ 2|q_3|^2\biggl(\frac{1}{1+2\nu_{4}}-
\frac{1}{2}\biggr)^{-1}, \frac{1}{2}\biggl(
\frac{1}{1+2\nu_{4}}-\frac{1}{2}\biggr) \biggr\} &\le& \frac
{1}{2}
\biggl(\frac{1}{1+2\nu_{3}}-\frac{1}{2}\biggr),
\\
\max\biggl\{ |q_2|^2\biggl(\frac{1}{1+2\nu_{2}}-
\frac{1}{2}\biggr)^{-1}, \biggl(\frac{1}{1+2\nu_{2}}-
\frac{1}{2}\biggr) \biggr\} &\le& \frac{1}{2}\biggl(
\frac
{1}{1+2\nu_{3}}-\frac{1}{2}\biggr)
\end{eqnarray*}
and for any odd number $n \ge5$,
\begin{eqnarray*}
\max\biggl\{ 2|q_n|^2\biggl(\frac{1}{1+2\nu_{n+1}}-
\frac{1}{2}\biggr)^{-1}, \frac{1}{2}\biggl(
\frac{1}{1+2\nu_{n+1}}-\frac{1}{2}\biggr) \biggr\} &\le& \frac
{1}{2}
\biggl(\frac{1}{1+2\nu_{n}}-\frac{1}{2}\biggr),
\\
\max\biggl\{ 2|q_{n-1}|^2\biggl(\frac{1}{1+2\nu_{n-1}}-
\frac
{1}{2}\biggr)^{-1}, \frac{1}{2} \biggl(
\frac{1}{1+2\nu_{n-1}}-\frac{1}{2}\biggr) \biggr\} &\le& \frac{1}{2}
\biggl(\frac{1}{1+2\nu_{n}}-\frac{1}{2}\biggr).
\end{eqnarray*}

Note that for any odd number $n \ge3$, $ \frac{1}{1+2\nu_{n}}-\frac
{1}{2} > \frac{1}{2}$ and even number $n \ge4$, $ \frac{1}{1+2\nu
_{n}}-\frac{1}{2} < \frac{1}{2}$ and $ \frac{1}{1+2\nu_{2}}-\frac
{1}{2} =\frac{3\ga}{6\ga+4} < \frac{1}{4}$. Namely, we only need to
show that
\begin{eqnarray*}
4|q_3|^2\biggl(\frac{1}{1+2\nu_{4}}-\frac{1}{2}
\biggr)^{-1} &\le&\biggl(\frac
{1}{1+2\nu_{3}}-\frac{1}{2}\biggr),
\\
2
|q_2|^2\biggl(\frac{1}{1+2\nu_{2}}-\frac
{1}{2}
\biggr)^{-1} &\le&\biggl(\frac{1}{1+2\nu_{3}}-\frac{1}{2}\biggr),
\\
4|q_n|^2\biggl(\frac{1}{1+2\nu_{n+1}}-\frac{1}{2}
\biggr)^{-1} &\le&\biggl(\frac
{1}{1+2\nu_{n}}-\frac{1}{2}\biggr),
\\
4|q_{n-1}|^2\biggl(\frac{1}{1+2\nu_{n-1}}-\frac{1}{2}
\biggr)^{-1} &\le&\biggl(\frac
{1}{1+2\nu_{n}}-\frac{1}{2}\biggr).
\end{eqnarray*}
We can rewrite these inequalities as
\begin{eqnarray*}
16|q_3|^2 \frac{1+2|\nu_{4}|}{1-2|\nu_{4}|} &\le&\frac{1+2|\nu
_{3}|}{1-2|\nu_{3}|},
\qquad8 |q_2|^2\frac{1+2|\nu_{2}|}{1-2|\nu_{2}|} \le
\frac{1+2|\nu_{3}|}{1-2|\nu_{3}|},
\\
16 |q_n|^2\frac{1+2|\nu_{n+1}|}{1-2|\nu_{n+1}|} &\le&\frac{1+2|\nu
_{n}|}{1-2|\nu_{n}|},
\qquad16|q_{n-1}|^2\frac{1+2|\nu_{n-1}|}{1-2|\nu
_{n-1}|} \le
\frac{1+2|\nu_{n}|}{1-2|\nu_{n}|}.
\end{eqnarray*}

Combing the fact that $|q_n|^2$ is decreasing in $n \ge2$ and $|\nu
_n|$ is also decreasing in~$n$, we only need to prove that
\[
16|q_3|^2 \le1,\qquad8 |q_2|^2
\frac{1+2|\nu_{2}|}{1-2|\nu_{2}|} \le\frac{1+2|\nu_{3}|}{1-2|\nu_{3}|}
\]
and for any odd number $n \ge5$,
\[
16|q_4|^2\frac{1+2|\nu_{n-1}|}{1-2|\nu_{n-1}|} \le\frac{1+2|\nu
_{n}|}{1-2|\nu_{n}|}.
\]
Since $|q_3(\ga)|^2 < |q_3(0)|^2=\frac{2}{105}$, the first inequality
holds for all $\ga>0$. The second inequality is rewritten as
\[
\bigl|q_2(\ga)\bigr|^2 =\frac{(3\ga+1)(3\ga+6)}{(3\ga+3)(3\ga
+4)^2(3\ga
+5)} \le\frac{5\ga+4}{24(3\ga+2)}
\]
and since the coefficients of the polynomial
\[
(5\ga+4) (3\ga+3) (3\ga+4)^2(3\ga+5)-24(3\ga+2) (3\ga+1) (3\ga+6)
\]
are all positive, it is satisfied for any $\ga> 0$.

Finally, by Lemma~\ref{lem:n3} below, to show the last inequality we
only need to show that
\[
16 \bigl|q_4(\ga)\bigr |^2 \le\frac{1+2|\nu_{3}(\ga) |}{1-2|\nu
_{3}(\ga)
|}
\frac{1-2|\nu_{2}(\ga) |}{1+ 2|\nu_{2}(\ga) |}
\]
and it follows from the fact
\begin{eqnarray*}
16 \bigl|q_4(\ga) \bigr|^2 &\le&16\bigl|q_4(0)\bigr|^2=
\frac{5}{21} < \frac{1}{3} < \frac{5\ga+4}{3(3\ga+2)}
\\
&=&
\frac{1+2|\nu_{3}(\ga) |}{1-2|\nu
_{3}(\ga) |}\frac{1-2|\nu_{2}(\ga) |}{1+ 2|\nu_{2}(\ga) |}.
\end{eqnarray*}
\upqed
\end{pf}

%le6.16 #&#
%
\begin{lemma}\label{lem:n3}
For any $n \ge2$,
\[
\frac{1+2|\nu_{n+1}(\ga) |}{1-2|\nu_{n+1}(\ga) |}\frac{1-2|\nu
_{n}(\ga) |}{1+ 2|\nu_{n}(\ga) |} \ge\frac{1+2|\nu_{3}(\ga)
|}{1-2|\nu_{3}(\ga) |}
\frac{1-2|\nu_{2}(\ga) |}{1+ 2|\nu_{2}(\ga
) |}.
\]
\end{lemma}

\begin{pf}
Consider a function $f(x,a)=\frac{(1+ax)(1-a)}{(1-ax)(1+a)}$ for
$0<a<1$ and $0<\allowbreak x<1$. Then it is easy to see that $\partial
_x f (x,a) >
0 $ and $\partial_a f (x,a) < 0 $ for all $0<x<1$ and $0<a<1$.
Therefore, for any $n \ge2$,
\begin{eqnarray*}
\frac{1+2|\nu_{n+1}(\ga) |}{1-2|\nu_{n+1}(\ga) |}\frac{1-2|\nu
_{n}(\ga) |}{1+ 2|\nu_{n}(\ga) |} &=& f\biggl(\frac{n+\ga}{n+2\ga
}, 2\bigl|\nu
_{n}(\ga)\bigr| \biggr)
\\
& \ge& f\biggl(\frac{2+\ga}{2+2\ga},2\bigl|\nu_n(\ga)\bigr
|\biggr) \ge f
\biggl(\frac{2+\ga
}{2+2\ga},2\bigl|\nu_2(\ga)\bigr|\biggr)
\\
& =& \frac{1+2|\nu_{3}(\ga) |}{1-2|\nu
_{3}(\ga) |}
\frac{1-2|\nu_{2}(\ga) |}{1+ 2|\nu_{2}(\ga) |}.
\end{eqnarray*}
\upqed
\end{pf}

\begin{pf*}{Proof of Proposition~\ref{prop:a} for $\ga< \frac{2}{3}$}
Define $\{\a_n(\ga)\}$ as (\ref{defan}). Then, for sufficiently
large $n \in\N$, $\a_n(\ga)=1$. Therefore,
\begin{eqnarray*}
&&\limsup_{n \to\infty} \biggl( \bigl(1+2 \nu_n (\ga)
\bigr) \biggl( p_n + \frac{|q_n (\ga)|}{\a_n(\ga)} + \bigl
|q_{n-1}(\ga)\bigr|
\a_{n-1}(\ga) \biggr) \biggr)
\\
&&\qquad\le\limsup_{n \to\infty} \biggl( \bigl(1+2
\nu_n(\ga)\bigr) \biggl( \frac
{1}{2} + \bigl|q_n(\ga)\bigr|
+ \bigl|q_{n-1}(\ga)\bigr| \biggr) \biggr) =\frac{1}{2} < 1
\end{eqnarray*}
holds. Then, to show Proposition~\ref{prop:a}, we only need to show
that with this $\a_n(\ga)$,
\[
\bigl(1+2 \nu_n(\ga)\bigr) \biggl( p_n(\ga) +
\frac{|q_n(\ga)|}{\a_n(\ga)} + \bigl|q_{n-1}(\ga)\bigr| \a
_{n-1}(\ga) \biggr) <1
\]
for all $n \ge2$. Then, since for any $\ga< \frac{2}{3}$ and $n \ge
2$, $p_n(\ga) < \frac{1}{2}$, the proof is complete with Lemma~\ref
{lem:mainforsmall}.
\end{pf*}

%%%%%%%%%%%%%%%%%%%%%%%%%%%%%%%%%%%%%%%%%%%%%%%%%%%%%%%%%%%%%%%%%%%%%%%%%%%%%%%%%%%%%%%%%%%%%%%%%%%%%%%%%%%%%%%%%%%%%%%%%%%%%%%%%%%%%%%%%%%%%%

Define $\b_n$ as follows:
%
%e6.23 #&#
%
\begin{equation}
\label{defbn} \b_n= %
\cases{ \max\biggl\{
\bigl|q_1(\ga)\bigr|\biggl(\displaystyle\frac{1}{1-\nu_1(\ga
)}-\frac{1}{2}
\biggr)^{-1}, 1 \biggr\}\vspace*{2pt}
\cr
\qquad\mbox{if } n = 1,\vspace*{2pt}
\cr
\max\biggl\{
2\bigl|q_n(\ga)\bigr|\biggl(\displaystyle\frac{1}{1-\nu_n(\ga
)}-\frac{1}{2}
\biggr)^{-1}, 1 \biggr\}\vspace*{2pt}
\cr
\qquad\mbox{if } n \ge3 \mbox{ and } n \mbox{
is odd},\vspace*{2pt}
\cr
\biggl(\max\biggl\{ 2\bigl|q_n(\ga)\bigr|\biggl(
\displaystyle\frac{1}{1-\nu_{n+1}(\ga)}-\frac
{1}{2}\biggr)^{-1}, 1 \biggr\}
\biggr)^{-1}\vspace*{2pt}
\cr
\qquad\mbox{if } n \ge2 \mbox{ and } n \mbox{ is even}
} %
\end{equation}
and $\b_0=0$. Obviously, $\b_n=1$ for any $n \ge n_0$.

%le6.17 #&#
%
\begin{lemma}\label{lem:mainforsmallb}
For any $n \ge1$
\[
\bigl(1- \nu_n(\ga)\bigr) \biggl( \frac{1}{2} +
\frac{|q_n(\ga)|}{\b_n(\ga
)} + \bigl|q_{n-1}(\ga)\bigr| \b_{n-1}(\ga) \biggr) \le
1.
\]
\end{lemma}

\begin{pf}
By the definition,
\[
\bigl(1- \nu_1(\ga)\bigr) \biggl( \frac{1}{2} +
\frac{|q_1(\ga)|}{\b_1(\ga
)} \biggr) \le1
\]
and for any odd number $n \ge3$,
\begin{eqnarray*}
&& \bigl(1- \nu_n(\ga)\bigr) \biggl( \frac{1}{2} +
\frac{|q_n (\ga)|}{\b
_n(\ga)} + \bigl|q_{n-1}(\ga)\bigr| \b_{n-1}(\ga) \biggr)
\\
&&\qquad\le\bigl(1- \nu_n(\ga) \bigr) \biggl( \frac{1}{2} +
{\biggl(\frac{1}{1-\nu _n(\ga)}-\frac{1}{2}\biggr) \bigl|q_n (
\ga)\bigr|}/{\bigl(2\bigl|q_n (\ga)\bigr|\bigr)}
\\
&&\hspace*{90pt}{}+ \biggl(\frac
{1}{1-\nu_n(\ga)}-
\frac{1}{2}\biggr) \bigl|q_{n-1}(\ga)\bigr|/ {\bigl(2\bigl
|q_{n-1}(\ga)\bigr|
\bigr)} \biggr) =1.
\end{eqnarray*}
Next, for $n=2$,
\begin{eqnarray*}
&& \bigl(1- \nu_2(\ga)\bigr) \biggl( \frac{1}{2} +
\frac{|q_2(\ga)|}{\b_2(\ga
)} + \bigl|q_1(\ga)\bigr|\b_1(\ga) \biggr)
\\
&&\qquad\le\bigl(1- \nu_2(\ga)\bigr) \biggl( \frac{1}{2} +
\max\biggl\{ 2\bigl|q_2(\ga)\bigr|^2\biggl(\frac{1}{1-\nu
_{3}(\ga)}-
\frac{1}{2}\biggr)^{-1},
\\
&&\hspace*{186pt}\frac{1}{2}\biggl(
\frac
{1}{1-\nu_{3}(\ga)}-\frac{1}{2}\biggr) \biggr\}
\\
&&\hspace*{90pt}{} + \max\biggl\{ \bigl|q_1(\ga)\bigr|^2\biggl(
\frac{1}{1-\nu_{1}(\ga)}-\frac
{1}{2}\biggr)^{-1}, \biggl(
\frac{1}{1-\nu_{1}(\ga)}-\frac{1}{2}\biggr) \biggr\} \biggr)
\end{eqnarray*}
and for any even number $n \ge4$,
\begin{eqnarray*}
&& \bigl(1- \nu_n(\ga)\bigr) \biggl( \frac{1}{2} +
\frac{|q_n(\ga)|}{\b_n(\ga
)} + \bigl|q_{n-1}(\ga)\bigr|\b_{n-1}(\ga) \biggr)
\\
&&\qquad\le\bigl(1- \nu_n(\ga)\bigr) \biggl( \frac{1}{2} +
\max\biggl\{ 2\bigl|q_n(\ga)\bigr|^2\biggl(\frac{1}{1-\nu
_{n+1}(\ga)}-
\frac{1}{2}\biggr)^{-1},
\\
&&\hspace*{185pt}\frac
{1}{2}\biggl(
\frac{1}{1-\nu_{n+1}(\ga)}-\frac{1}{2}\biggr) \biggr\}
\\
&&\hspace*{92pt}{} + \max\biggl\{ 2\bigl|q_{n-1}(\ga)\bigr
|^2\biggl(
\frac{1}{1-\nu_{n-1}(\ga)}-\frac
{1}{2}\biggr)^{-1},
\\
&&\hspace*{183pt}\frac{1}{2}
\biggl(\frac{1}{1+2\nu_{n-1}(\ga)}-\frac
{1}{2}\biggr) \biggr\} \biggr).
\end{eqnarray*}
To complete the proof, we will show that
\begin{eqnarray*}
&&\max\biggl\{ 2\bigl|q_2(\ga)\bigr|^2\biggl(\frac{1}{1-\nu
_{3}(\ga)}-
\frac
{1}{2}\biggr)^{-1} , \frac{1}{2}\biggl(
\frac{1}{1-\nu_{3}(\ga)}-\frac{1}{2}\biggr) \biggr\}
\\
&&\qquad\le\frac{1}{2}
\biggl(\frac{1}{1-\nu_{2}(\ga)}-\frac{1}{2}\biggr),
\\
&&\max\biggl\{ \bigl|q_1(\ga)\bigr|^2\biggl(\frac{1}{1-\nu
_{1}(\ga)}-
\frac
{1}{2}\biggr)^{-1}, \biggl(\frac{1}{1-\nu_{1}(\ga)}-
\frac{1}{2}\biggr) \biggr\} \le \frac{1}{2}\biggl(
\frac{1}{1-\nu_{2}(\ga)}-\frac{1}{2}\biggr)
\end{eqnarray*}
and for any even number $n \ge4$,
\begin{eqnarray*}
&& \max\biggl\{ 2\bigl|q_n(\ga)\bigr|^2\biggl(
\frac{1}{1-\nu_{n+1}(\ga)}-\frac
{1}{2}\biggr)^{-1}, \frac{1}{2}
\biggl(\frac{1}{1-\nu_{n+1}(\ga)}-\frac{1}{2}\biggr) \biggr\}
\\
&&\qquad\le\frac{1}{2}\biggl(\frac{1}{1-\nu_{n}(\ga)}-\frac{1}{2}
\biggr),
\\
&& \max\biggl\{ 2\bigl|q_{n-1}(\ga)\bigr|^2\biggl(
\frac{1}{1-\nu_{n-1}(\ga)}-\frac
{1}{2}\biggr)^{-1}, \frac{1}{2}
\biggl(\frac{1}{1+2\nu_{n-1}(\ga)}-\frac
{1}{2}\biggr) \biggr\}
\\
&&\qquad\le\frac{1}{2}\biggl(\frac{1}{1-\nu_{n}(\ga)}-\frac{1}{2}
\biggr).
\end{eqnarray*}

Note that for any even number $n \ge2$, $ \frac{1}{1-\nu_{n}(\ga
)}-\frac{1}{2} > \frac{1}{2}$ and odd number $n \ge3$, $ \frac
{1}{1-\nu_{n}(\ga)}-\frac{1}{2} < \frac{1}{2}$ and $ \frac
{1}{1-\nu_{1}(\ga)}-\frac{1}{2} =\frac{1}{6} < \frac{1}{4}$.
Namely, we only need to show that
\begin{eqnarray*}
4\bigl|q_2(\ga)\bigr|^2\biggl(\frac{1}{1-\nu_{3}(\ga)}-
\frac{1}{2}\biggr)^{-1} &\le&\biggl(\frac
{1}{1-\nu_{2}(\ga)}-
\frac{1}{2}\biggr),
\\
2 \bigl|q_1(\ga)\bigr|^2\biggl(\frac{1}{1-\nu_{1}(\ga)}-
\frac{1}{2}\biggr)^{-1} &\le& \biggl(\frac{1}{1-\nu_{2}(\ga)}-
\frac{1}{2}\biggr),
\\
4\bigl|q_n(\ga)\bigr|^2\biggl(\frac{1}{1-\nu_{n+1}(\ga)}-
\frac{1}{2}\biggr)^{-1} &\le& \biggl(\frac{1}{1-\nu_{n}(\ga)}-
\frac{1}{2}\biggr),
\\
4\bigl|q_{n-1}(\ga)\bigr|^2\biggl(\frac{1}{1-\nu_{n-1}(\ga)}-
\frac{1}{2}\biggr)^{-1} &\le&\biggl(\frac{1}{1-\nu_{n}(\ga)}-
\frac{1}{2}\biggr)
\end{eqnarray*}
for any even number $n \ge4$.
We can rewrite these inequalities as
\begin{eqnarray*}
16\bigl|q_2(\ga)\bigr|^2 \frac{1+|\nu_{3}(\ga)|}{1-|\nu_{3}(\ga
)|} &\le&
\frac{1+|\nu_{2}(\ga)|}{1-|\nu_{2}(\ga)|},
\\
8 \bigl|q_1(\ga)\bigr|^2
\frac
{1+|\nu_{1}(\ga)|}{1-|\nu_{1}(\ga)|} &\le&\frac{1+|\nu_{2}(\ga
)|}{1-|\nu_{2}(\ga)|},
\\
16 \bigl|q_n(\ga)\bigr|^2\frac{1+|\nu_{n+1}(\ga)|}{1-|\nu
_{n+1}(\ga)|} &\le&
\frac{1+|\nu_{n}(\ga)|}{1-|\nu_{n}(\ga)|},
\\
16\bigl|q_{n-1}(\ga)\bigr|^2
\frac{1+|\nu_{n-1}(\ga)|}{1-|\nu_{n-1}(\ga)|} &\le&\frac
{1+|\nu_{n}(\ga)|}{1-|\nu_{n}(\ga)|}.
\end{eqnarray*}

Combing the fact that $|q_n(\ga)|^2$ is decreasing in $n \ge2$ and
$|\nu_n(\ga)|$ is also decreasing in $n$, we only need to prove that
\[
16\bigl|q_2(\ga)\bigr|^2 \le1,\qquad8 \bigl|q_1(
\ga)\bigr|^2\frac{1+|\nu_{1}(\ga)|}{1-|\nu
_{1}(\ga)|} \le\frac{1+|\nu_{2}(\ga)|}{1-|\nu_{2}(\ga)|}
\]
and for any even number $n \ge4$,
\[
16\bigl|q_3(\ga)\bigr|^2\frac{1+|\nu_{n-1}(\ga)|}{1-|\nu
_{n-1}(\ga)|} \le
\frac{1+|\nu_{n}(\ga)|}{1-|\nu_{n}(\ga)|}.
\]
Since $|q_2(\ga)|^2 = \frac{(3\ga+1)(2+\ga)}{(3\ga+5)(3\ga
+4)^2(\ga+1)} < \frac{(3\ga+1)(2+\ga)}{16(3\ga+5)(\ga+1)}=\frac
{3\ga^2+7\ga+2}{16(3\ga^2+8\ga+5)}$, the first inequality holds.
The second inequality is rewritten as
\[
\bigl|q_1(\ga)\bigr|^2 =\frac{6\ga(\ga+1)}{(3\ga+1)(3\ga
+2)^2(3\ga+3)} = \frac{2\ga}{(3\ga+1)(3\ga+2)^2}
< \frac{5\ga+3}{24(3\ga+1)}
\]
and since the coefficients of the polynomial\vspace*{1pt}
\[
(5\ga+3) (3\ga+2)^2-48\ga
\]
are all positive, it is satisfied for any $\ga> 0$.

Finally, by the same argument of the proof of Lemma~\ref{lem:n3}, to
prove the last inequality we only need to show that\vspace*{1pt}
\[
16 \bigl|q_3(\ga) \bigr|^2 \le\frac{1+|\nu_{2}(\ga)|}{1-|\nu
_{2}(\ga
)|}
\frac{1-|\nu_{1}(\ga)|}{1+ |\nu_{1}(\ga)|}
\]
and it follows from the fact
\begin{eqnarray*}
16 \bigl|q_3(\ga) \bigr|^2 &\le&16\bigl|q_3(0)\bigr|^2=
\frac{32}{105} < \frac{1}{3} < \frac{5\ga+3}{3(3\ga+1)}
\\
&=&
\frac{1+|\nu_{2}(\ga)|}{1-|\nu
_{2}(\ga)|}\frac{1-|\nu_{1}(\ga)|}{1+ |\nu_{1}(\ga)|}.
\end{eqnarray*}
\upqed
\end{pf}

\begin{pf*}{Proof of Proposition~\ref{prop:b} for $\ga< \frac{2}{3}$}
Define $\{\beta_n(\ga)\}$ as (\ref{defbn}). Then, for sufficiently
large $n \in\N$, $\beta_n(\ga)=1$. Therefore,
\begin{eqnarray*}
&&\limsup_{n \to\infty} \biggl( \bigl(1- \nu_n(\ga)\bigr)
\biggl( p_n (\ga) + \frac{|q_n(\ga)|}{\beta_n(\ga)} + \bigl
|q_{n-1}(\ga)\bigr|
\b_{n-1}(\ga) \biggr) \biggr)
\\
&&\qquad\le\limsup_{n \to\infty} \biggl( \bigl(1-
\nu_n(\ga)\bigr) \biggl( \frac
{1}{2} + \bigl|q_n(\ga)\bigr|
+ \bigl|q_{n-1}(\ga)\bigr| \biggr) \biggr) =\frac{1}{2} < 1
\end{eqnarray*}
holds. Then, to show Proposition~\ref{prop:b}, we only need to show
that with this $\beta_n(\ga)$,
\[
\bigl(1- \nu_n(\ga)\bigr) \biggl( p_n(\ga) +
\frac{|q_n(\ga)|}{\b_n(\ga)} + \bigl|q_{n-1}(\ga)\bigr| \b
_{n-1}(\ga) \biggr) <1
\]
for all $n \ge1$. Then, since for any $\ga< \frac{2}{3}$ and $n \ge
1$, $p_n(\ga) < \frac{1}{2}$, the proof is complete with Lemma~\ref
{lem:mainforsmallb}.
\end{pf*}
\end{appendix}

\section*{Acknowledgments}

The author expresses her sincere thanks to Professor Herbert Spohn for
posing the question which motivated this work. She also thanks
Professor Takeshi Katsura for helpful discussions.

%\begin{supplement}[id=suppA]
%\slink[doi]{10.1214/00-AOPXXXXSUPP} %[doi,text={...}] - jei reikia
%suskaldyti doi
%\sfilename{aopXXXX\_supp.pdf}
%\end{supplement}

%\bibitem[\protect\citeauthoryear{}{}]{r1}
%\bibitem{r1}
% imsref loaded by audrone.aklyte, 2014-04-08 16:23:57
%

\printaddresses

\end{document}